\documentclass[11pt,a4paper]{amsart}
\usepackage[utf8]{inputenc}

\usepackage[mono=false]{libertine}
\usepackage{libertinust1math}
\usepackage[T1]{fontenc}
\usepackage[cal=euler,scr=boondoxo,bb=pazo]{mathalfa}
\usepackage[scaled=0.88]{beramono}

\usepackage{amsmath}
\usepackage{amsfonts}
\usepackage{amssymb}
\usepackage{amsthm}
\usepackage{amscd}
\usepackage{bbm}
\usepackage{xcolor-material}
\usepackage[unicode, psdextra, colorlinks=true, linkcolor=GoogleBlue, citecolor=GoogleGreen, urlcolor=GoogleBlue, linktocpage]{hyperref}
\usepackage{tikz-cd}
\usepackage{mathrsfs}
\usepackage{multirow}
\usepackage{stmaryrd}
\usepackage{mathtools}
\usepackage{etoolbox}
\usepackage{tensor}
\usepackage[shortlabels]{enumitem}
\usepackage[capitalize]{cleveref}

\newlist{thmlist}{enumerate}{1}
\setlist[thmlist]{label=(\alph{thmlisti}), ref=\thethm.(\alph{thmlisti}),noitemsep}
\newlist{proplist}{enumerate}{1}
\setlist[proplist]{label=(\alph{proplisti}), ref=\thethm.(\alph{proplisti}),noitemsep}
\newlist{lmlist}{enumerate}{1}
\setlist[lmlist]{label=(\alph{lmlisti}), ref=\thethm.(\alph{lmlisti}),noitemsep}
\newlist{corrlist}{enumerate}{1}
\setlist[corrlist]{label=(\alph{corrlisti}), ref=\thethm.(\alph{corrlisti}),noitemsep}
\newlist{rmqlist}{enumerate}{1}
\setlist[rmqlist]{label=(\alph{rmqlisti}), ref=\thethm.(\alph{rmqlisti}),noitemsep}

\Crefname{thm}{Theorem}{Theorems}
\Crefname{thmI}{Theorem}{Theorems}
\Crefname{lm}{Lemma}{Lemmas}
\Crefname{prop}{Proposition}{Propositions}
\Crefname{corr}{Corollary}{Corollaries}
\Crefname{rmq}{Remark}{Remarks}
\Crefname{defi}{Definition}{Definitions}
\Crefname{conj}{Conjecture}{Conjectures}
\Crefname{exe}{Example}{Examples}

\Crefname{thmlisti}{Theorem}{Theorems}
\Crefname{lmlisti}{Lemma}{Lemmas}
\Crefname{proplisti}{Proposition}{Propositions}
\Crefname{corrlisti}{Corollary}{Corollaries}
\Crefname{rmqlisti}{Remark}{Remarks}

\renewcommand\labelenumi{(\roman{enumi})}
\renewcommand\theenumi\labelenumi

\newcommand{\codim}{\mathrm{codim}}
\newcommand{\bbk}{\mathbbm{k}}

\newcommand{\Ker}{\operatorname{Ker}}

\renewcommand{\Im}{\operatorname{Im}}
\newcommand{\Id}{\operatorname{Id}}
\newcommand{\id}{\operatorname{id}}

\newcommand{\Hom}{\operatorname{Hom}}

\newcommand{\Ext}{\operatorname{Ext}}
\newcommand{\End}{\operatorname{End}}

\newcommand{\bbN}{\mathbb{N}}
\newcommand{\bbZ}{\mathbb{Z}}

\newcommand{\bbQ}{\mathbb{Q}}

\newcommand{\bbP}{\mathbb{P}}
\newcommand{\bbC}{\mathbb{C}}

\newcommand{\ra}{\rightarrow}

\newcommand{\xra}[1]{\xrightarrow{#1}}

\renewcommand{\phi}{\varphi}

\newcommand{\Ocal}{\ensuremath{\mathcal{O}}}

\renewcommand{\bf}[1]{\mathbf{#1}}
\newcommand{\bb}[1]{\mathbb{#1}}
\newcommand{\cal}[1]{\mathcal{#1}}
\newcommand{\fk}[1]{\ensuremath{\mathfrak{#1}}}
\newcommand{\ona}[1]{\operatorname{#1}}

\newcommand{\ui}{\mathbf{i}}
\newcommand{\uj}{\mathbf{j}}
\newcommand{\uh}{\mathbf{h}}
\newcommand{\bfk}{\mathbbm{k}}
\newcommand{\cusp}{\operatorname{cusp}}
\newcommand{\Comp}{\operatorname{Comp}}
\newcommand{\Pol}{\operatorname{Pol}}
\newcommand{\Poll}{\operatorname{Poll}}

\renewcommand{\mod}{\operatorname{mod}}
\newcommand{\ev}{\operatorname{ev}}
\newcommand{\bfZ}{\mathbf Z}
\newcommand{\bfF}{\mathbf F}

\newcommand{\frakS}{\mathfrak{S}}
\newcommand{\rmS}{\mathrm{S}}
\newcommand{\rmM}{\mathrm{M}}
\newcommand{\rmR}{\mathrm{R}}

\newcommand{\T}{\mathcal{T}}
\newcommand{\FT}{\mathcal{F}}
\newcommand{\WR}{\mathfrak{W}} 
\newcommand{\Z}{\mathfrak{Z}} 
\newcommand{\Sc}{\mathcal{S}} 
\newcommand{\loc}{\mathrm{loc}}
\newcommand{\nc}{\text{nc}}
\newcommand{\doubleS}[2]{\tensor[^{#1}]{\fk{S}}{^{#2}}}

\renewcommand{\leqslant}{\leq}
\renewcommand{\geqslant}{\geq}

\newcommand{\dmerge}[4]{
	\draw (#1,#2) .. controls (#1,#4*0.5+#2*0.5) and (#3*0.5+#1*0.5,#4*0.5+#2*0.5) .. (#3*0.5+#1*0.5,#4);
	\draw (#3,#2) .. controls (#3,#4*0.5+#2*0.5) and (#3*0.5+#1*0.5,#4*0.5+#2*0.5) .. (#3*0.5+#1*0.5,#4);
}
\newcommand{\dsplit}[4]{
	\draw (#3*0.5+#1*0.5,#2) .. controls (#3*0.5+#1*0.5,#4*0.5+#2*0.5) and (#3,#4*0.5+#2*0.5) .. (#3,#4);
	\draw (#3*0.5+#1*0.5,#2) .. controls (#3*0.5+#1*0.5,#4*0.5+#2*0.5) and (#1,#4*0.5+#2*0.5) .. (#1,#4);
}
\newcommand{\opbox}[5]{
	\draw (#1,#2) rectangle (#3,#4);
	\node[] at (#3*0.5+#1*0.5,#4*0.5+#2*0.5) {#5};
}
\newcommand{\ntxt}[3]{
	\node[text height=1.2ex,text depth=.25ex] at (#1,#2) {#3};
}
\newcommand{\crosin}[4]{
	\draw (#1,#2) .. controls (#1,#4*0.6+#2*0.4) and (#3,#4*0.4+#2*0.6) .. (#3,#4);
	\draw (#3,#2) .. controls (#3,#4*0.6+#2*0.4) and (#1,#4*0.4+#2*0.6) .. (#1,#4);
}

\newcommand{\strex}[2]{
	\draw (#1-0.3,#2-0.3) -- (#1+0.3,#2+0.3);
	\draw (#1+0.3,#2-0.3) -- (#1-0.3,#2+0.3);
}
\newcommand{\strdot}[2]{
	\fill (#1,#2) circle (11pt);
}

\newlength{\dhatheight}

\theoremstyle{plain}
\newtheorem{thm}{Theorem} [section]
\newtheorem{thmI}{Theorem}

\newtheorem{lm}[thm]{Lemma}
\newtheorem{prop}[thm]{Proposition}
\newtheorem{corr}[thm]{Corollary}

\newtheorem{conj}[thm]{Conjecture}

\theoremstyle{definition}
\newtheorem{defi}[thm]{Definition}
\newtheorem*{defi*}{Definition}

\theoremstyle{remark}
\newtheorem{exe}[thm]{Example}
\newtheorem{rmq}[thm]{Remark}
\newtheorem*{nota}{Notation}

\AtEndEnvironment{ncorr}{\qed}

\title{KLR and Schur algebras for curves and semi-cuspidal representations}
\author{\sc Ruslan Maksimau}
\address{\parbox[b]{\linewidth}{Institut Montpelli\'erain Alexander Grothendieck, Universit\'e de Montpellier, France \\
Laboratoire Analyse Géométrie Modélisation, CY Cergy Paris Universit\'e, France}
}
\email{ruslan.maksimau@umontpellier.fr, ruslmax@gmail.com}

\author{\sc Alexandre Minets}
\address{Institute of Science and Technology Austria, Klosterneuburg, Austria}
\email{alexandre.minets@ist.ac.at}

\begin{document}
\begin{abstract}
Given a smooth curve $C$, we define and study analogues of KLR algebras and quiver Schur algebras, where quiver representations are replaced by torsion sheaves on $C$.
In particular, they provide a geometric realization for certain affinized symmetric algebras.
When $C=\bbP^1$, a version of curve Schur algebra turns out to be Morita equivalent to the imaginary semi-cuspidal category of the Kronecker quiver in any characteristic.
As a consequence, we argue that one should not expect to have a reasonable theory of parity sheaves for affine quivers.
\end{abstract}
\date{\today}
\maketitle
\tableofcontents

\section{Introduction}

\subsection{Motivation}
KLR algebras were introduced by Khovanov and Lauda~\cite{KL_DACQ2009} and Rouquier~\cite{Rou_2A2008} as a tool for categorification of quantum groups.
The geometric construction of this algebras was given by Varagnolo and Vasserot \cite{VV_CBK2011} and Rouquier~\cite{Rou_QHA22012}.
The positive characteristic version of this construction was done in~\cite{Mak_CBKA2015}.

Let us recall this geometric construction.
Let $\Gamma$ be a quiver without loops and let $\alpha$ be the dimension vector.
To this data we can associate a complex variety $\bfZ_\alpha$.
Its points are parameterized by triples, consisting of a representation of $\Gamma$ having dimension $\alpha$ together with two full flags of subrepresentations on it.
Then the algebra $R(\alpha)$ is isomorphic to the equivariant Borel-Moore homology $H_*^{G_\alpha}(\bfZ_\alpha)$, where $G_\alpha$ is a certain group of gauge transformations.  
The union of categories of (graded, projective, finitely generated) $R(\alpha)$-modules can be then equipped with induction and restriction functors.
These functors categorify product and coproduct in the quantum group $U^-_q(\mathfrak g_\Gamma)$, where $\mathfrak g_\Gamma$ is the Kac-Moody Lie algebra associated to $\Gamma$.

One of our motivations was to generalize this construction to other objects.
Namely, recall that by a theorem of Ringel-Green \cite{Rin_HAQG1990,Gre_HAHA1995} the quantum group $U^-_q(\mathfrak g_\Gamma)$ can be also realized as the spherical Hall algebra of the category of representations of $\Gamma$.
Another class of categories whose Hall algebras were actively considered is categories $\ona{Coh}C$ of coherent sheaves over smooth curves, see~\cite{Sch_LHA2012} for an overview.
In particular, starting with an elliptic curve, we get the elliptic Hall algebra, which was extensively studied under many different guises~\cite{MS_HSAT2017,BS_HAEC2012,Neg_SAR2014,SV_EHAE2013}.
Proceeding by analogy with quivers, we expect that KLR-like algebras associated to the category $\ona{Coh}C$ will provide an interesting categorification of the Hall algebra of $C$.

In the present paper, we are making first steps in this direction.
Namely, given a smooth curve $C$ we consider the moduli stack $\T_C = \ona{Tor}C$, which parameterizes torsion sheaves on $C$. 
Repeating the construction of KLR algebras, we consider the moduli of triples, consisting of a torsion sheaf of length $n$ together with two full flags of subsheaves.
Its Borel-Moore homology gets equipped with a convolution product, and we call the resulting algebra $\cal R^C_n$ the \textit{curve KLR algebra}.
\footnote{KLR algebras are often called ``quiver Hecke algebras'', so we could call $\cal R^C_n$ ``curve Hecke algebra''. We opted to not use this terminology, since Hecke algebras already appear in too many different contexts} 
Further, replacing full flags by partial flags, we define and study the \textit{curve Schur algebras} $\Sc^C_n$.
We obtain the following explicit description of $\Sc^C_n$, see \cref{sec:TorSh} for notations.

\begin{thmI}[\cref{SchurPolRepFaith}]\label{thm:A}
Let $P_n = \bigoplus_{\lambda\in \Comp(n)} \bf P_n^{\frakS_\lambda}$.
The algebra $\Sc^C_n$ can be identified with the subalgebra of $\End(P_n)$, generated by multiplication operators $P_n\subset\End(P_n)$, inclusions of invariants $\rmS_\lambda^{\lambda'}:\bf{P}_n^{\fk{S}_\lambda}\hookrightarrow \bf{P}_n^{\fk{S}_{\lambda'}}$ (split), and the merge operators
\[
\rmM_{\lambda'}^\lambda : \bf{P}_n^{\fk{S}_{\lambda'}}\to \bf{P}_n^{\fk{S}_{\lambda}},\qquad
\rmM_{\lambda'}^\lambda(P) = \sum_{a\in \fk{S}_\lambda/\fk{S}_{\lambda'}} \left(y\prod_{i=1}^{\lambda_k}\prod_{j=1}^{\lambda_{k+1}}\left(1+\frac{\Delta_{\tilde{\lambda}_{k-1}+i,\tilde{\lambda}_{k}+j}}{x_{\tilde{\lambda}_{k}+j}-x_{\tilde{\lambda}_{k-1}+i}}\right)\right)^a,
\]
where $\lambda' = (\lambda_1,\ldots,\lambda_r)$ is a composition of $n$, $\tilde{\lambda}_k=\sum_{1\leq i\leq k} \lambda_i$, and $\lambda =  (\lambda_1,\ldots,\lambda_{k-1},\lambda_k+\lambda_{k+1},\lambda_{k+2},\ldots,\lambda_r)$.
\end{thmI}

We also provide an explicit basis and a diagrammatic presentation for $\Sc^C_n$, see \cref{lm:basis-Schur-zigzag}.

It turns out that the integral version of $\Sc^C_n$ for $C=\bbP^1$ is intimately related to the representation theory of KLR algebras in type $\widehat{\mathfrak{sl}}_2$.

\subsection{Semi-cuspidal categories} 
Let $\Gamma$ be a quiver of affine type.
It is known~\cite{McN_RKAI2017,KM_SKAA2017} (under some conditions on the characteristic of the base field) that the KLR algebra $R(\alpha)$ is properly stratified, see~\cite{Kle_AHWC2015} for the definition of this property.
Informally speaking, this means that one can slice the category of $R(\alpha)$-modules into a collection of categories $C(n\xi)\mathrm{-mod}$, where $\xi$ is a positive root.
The category $C(n\xi)\mathrm{-mod}$ is the category of semi-cuspidal $R(n\xi)$-modules.
It is easy to describe if $\xi$ is a real root, but becomes much more complicated when $\xi=\delta$ is the imaginary root.
In the present paper, we shed some light on this problem by finding an explicit diagrammatic algebra, which is Morita equivalent to $C(n\delta)$ in any characteristic.

When working over $\bbk$ a field of characteristic zero, this was already done in~\cite{KM_AZAI2019}.
In this case $C(n\delta)$ can be shown to be Morita equivalent to $e_0 C(n\delta)e_0$ for some simple and explicit idempotent $e_0$.
For any $\bbZ_{\ge 0}$-graded symmetric algebra $F$, Kleshchev and Muth introduce affinized symmetric algebra $\WR_n(F)$ of rank $n$, and then prove an isomorphism $e_0 C(n\delta)e_0\simeq \WR_n(F)$ for a specific choice of $F$.
In particular, in type $\widehat{\mathfrak{sl}}_2$ one has $F = \bbk[c]/(c^2)$.

In positive characteristic, the algebras $C(n\delta)$ and $e_0 C(n\delta)e_0$ are not Morita equivalent any more.
It is possible to find a more complicated idempotent $e$ such that the algebras $C(n\delta)$ and $eC(n\delta)e$ are Morita equivalent.
However, no explicit description of $eC(n\delta)e$ is known in general.

The starting point of our contribution is the following observation:
\begin{thmI}[\cref{KLRisowreath}, \cref{subs:P1anyring}]
We have an isomorphism of algebras $\cal R_n^C \simeq \WR_n(H^*(C,\bbQ))$.
When $C = \bb P^1$, this isomorphism holds over any field $\bbk$.
\end{thmI}

This suggests that the curve Schur algebras $\Sc_n^C$ can be related to the imaginary semi-cuspidal categories. 
In effect, let $\Gamma$ be the Kronecker quiver, and $C=\bbP^1$.
Using the well-known derived equivalence between coherent sheaves on $\bbP^1$ and representations of $\Gamma$, we produce a homomorphism $\Phi_n:eC(n\delta)e\to \Sc_n^{\bbP^1}$.
It is constructed in a geometric fashion, and is defined over any field, as well as $\bbk=\bbZ$.
It turns out that $\Phi_n$ is bijective if $\bfk$ is a field of characteristic zero and is injective for $\bfk=\bbZ$, with the image $\widetilde{\Sc}_n^{\bbP^1}:=\Im\Phi_n$ being a sublattice of full rank in $\Sc_n^{\bbP^1}$.
Note that $\Phi_n$ is not an isomorphism; this discrepancy is related to the fact that the integral cohomology groups of the stack $\T_C$ are not generated by tautological classes, see \cref{ex:d2nonsurj,prop:phi=TH}.
In conclusion, we get the following result:
\begin{thmI}[\cref{thm:cusp-computed}]
Let $\Gamma$ be the Kronecker quiver, and let $\delta$ be the imaginary simple root.
Denote $\widetilde{\Sc}_n^{\bb F_p} = \widetilde{\Sc}_n^{\bbP^1}\otimes_\bbZ \bb F_p$.
For any $n>0$ and $p$ prime, we have an isomorphism $eC_{\bb F_p}(n\delta)e\simeq \widetilde{\Sc}_n^{\bb F_p}$.
\end{thmI}
Note that $H^*(\bbP^1,\bbk)\simeq \bbk[c]/(c^2)$, so as a byproduct we obtain a new geometric proof of the isomorphism of Kleshchev-Muth in type $\widehat{\mathfrak{sl}}_2$.

The sublattice $\widetilde{\Sc}_n^{\bbP^1}$ can be described in terms of \cref{thm:A}.
Namely, we give a certain explicit sublattice $\widetilde{P}_n\subset P_n$ which is preserved under the action of $\widetilde{\Sc}_n^{\bbP^1}$.
The algebra $\widetilde{\Sc}_n^{\bbP^1}$ is then generated by multiplication operators in $\widetilde{P}_n$, together with split and merge operators $\rmS_\lambda^{\lambda'}$, $\rmM^\lambda_{\lambda'}$.
This allows us to obtain a diagrammatic description, an explicit basis and a polynomial representation $\widetilde{P}_n\otimes_\bbZ \bb F_p$ for $\widetilde{\Sc}_n^{\bb F_p}$.
We conjecture that this representation is faithful, see \cref{conj:pol-Fp-faith}.

In \cref{app:parity} we discuss a consequence of the fact that the map $\Phi_n$ is not surjective over $\bb F_p$.
We show that for the Kronecker quiver the fibers of the flag version of Springer resolution have even cohomology groups over $\bbZ$.
For a quiver of Dynkin type, this would be enough to exhibit a nice theory of parity sheaves on the quiver variety~\cite{Mak_CBKA2015}.
However, the existence of such theory for the Kronecker quiver would imply surjectivity of $\Phi_n$.

\subsection{Future work}
We expect that applying our approach to curves with orbifold points will shed light on the semi-cuspidal category $C(n\delta)\rm{-mod}$ in other types.
It would be also interesting to deduce some combinatorics of $C(n\delta)\rm{-mod}$ from our explicit description of $eC(n\delta)e$. 

Concerning the categorification questions, the next logical steps would be to consider Schur algebras for the whole category $\ona{Coh}C$, including sheaves of positive rank.
We plan to investigate this in the future.
For $C=\bbP^1$, partial results in this direction were obtained in~\cite{SVV_CCQL2019}.
For $C=E$ an elliptic curve, we hope to obtain a categorification of elliptic Hall algebra, compatible with the action of the braid group $B_3$ on $D^b(\ona{Coh}E)$.

\subsection{Organization of paper}
We start by recalling the theory of convolution algebras and their localization in \cref{ConvSetup}.
Next, we introduce the moduli stack of (flags of) torsion sheaves on a smooth curve and prove some its properties in \cref{sec:TorSh}.
In \cref{sec:schur} we introduce curve Schur algebras $\Sc_n^C$, and construct a basis and a faithful representation for them.
A certain simple subalgebra of $\Sc_n^C$ is described by generators and relations in \cref{sec:KLRcurve}.
In \cref{sec:P1int}, we discuss in detail the integral version of $\Sc_n^C$ for $C = \bbP^1$.
In \cref{sec:KLRquivers}, we recall some properties of KLR algebras and their divided power version.
In \cref{sec:semicusp}, we provide a description of the semi-cuspidal category of the Kronecker quiver in positive characteristic in terms of $\Sc_n^C$.
Finally, these results are used in \cref{app:parity} to show that there is no satisfactory theory of parity sheaves for the Kronecker quiver.

\subsubsection*{Acknowledgments}
A crucial role in the genesis of this paper was played by Alexander Kleshchev. We thank him for numerous fruitful discussions, for sharing his insight into semi-cuspidal representations of KLR algebras and particularly for pointing out to us that the map $\Phi_n$ should not be surjective over $\bbZ$.
The authors would also like to thank Anton Mellit, Olivier Schiffmann, \'Eric Vasserot for stimulating discussions.
This collaboration between the two authors started from a conversation at the workshop ``Geometric representation theory and low-dimensional topology'' at ICMS, Edinburgh, and was in parts conducted during the thematic trimester on representation theory at IHP, Paris.

\subsubsection*{Notations}
All varieties we consider are defined over $\bbC$, and $\dim (-)$ always means the complex dimension.
The coefficient ring of $H_*$ is denoted by $\bfk$.
We always assume either that $\bbk$ is a field, or $\bbk = \bbZ$.
In \cref{sec:schur,sec:KLRcurve} we additionally assume that $\bfk$ is a field of characteristic zero.
We will almost always drop the coefficient ring from the notation.

For any $G$-variety $X$ we define $H^*_G(X):=H_{2\dim X-*}^G(X)$ by abuse of notation.
When $X$ is smooth, we recover the usual cohomology groups, while for general $X$ this is usually not true.
We introduce this notation solely for the purpose of getting correct gradings later on, and will avoid it whenever possible.
We will \textit{never} consider usual cohomology groups for singular varieties.

\section{Localization of convolution algebras}\label{ConvSetup}
\subsection{Borel-Moore homology and refined pullbacks}
Recall that for an algebraic variety $X$, its Borel-Moore homology is defined as relative homology with respect to some compactification of $X$.
In what follows, we will drop the superscript and write $H_*(X)=H_*^{BM}(X)$.
For any proper map $f:X\to Y$, we have the direct image $f_*:H_*(X)\to H_*(Y)$.
For any lci\footnote{that is, a composition of a regular embedding and a smooth map} morphism $g:X\to Y$, we have the pullback map
\[
g^*:H_*(Y)\to H_{*+2d}(X),
\] 
where $d$ is the relative dimension of $g$.
Further, let $h:Y'\to Y$ be an arbitrary morphism.
Form a cartesian square
\begin{equation}\label{CartforGysin}
\begin{tikzcd}
Y'\times_Y X\ar[r,"g'"]\ar[d,"h'"] & Y'\ar[d,"h"]\\
X\ar[r,"g"] & Y
\end{tikzcd}
\end{equation}
Then one can define the \textit{refined pullback} $(g')^!_g:H_*(Y')\to H_{*+2d}(Y'\times_Y X)$.
In particular, if $X,Y'\subset Y$ are closed subvarieties, and both $X$ and $Y$ are smooth, we get a restriction map $H_*(Y')\to H_{*-2\codim_Y X}(Y'\cap X)$.
\begin{rmq}
As notation suggests, $(g')^!_g$ depends on the whole cartesian square~\eqref{CartforGysin}, and not just the map $g'$.
However, we will often drop the subscript, when the choice of cartesian square is clear.
\end{rmq}
For a closed embedding of smooth varieties $X\subset Y$, we denote its normal bundle by $N_X Y$.
We say that a diagram is a fiber diagram if all squares in it are cartesian.
\begin{prop}\label{GysinProp}
\begin{proplist}
	\item \label{GysinProp:a} For any fiber diagram
	\[
	\begin{tikzcd}
		Y'\times_Y X_2\ar[r,"g'_2"]\ar[d] & Y'\times_Y X_1\ar[r,"g'_1"]\ar[d] & Y'\ar[d]\\
		X_2\ar[r,"g_2"] & X_1\ar[r,"g_1"] & Y
	\end{tikzcd}
	\]
	we have $(g'_1\circ g'_2)^!_{g_1\circ g_2}=(g'_2)^!_{g_2}\circ(g'_1)^!_{g_1}$, provided that $g_1$ and $g_2$ are lci;
	\item \label{GysinProp:b} consider a fiber diagram
	\[
	\begin{tikzcd}
		X''\ar[r,"g''"]\ar[d,"f'"] & Y''\ar[d,"f"]\\
		X'\ar[r,"g'"]\ar[d]& Y'\ar[d]\\
		X\ar[r,"g"]& Y
	\end{tikzcd}
	\]
	with both $g$ and $g'$ regular embeddings.
	Then $(g'')_g^!=e(h'^*(N_X Y)/N_{X'} Y')\cdot (g'')_{g'}^!$;
	\item \label{GysinProp:b2} consider a fiber diagram as in (b).
	If $f$ is proper, then $(g')^!_g\circ f_* = f'_*\circ (g'')_g^!$;
	\item \label{GysinProp:c} consider a fiber diagram
	\[
	\begin{tikzcd}
		X''\ar[r,"g''"]\ar[d,"h''"] & Y''\ar[r]\ar[d,"h'"]&Z'\ar[d,"h"]\\
		X'\ar[r,"g'"]\ar[d]& Y'\ar[r]\ar[d]&Z\\
		X\ar[r,"g"]& Y &
	\end{tikzcd}
	\]
	with $g$ and $h$ lci morphisms.
	Then $(g'')^!\circ(h')^!=(h'')^!\circ(g')^!$;
	\item \label{GysinProp:d} suppose $g$ in~\eqref{CartforGysin} is a closed embedding.
	Then $(g')^!\circ (g')_*(-)=e((h')^*N_X Y)\cdot -$.
\end{proplist}
\end{prop}
\begin{proof}
For (a-d), see~\cite{FM_CFSS1981} and \cite[Chapter 6]{Ful_IT1998}.
The part (e) follows by setting $Y''=X'$ in (c), and further $Y'=Y''=X'$ in (b).
\end{proof}

\subsection{Localization theorem}
Let $T\subset G$ be a reductive group together with a fixed maximal torus, and denote by $W$ the corresponding Weyl group.
In this paper, we will be chiefly considering equivariant Borel-Moore homology groups.
\cref{GysinProp} extends to the equivariant case by the argument in~\cite[Appendix B.9]{AHR_GSSC2015}.
For brevity, we will always denote the $G$-equivariant cohomology of a point $H_G^*(pt)$ by $H_G$.

\begin{prop}[{\cite[III.\S 1]{Hsi_CTTT1975}}]\label{GvsTW}
Let $X$ be a $G$-variety, and $\bbk$ a field of characteristic $0$.
Then $H^G_*(X,\bbk)\simeq H^T_*(X,\bbk)^W$.
\end{prop}

Let $X$ be a $T$-variety.
The homology group $H_*^T(X)$ is naturally an $H_T$-module; we will write $H_*^T(X)_\loc$ for its localization $H_*^T(X)\otimes_{H_T}\ona{Frac}(H_T)$.

Let $X^T$ be the subvariety of points in $X$ fixed by $T$, and the inclusion $i_X:X^T\hookrightarrow X$ the natural embedding.
\begin{prop}[Localization theorem]\label{LocThm}
Let $T$ be an algebraic torus, and $X$ a $T$-variety. 
Suppose that $X^T$ is not empty.
Then the $\ona{Frac}(H_T)$-linear map 
\[
i_{X*}:H^T_*(X^T)_\loc\to H^T_*(X)_\loc
\]
is an isomorphism.
Moreover, assume that $\bfk$ is a torsion-free $\bbZ$-module.
Then for any $T$-equivariant closed embedding $X\hookrightarrow Y$ into a smooth $T$-variety $Y$, the map
\[
(i_X)^!_{i_Y}:H^T_*(X)_\loc\to H^T_*(X^T)_\loc
\]
is an isomorphism as well.
\end{prop}
\begin{proof}
First claim is proved in~\cite[III.\S 1]{Hsi_CTTT1975}.
Second claim is obtained by applying \cref{GysinProp:d}.
\end{proof}

\begin{rmq}\label{rmq:locthmcoef}
Note that we only need the assumption on $\bbk$ to assure that the Euler class in \cref{GysinProp:d} is not a zero divisor.
Thus the proposition will hold for other $\bbk$, if we can check this condition separately.
\end{rmq}

Applying \cref{GysinProp:d}, we get a useful corollary.
\begin{lm}\label{LocPullPush}
	Let $f:X'\ra Y'$ be a projective morphism of smooth $T$-varieties, $Y\subset Y'$ is a closed $T$-stable subvariety, and $X=Y\times_{Y'} X'$.
	Assume that fixed point sets $X^T$, $Y^T$ are non-empty, and let $f_T:X^T\to Y^T$ be the restriction of $f$.
	Then we have a base change formula in localized homology groups:
	\[
	i^!_Yf_*(-)=e(N_{(Y')^T}Y')\cdot f_{T*}\left( e(N_{(X')^T}X')^{-1}\cdot i^!_X(-) \right).
	\]
\end{lm}

\begin{corr}
Let $X$ be a $G$-variety.
Assume that $H_*^G(X)$ is a torsion-free $H_G$-module.
Then the composition
\[
H_*^G(X)\subset H_*^T(X) \to H_*^T(X)_\loc \simeq H_*^T(X^T)_\loc
\]
is injective.
\end{corr}
\begin{proof}
It's enough to check that for any $a\in H_*^G(X)$, its annihilator inside the $H_T$-module $H^T_*(X)$ is trivial.
Since $H_*^G(X)$ is torsion free, we have $\ona{Ann}(a)\cap H_G = 0$.
If $pa = 0$ for $p\in H_T$, then $\prod_{\sigma \in W} \sigma(p)$ lies in the above intersection.
Since $H_T$ is integral, we conclude that $p=0$.
\end{proof}

Note that for a $G$-variety $X$, its homology $H_*^T(X^T)$ acquires a $W$-action, induced diagonally from the actions on $X^T$ and $T$.
We therefore obtain an embedding $H_*^G(X)\subset \left( H_*^T(X^T)_\loc \right)^W$.

\subsection{Convolution algebras}\label{convalg}
Let $\pi: Y\to X$ be a proper morphism between smooth varieties.
For any $k\geq 1$, define
\[
Z^{(k)}=\underbrace{Y \times_X \ldots\times_X Y }_{k\text{ times}},
\]
and let $j_k:Z^{(k)}\hookrightarrow Y^{k+1}$ be the natural embedding.
We will write $Z=Z^{(1)}$, $Z^{(0)}=Y$, and $j=j_1$.

For any finite set of indices $I=\{ i_1<i_2<\ldots <i_k\}$, where $i_k\leq n$, consider the natural projections onto the coordinates contained in $I$:
\[
p_I=p_{i_1\ldots i_k}:Y^{n}\to Y^{k}.
\]
We will denote the corresponding restrictions $Z^{(n-1)}\to Z^{(k-1)}$ by the same letter.
In particular, for any $1\leq i\leq k+1$, we have a map $p_i:Z^{(k)}\to Y$.
Since $Y$ is smooth, the embedding $(p_i,\id_Z):Z^{(k)}\hookrightarrow Y\times Z^{(k)}$ is regular, so that the pullback along it provides us with an $H^{-*}(Y)$-module structure on $H_*(Z^{(k)})$.
We will denote this action by $\gamma\cdot_i x$, where $\gamma\in H^*(Y)$, $x\in H_*(Z^{(k)})$.

Consider the following diagram with cartesian square:
\[
\begin{tikzcd}
Z\times Z\ar[d,hook] & Z^{(2)}\ar[l,"{(p_{12}, p_{23})}"']\ar[r,"p_{13}"]\ar[d,hook]& Z\\
Y^{2}\times Y^{2} & Y^{3}\ar[l] &
\end{tikzcd}
\]
For each $\gamma\in H^*(Y)$, we have the following convolution product on $H_*(Z)$:
\begin{align*}
*_\gamma:H_*(Z)\otimes H_*(Z)\to H_{*-2\dim Y-\deg \gamma}(Z),\\
a\otimes b\mapsto (p_{13})_*\left(\gamma\cdot_2(p_{12}, p_{23})^!(a\otimes b) \right),
\end{align*}
where the refined pullback $(p_{12}, p_{23})^!$ is defined with respect to the regular embedding $Y^{(3)}\hookrightarrow Y^{(2)}\times Y^{(2)}$.

\begin{prop}\label{ConvProduct}
$\cal A_\gamma=\cal A_\gamma(\pi)=(H_*(Z),*_\gamma)$ is an associative algebra.
\end{prop}
\begin{proof}
We have the following diagram with cartesian square:
\[
\begin{tikzcd}
Z^{(3)}\ar[d,"{(p_3,p_{123},p_{34})}"']\ar[r,"p_{134}"]&Z^{(2)}\ar[d,,"{(p_2,p_{12},p_{23})}"]\ar[r,"p_{13}"] &Z \\
Y\times Z^{(2)}\times Z\ar[d,"{\id_Y\times (p_2,p_{12},p_{23})\times \id_Z}"']\ar[r,"{\id_Y\times p_{13}\times \id_Z}"]& Y\times Z\times Z&\\
Y\times Y\times Z\times Z\times Z& &
\end{tikzcd}
\]
Lemma A.12(2) in~\cite{Min_CHAH2020} shows that we can do base change along the square.
In particular,
\[
(a*_\gamma b)*_\gamma c=p_{14*}(p_2\times p_3\times p_{12}\times p_{23}\times p_{34})^!(\gamma\otimes \gamma\otimes a\otimes b\otimes c).
\]
Using a similar diagram, we can prove the same equality for $a*_\gamma (b*_\gamma c)$, so that the associativity follows.
\end{proof}

In the same fashion, we have a map
\begin{equation}\label{Aaction}
\begin{aligned}
H_*(&Z) \otimes H_*(Y)\to H_{*-\deg\gamma}(Y),\\
a & \otimes  x\mapsto (p_{1})_*((\gamma x)\cdot_2 a).
\end{aligned}
\end{equation}

The following statement is proved analogously to \cref{ConvProduct}.
\begin{prop}\label{ConvModule}
The map~\eqref{Aaction} defines an $\cal A_\gamma$-module structure on $H_*(Y)$.
\end{prop}

\begin{nota}
In what follows, we will call $\gamma$ the \textit{twist}, and drop the subscript if $\gamma=1$.
\end{nota}

\begin{exe}\label{DiagEmb}
Consider the identity map $Y\to Y$.
The associated convolution algebra is simply $H_*(Y)$ together with intersection product.
Moreover, the closed embedding $Y\simeq Y\times_Y Y\hookrightarrow Y\times_X Y = Z$ defines a homomorphism of algebras $H_*(Y)\to \cal A$, and restriction of the action in \cref{ConvModule} to $H_*(Y)$ coincides with the left action of $H_*(Y)$ on itself.
\end{exe}

\begin{exe}\label{EmbMatrix}
Suppose $Y$ is proper, and consider the map $Y\to pt$.
The associated convolution algebra is the matrix algebra $\End(H_*(Y))$.
Moreover, the closed embedding $Z = Y\times_X Y\hookrightarrow Y\times Y$ defines a homomorphism of algebras $\cal A\to \End(H_*(Y))$, which coincides with the map induced by~\eqref{Aaction}.
\end{exe} 

\subsection{Localization of convolution algebras}
Suppose now that $T$ is an algebraic torus, $X$ and $Y$ are $T$-varieties, and $\pi$ is $T$-equivariant.
Note that $Z^T=Y^T\times_{X^T} Y^T$.
Let us further assume that $\pi_{Y^T}:Y^T\to X^T$ is a submersion, so that $Z^{(k)T}$ is smooth for any $k\geq 1$.
Therefore, \cref{ConvProduct} produces algebra structures on $H_*^T(Z)$ and $H_*^T(Z^T)_\loc$.
Let us call these algebras $\cal A_\bullet$ and $\cal A_\bullet^T$ respectively, where subscripts stand for the twist.

Let $e(Y)\in H^*_T(Y^T)$ denote the equivariant Euler class of the normal bundle $N_{Y^T}Y$.
\begin{prop}\label{ConvLocAlg}
The localization map $i_Z^!$ induces an algebra homomorphism $\cal A\to \cal A^T_{e(Y)^{-1}}$.
\end{prop}
\begin{proof}
Consider the localization diagram:

\[
\begin{tikzcd}
Z\times Z & Z^{(2)}\ar[l,"p_{12}\times p_{23}"']\ar[r,"p_{13}"] & Z\\
Z^T\times Z^T\ar[u,"i_Z\times i_Z"] & Z^{(2)T}\ar[l,"p_T"']\ar[r,"q_T"]\ar[u,"i_{Z^{(2)}}"] & Z^T\ar[u,"i_Z"]
\end{tikzcd}
\]
Note that the left square is cartesian, while the one on the right is only commutative.
By \cref{GysinProp:c}, we have 
\[
(i_{Z^{(2)}})^!_{i_{Y^4}} \circ (p_{12}\times p_{23})^!=(p_T)^!_{p_{12}\times p_{23}}\circ (i_{Z\times Z})^!_{i_{Y^4}}.
\]
On the other hand, by \cref{LocPullPush} we have 
\begin{align*}
i_Z^! p_{13*}(-) & = q_{T*}(e(N_{Y^{4T}}Y^4)^{-1}f_T^*(e(N_{Y^{2T}}Y^2))\cdot (i_{Z^{(2)}})^!_{i_{Y^4}}(-))\\
& = q_{T*}\left(e(Y)^{-2}\cdot_2 (i_{Z^{(2)}})^!_{i_{Y^4}}(-)\right)
\end{align*}
Finally, \cref{GysinProp:b} shows that
\[
(p_T)^!_{p_{12}\times p_{23}}=e(Y)\cdot_2 p_T^*.
\]
Putting everything together, we get
\begin{align*}
i_Z^!\circ p_{13*}\circ (p_{12}\circ p_{23})^!(-) & =q_{T*}\left(e(Y)^{-2}\cdot_2 (i_{Z^{(2)}})^!_{i_{Y^4}}\circ (p_{12}\circ p_{23})^!(-)\right)\\
& = q_{T*}\left(e(Y)^{-2}\cdot_2 (p_T)^!_{p_{12}\times p_{23}}\circ i_{Z^2}^!(-)\right)\\
& = q_{T*}\left(e(Y)^{-1}\cdot_2 p_T^*\circ i_{Z^2}^!(-)\right),
\end{align*}
which proves that $i_Z^!$ commutes with multiplication.
\end{proof}

\begin{prop}\label{ConvLocMod}
We have a commutative square
\[
\begin{tikzcd}
	\cal A\otimes H_*(Y)\ar[d,"i_Z^!\times i_Y^*"]\ar[r] & H_*(Y)\ar[d,"i_Y^*"]\\
	\cal A^T_{e(Y)^{-1}}\otimes H_*^T(Y^T)_\loc \ar[r] & H_*^T(Y^T)_\loc
\end{tikzcd}
\]
where the horizontal maps are defined by~\eqref{Aaction}.
\end{prop}
\begin{proof}
Analogously to \cref{ConvLocAlg}, we have
\begin{align*}
i_Y^*(p_1)_*(p_2\times \id)^* & = (p_{1T})_*\left( e(Y)^{-1}\cdot i^*_Z\circ(\id\times p_2)^* \right)\\
& = (p_{1T})_*\left( e(Y)^{-1}\cdot (\id\times p_{2T})^*\cdot (i_Z\times i_Y)^*\right),
\end{align*}
which proves the statement.
\end{proof}

\begin{rmq}
Suppose $Y^T$ is proper, and write $\gamma = e(Y)$.
Consider the following commutative square:
\[
\begin{tikzcd}
Z\ar[r,"j"] & Y\times Y\\
Z^T\ar[r,"j_T"]\ar[u,"i_Z"] & Y^T\times Y^T\ar[u,"i_{Y^2}"]
\end{tikzcd}
\]
Similarly to \cref{ConvLocAlg}, the composition $i_{Y^2}^*\circ j_*:H_*^T(Z)\to H_*^T((Y^T)^2)_{loc}$ defines a homomorphism
\begin{equation}\label{PolyRepFactor}
\cal A\to \End_{\deg \gamma}H_*(Y^T)_\loc,
\end{equation}
where the product on the right is given by $(a,b) = \gamma^{-1}\cdot (a\circ b)$.
\cref{LocPullPush} applied to the square above shows that~\eqref{PolyRepFactor} factors as
\[
\cal A\xra{i_Z^!} \cal A^T_{\gamma^{-1}}\to \End_{\deg \gamma}H_*(Y^T)_\loc,
\]
where the second map is defined as in \cref{EmbMatrix}.
\end{rmq}

\subsection{Convolution from finite group action}\label{ConvEx}
Let us conclude this section with an easy example, which will become useful later.
Namely, let $\Gamma$ be a finite group acting on a smooth variety $X$, and set $Y=\Gamma\times X$, with $\pi: Y\to X$ being the projection.
We clearly have $Z = \Gamma^2\times X$, and
\begin{equation}\label{ExAlg}
H_*(Y)\simeq \bfk[\Gamma]\otimes H_*(X),\quad H_*(Z)\simeq \bfk[\Gamma]^{\otimes 2}\otimes H_*(X).
\end{equation}

Fix a class $\gamma \in H^*(X)$, and let
\[
\gamma^\circ = \sum_{g\in\Gamma} g\otimes \gamma^g \in H_*(Y),
\]
where $x^g$ denotes the image of $x\in H_*(X)$ under the action of $g\in \Gamma$.
Consider the algebra $\cal A_{\gamma^\circ}$.
As a vector space, it is isomorphic to $H_*(Z)$, while the product is given by
\begin{equation}\label{ExProd}
(g_1\otimes g_2\otimes x)* (h_1\otimes h_2\otimes y)=\delta_{g_2,h_1}(g_1\otimes h_2\otimes xy\gamma^{g_2}),
\end{equation}

Note that if we equip $Y=\Gamma\times X$ with diagonal $\Gamma$-action, $\pi$ becomes $\Gamma$-equivariant.
Moreover, $\gamma^\circ\in H_*(Y)$ is a $\Gamma$-invariant class.
Therefore, $\Gamma$ acts on $\cal A_{\gamma^\circ}$ via algebra automorphisms; under the isomorphism~\eqref{ExAlg}, it gets identified with the diagonal action on $\bfk[\Gamma]^{\otimes 2}\otimes H_*(X)$.
Consider the $\Gamma$-invariant subalgebra $\cal A^\Gamma_{\gamma^\circ}\subset \cal A_{\gamma^\circ}$.
Its basis is given by elements
\[
\xi_{(g,x)}=\sum_{h\in \Gamma}h\otimes hg\otimes x^h.
\]
Using the formula~\eqref{ExProd}, we get
\begin{align}
\label{eq:mult-xi}
\xi_{(g,x)}*\xi_{(h,y)} & =\sum_{f_1,f_2\in \Gamma}\delta_{f_1g,f_2}(f_1\otimes f_2h\otimes x^{f_1}(y\gamma)^{f_2})= \sum_{f\in \Gamma} f\otimes fgh\otimes x^{f}(y\gamma)^{fg}\\
&=\xi_{(gh,x(y\gamma)^g)}.\nonumber
\end{align}
In particular, assume that $\gamma$ is invertible and of even degree.
Denote $\widetilde{\xi}_{(g,x)}=\xi_{(g,x\gamma^{-1})}$.
Then $\widetilde{\xi}_{(g,x)}*\widetilde{\xi}_{(h,y)}=\widetilde{\xi}_{(gh,xy^g)}$, so that $\cal A^\Gamma_{\pi^*\gamma}$ is isomorphic to the semi-direct tensor product $H_*(X)\rtimes\bbC[\Gamma]$.

In the same way, $H_*(Y)^\Gamma$ is an $\cal A^\Gamma_{\gamma^\circ}$-module.
We have
\[
H_*(X)=H_*(Y)^\Gamma\hookrightarrow H_*(Y),\quad x\mapsto x^\circ.
\]
Under this identification, the action is given by
\begin{equation}\label{ExAct}
\xi_{(g,x)}.y=x(y\gamma)^g,
\end{equation}
or equivalently $\widetilde{\xi}_{(g,x)}.y = xy^g\frac{\gamma^g}{\gamma}$.
Setting $\widetilde{\psi}_y=y\gamma^{-1}$, we get $\widetilde{\xi}_{(g,x)}.\widetilde{\psi}_y=\widetilde{\psi}_{xy^g}$.

\begin{rmq}
The action map $a:Y=\Gamma\times X\to X$ is $\Gamma$-equivariant, where $\Gamma$ acts by multiplication on the first coordinate of $Y$.
It is easy to check that the map $Y\to Y$, $(g,x)\mapsto (g,g.x)$ induces an isomorphism of algebras $\cal A_{\pi^*\gamma}^\Gamma(a)\simeq \cal A^\Gamma_{\gamma^\circ}(\pi)$.
\end{rmq}
\section{Torsion sheaves on curves}\label{sec:TorSh}
\subsection{Flag varieties}\label{subs:Flag}
In this subsection, we recall some standard facts about flag varieties.

For each $n$, consider $\bbC^n$ together with its standard basis $e_1,\ldots, e_n$, and let $V_k=\bigoplus_{i=1}^k \bbC e_i$ for any $1\leq k\leq n$.
Let $G_n=GL(\bbC^n)$, $\fk{S}_n$ the symmetric group on $n$ symbols, and let $T_n\subset B_n\subset G_n$ be the maximal torus and the Borel subgroup, associated to the basis above.
We call a tuple of positive integers $\lambda=(\lambda_1,\ldots \lambda_k)$ a \textit{composition} of $n$, if $\sum_i \lambda_i = n$ (the length $k$ is not fixed), and denote by $\Comp(n)$ the set of thereof. 
We also set $\tilde{\lambda}_i = \lambda_1+\ldots+\lambda_i$.
We introduce the following index subsets of $\bbZ^2$:
\begin{align*}
N_\lambda & = \bigcup_{0\leq i< j\leq r-1} [\tilde{\lambda}_i+1,\tilde{\lambda}_{i+1}]\times [\tilde{\lambda}_j+1,\tilde{\lambda}_{j+1}],\\
I_\lambda & = N_\lambda\cup\left(\bigcup_i [\tilde{\lambda}_i+1,\tilde{\lambda}_{i+1}]^2\right)\setminus \left\{ (i,i):1\leq i\leq n \right\}.
\end{align*}
For each $\lambda\in \Comp(n)$, denote 
\[
\fk{S}_\lambda=\fk{S}_{\lambda_1}\times \ldots\times \fk{S}_{\lambda_k}\subset \fk{S}_n,\qquad G_\lambda=G_{\lambda_1}\times \ldots\times G_{\lambda_k}\subset G_n,\qquad P_\lambda=G_\lambda B_n.
\]
The partial flag variety $G_n/P_\lambda$ will be denoted by $\mathscr{F}_\lambda$; we will also write $\mathscr{F}_n:=\mathscr{F}_{1^n}$ for the complete flag variety.

One can identify $\fk{S}_n/\fk{S}_\lambda$ with the set of minimal length coset representatives, which we denote by $\fk{S}^\lambda$:
\[
\fk{S}^\lambda = \left\{ \sigma\in \fk{S}_n \mid \sigma(i)<\sigma(j)\text{ if } \tilde\lambda_k < i<j\leq \tilde\lambda_{k+1} \text{ for some $k$} \right\}.
\]
We have analogous identifications for right and double cosets:
\begin{align*}
\fk{S}_\lambda\backslash \fk{S}_n & \simeq \tensor[^\lambda]{\fk{S}}{} :=(\fk{S}^\lambda)^{-1} = \left\{ \sigma^{-1}\mid \sigma \in \fk{S}^\lambda \right\},\\
\fk{S}_\mu\backslash \fk{S}_n/\fk{S}_\lambda & \simeq \doubleS{\mu}{\lambda} := \tensor[^\mu]{\fk{S}}{} \cap \fk{S}^\lambda.
\end{align*}

For any $w\in\fk{S}_n$, let $F_{w}\in\mathscr{F}_\lambda$ be the flag $w.V_{\lambda_1}\subset w.V_{\lambda_1+\lambda_2}\subset\ldots\subset \bbC^n$.
Note that $F_w$ depends only on $w\frakS_\lambda$.
The flags $F_{w}$ are precisely the $T_n$-fixed points in $\mathscr F_\lambda$.
Moreover, they are in one-to-one correspondence with left $B_n$-orbits in $\mathscr F_\lambda$:
\begin{equation*}
	\mathscr F_\lambda = \bigsqcup_{w\in \fk{S}^\lambda} B_n.F_w.
\end{equation*}
Let us denote $O_{w}^\lambda=B_n.F_w\subset \mathscr F_\lambda$; we will omit the superscript when the choice of parabolic subgroup is clear.
Each of these strata is an affine space.

The Bruhat order on $\fk{S}_n$ induces a partial order on $\fk{S}^\lambda$.
It coincides with the orbit closure order on $\mathscr{F}_\lambda$:
\[
\forall w_1,w_2\in \frakS^\lambda,\qquad [w_1]\leq [w_2] \Leftrightarrow O_{w_1}\subseteq \overline{O_{w_2}}.
\]

For any two composition $\lambda,\mu\in\Comp(n)$, the orbits in $\mathscr F_\mu\times \mathscr F_\lambda$ with respect to the diagonal action of $G_n$ are parametrized by double cosets.
Moreover, we have two stratifications, the first one is compatible with the $G_n$-action, the second one is a stratification by affine spaces:
\begin{equation}\label{SchubertCells}
\mathscr F_\mu\times \mathscr F_\lambda
 = \bigsqcup_{w\in \doubleS{\mu}{\lambda}} \Omega_w
 = \bigsqcup_{w\in \doubleS{\mu}{\lambda}} \left(\bigsqcup_{\substack{(w_1,w_2)\in \fk{S}^\lambda\times \fk{S}^\mu\\ w_1^{-1}w_2\in \fk{S}_\lambda w \fk{S}_\mu}} O_{w_1,w_2}\right).
\end{equation}
where $\Omega_w = G_n.(F_e,F_{w})$, and $O_{w_1,w_2} = O^{\mu\lambda}_{w_1,w_2} = \Omega_w\cap (O_{w_1}\times \mathscr{F}_\lambda)$.
Note that each strata $O_{w_1,w_2}$ contains exactly one $T_n$-fixed point $(F_{w_1},F_{w_2})\in O_{w_1,w_2}$.

For later use, we denote $P_{\mu\lambda}^w := P_\mu \cap w.P_\lambda = \ona{Stab}_{G_n}(F_e,F_{w})$.
It is clear that $P_{\mu\lambda}^w$ retracts to the reductive group $G_{\mu\lambda}^w:=G_\mu\cap w.G_\lambda$, whose Weyl group is given by $\frakS_\mu \cap w.\frakS_\lambda$.

We will write $H^*_{T_n}(pt)=\bfk[x_1,\ldots, x_n]$, where $x_i$ is the first Chern class of the line bundle $\bbC e_i$, and $\deg x_i=2$.
We will use $x_i$ and $\bbC e_i$ interchangeably.
In accordance with \cref{GvsTW}, we have
\[
H_{G_n}^*(pt) = \bfk[x_1,\ldots, x_n]^{\fk{S}_n},\qquad H_{G_n}^*(\mathscr{F}_\lambda) = H_{P_\lambda}^*(pt) = \bfk[x_1,\ldots, x_n]^{\fk{S}_\lambda}.
\]
The Euler classes of tangent spaces at $T_n$-fixed points are expressed by the following formulae:
\begin{gather*}
e(T_{F_w}\mathscr{F}_\gamma) = \prod_{(i,j)\in wN_\lambda} (x_j-x_i),\qquad\quad 
e\left( T_{(F_{w_1},F_{w_2})} G_n.(F_{w_1},F_{w_2})\right) = \prod_{(i,j)\in w_1N_\mu\cup w_2N_\lambda}(x_j-x_i).
\end{gather*}

\subsection{Torsion sheaves on a smooth curve}
Let $C$ be a smooth projective curve over $\bbC$, and denote by $\Ocal=\Ocal_{C}$ its structure sheaf.
Let $\T = \ona{Tor}C$ be the moduli stack of torsion sheaves on $C$.
It has a decomposition into connected components
\[
\T = \bigsqcup_{n\in\bbZ_{\geqslant 0}} \T_n,
\]
where $\T_n$ stands for the moduli stack of torsion sheaves of degree $n$.

The stack $\T_n$ possesses an explicit presentation as a quotient.
Namely, let ${\cal Quot}_{n}(\bbC^n\otimes \Ocal)$ be the $\cal Quot$-scheme for constant Hilbert polynomial $P_{\cal E}=n$.
Recall (see~\cite{potierlectures} for details) that its $\bbC$-points are given by quotients $\phi:\bbC^n\otimes \Ocal\twoheadrightarrow \cal E$, where $\cal E$ is a torsion sheaf of degree $n$.
This $\cal Quot$-scheme is smooth, and its tangent space at $\phi$ is
\begin{equation}\label{tanQuot}
	T_\phi {\cal Quot}_{n}(\bbC^n\otimes \Ocal) \simeq \Hom(\Ker\phi,\cal E).
\end{equation}
Moreover, $\cal Quot$-scheme has a natural $G_n$-action by automorphisms of $\bbC^n\otimes \Ocal$.
Define $Q_n$ as its open subscheme, consisting of quotients which induce isomorphism on global sections:
\[
Q_n= \left\{ \phi:\bbC^n\otimes  \Ocal\twoheadrightarrow \cal E \mid H^0(\phi)\text{ is an isomorphism} \right\} \subset {\cal Quot}_{n}(\bbC^n\otimes \Ocal).
\]
Note that $Q_n$ inherits $G_n$-action.

\begin{lm}[{\cite{potierlectures}}]\label{quot}
We have an isomorphism of stacks $[Q_n/G_n]\simeq \T_n$.
\end{lm}
In particular, each $\T_n$ is smooth, since the lemma above provides it with a smooth atlas.

When $n=1$, we have isomorphisms $\cal Quot_1(\Ocal)\simeq Q_1\simeq C$, and the action of $G_1\simeq \bb G_m$ is trivial.
In view of this, denote by $p_{ij}:Q_1\times Q_1\times C\simeq C\times C\times C\to C\times C$ the natural projections (as in \cref{convalg}).

\begin{lm}[{\cite[Lemma 3.2]{Min_CHAH2020}}]
Let $\cal K, \cal E\in \ona{Coh}(Q_1\times C)\simeq \ona{Coh}(C\times C)$ be the universal families of kernels and images of quotients $\Ocal\ra \cal E$ respectively.
Then
\begin{gather*}
	p_{12*}\cal Hom(p_{13}^* \cal K,p_{23}^* \cal E)\simeq \frac{x_2}{x_1}\Ocal_{C\times C}(\Delta),
\end{gather*}
where $\Delta\subset C\times C$ is the diagonal.
\end{lm}

\begin{nota}
When working with $Q_1^n$, we will write $\cal K_i = p^*_{i,n+1}\cal K$, $\cal E_i = p^*_{i,n+1}\cal E$, and further denote the sheaf $p_{1\ldots n*}\cal Hom(\cal K_i,\cal E_j)$ of global sections along $C$ by $\Hom(\cal K_i,\cal E_j)$.
\end{nota}

By~\cite[Lemma 3.1]{Min_CHAH2020}, we have an identification
\[Q_n^{T_n}=(Q_1)^n=C^n.\]
The normal bundle to the fixed point set $N_{Q_1^n}Q_n$ is given by the following formula:
\begin{align}\label{QuotVBun}
N_{Q_1^n}Q_n = (TQ_n)|_{Q^n_1}/TQ^n_1 = \bigoplus_{i\neq j} \Hom(\cal K_i,\cal E_j) = \bigoplus_{i\neq j} \frac{x_j}{x_i}\Ocal(\Delta_{ij}),
\end{align}
where $\Delta_{ij}\subset C^n$ is the preimage of $\Delta$ under the natural projection $p_{ij}:C^n\to C^2$.

For any $I\subset [1,n]$, write $V_I = \bigoplus_{i\in I}\bbC e_i$.
Let $S\in 2^{[1,n]}$ be a collection of subsets of $[1,n]$.
Let $\overline{S}\in 2^{[1,n]}$ be the smallest collection which contains $S$ and is stable under taking intersections and complements.
It gives rise to a disjoint union $[1,n] = \bigsqcup_j I_j$, with subsets $I_j$ being subsets in $\overline{S}$, which are minimal under inclusion.
Consider the subset $\widetilde{Q}_S\subset Q_n$ consisting of quotients $\phi: \bbC^n\twoheadrightarrow\cal E$, such that
\begin{equation}
H^0(\phi)|_{V_I\otimes \Ocal}: V_I\to H^0(\Im \phi|_{V_I\otimes \Ocal})
\end{equation}
is an isomorphism for any $I\in S$.
Further, we denote $Q_S:=\prod_j Q_{V_{I_j}\otimes \Ocal}$.
\begin{lm}\label{lm:flags-smooth}
$\widetilde{Q}_S$ is a smooth closed subvariety of $Q_n$.
More specifically, $\widetilde{Q}_S$ is a vector bundle over $Q_S$.
\end{lm}
\begin{proof}
If $S=\{I\}$ consists of one subset, then $\widetilde{Q}_S$ is closed in $Q_n$ by~\cite[Proposition 1.8]{Min_CHAH2020}.
For general $S$, $\widetilde{Q}_S$ is closed as an intersection of closed subvarieties.

Let $I\subset [1,n]$, and $J$ its complement.
Consider an action of $\bbC^*$ on $\bbC^n$, which has weight $1$ on $V_I$ and is trivial on $V_J$.
This induces an action $a_I$ of $\bbC^*$ on $Q_n$.
Moreover, analogously to~\cite[Lemma 3.1]{Min_CHAH2020} we have $(Q_n)^{\bbC^*} = Q_{V_I\otimes \Ocal}\times Q_{V_J\otimes \Ocal}$, and the corresponding attracting set is $\widetilde{Q}_{\{I\}}$.

For a general $S$, consider an action of torus $T_S = \prod_{I\in S} (\bb G_m)_I$ on $Q_d$, where for each $I$ the action of $(\bb G_m)_I$ is given by $a_I$. 
Taking intersections, we see that the fixed point set of this action is $Q_S$, and the attracting set is $\widetilde{Q}_S$.
Bia{\l}ynicki-Birula theorem~\cite{Bia_TAAG1973a} then implies that $\widetilde{Q}_S$ is a vector over $Q_S$, and as such is smooth.
\end{proof}

Let $\lambda=(\lambda_1,\ldots,\lambda_k)$ be a composition of $n$.
Consider the stack of flags of torsion sheaves of type $\lambda$:
\[
\FT_\lambda=\left\{ 0=\cal E_0\subset \cal E_1\subset \ldots \subset \cal E_k : \deg \left(\cal E_i/\cal E_{i-1}\right)=\lambda_i \right\}.
\]
The stack $\FT_\lambda$ has a quotient presentation analogous to \cref{quot}:
\[
\FT_\lambda\simeq[\widetilde{Q}_\lambda/P_\lambda].
\]
Here, $\widetilde{Q}_\lambda = \widetilde{Q}_S$ for $S = \{[1,\tilde{\lambda}_i]\}_{i=1}^{k}$; we will also call this collection of intervals $\lambda$ by abuse of notation.
Analogously to~\eqref{QuotVBun}, we have
\begin{equation}\label{FlagQuotVBun}
N_{Q_1^n}\widetilde{Q}_\lambda = (T\widetilde{Q}_\lambda)|_{Q^n_1}/TQ^n_1 =  \bigoplus_{\substack{(i,j)\in [1,n]^2\setminus N_\lambda\\i\neq j}} \Hom(\cal K_i,\cal E_j) = \bigoplus_{(i,j)\in I_\lambda} \frac{x_i}{x_j}\Ocal(\Delta_{ij}).
\end{equation}

We have maps
\begin{center}
\begin{tabular}{rc}
$
\begin{tikzcd}
& \FT_\lambda\ar[dl,"q_\lambda"']\ar[dr,"p_\lambda"] &\\
\T_\lambda & & \T_d
\end{tikzcd}
$&
\begin{tabular}{l}\vspace{0.5ex}
$q_\lambda(\cal E_1\subset \cal E_2\subset\ldots \subset\cal E_k)=(\cal E_1,\cal E_2/\cal E_1,\ldots, \cal E_k/\cal E_{k-1}),$\\
$p_\lambda(\cal E_1\subset \cal E_2\subset\ldots \subset\cal E_k)=\cal E_k,$\\
\end{tabular}
\end{tabular}
\end{center}
where $\T_\lambda\coloneqq \T_{\lambda_1}\times \ldots \times \T_{\lambda_k}$.
Note that the map $q_\lambda$ is not representable, since it is not faithful on automorphism groups of points.
However, it is a stack vector bundle, that is it comes from a two-term complex of vector bundles on the base, see~\cite[Corollary 3.2]{GHS_MMCH2011}.
For instance, when $k=2$, this complex is $R\Hom_C(\cal E_1,\cal E_2)[1]$, where $\cal E_i$ is the universal sheaf on $\T_{\lambda_i}\times C$.
In particular, pulling back along $q_\lambda$ induces an isomorphism $H^*(\FT_\lambda)\simeq H^*(\T_\lambda)$.

On the other hand, $p_\lambda$ is induced by the embedding $\widetilde{Q}_\lambda\hookrightarrow Q_n$:
\[
p_\lambda:\FT_\lambda \simeq[\widetilde{Q}_\lambda/P_\lambda] \simeq [G_n\times_{P_\lambda}\widetilde{Q}_\lambda/G_n]\to [Q_n/G_n]\simeq \T_n.
\]
In particular, it is a projective morphism.

Denote $Y_\lambda=G_n\times_{P_\lambda}\widetilde{Q}_\lambda$.
Then $(Y_\lambda)^{T_n}=\fk{S}_n/\fk{S}_\lambda \times C^n$, and the projection $(Y_\lambda)^{T_n}\to Q_n^{T_n}$ gets identified with the projection $\fk{S}_n/\fk{S}_\lambda \times C^n\to C^n$.

Let us denote $\bf{P}_n = \bf{P}_n(C) = H^*(C^n)[x_1,\ldots,x_n]$, where $\deg x_i = 2$.
\begin{prop}[{\cite[Theorem 1]{Hei_CMSC2012}}]\label{HofTor}
We have $H^*(\T_n)\simeq \bf{P}_n^{\frakS_n}$.
\end{prop}

Since $q_\lambda$ is a stack vector bundle, we also have
\begin{equation}\label{FlagCoh}
	H^*(\FT_\lambda)\simeq H^*(\T_\lambda) \simeq \bigotimes_i \bf{P}_{\lambda_i}^{\fk{S}_{\lambda_i}}=\bf{P}_n^{\fk{S}_\lambda}.
\end{equation}

\begin{lm}\label{isovialoc}
Let $co$ be the composition $H^*(\T_n)\simeq H^*_{G_n}(Q_n)\hookrightarrow H^*_{T_n}(Q_n) \xra{i_{Q_n}}H^*_{T_n}(C^n)\simeq H^*(\T_1^n)$.
Then the following square commutes:
\[\begin{tikzcd}
H^*(\T_n) \ar[r,"\sim"]\ar[d,hook,"co"] & \bf{P}_n^{\frakS_n}\ar[d,hook]\\
H^*(\T_1^n) \ar[r,equal] & \bf{P}_n
\end{tikzcd}\]
\end{lm}
\begin{proof}
Consider the following diagram, where the vertical arrows are given by restriction to $T_n$-fixed points, and $a:\frakS_n\times C^n \to C^n$ is the natural action:
\[\begin{tikzcd}
H^*_{G_n}(Q_n) \ar[r]\ar[d,"co"] & H^*_{G_n}(Y_{1^n}) \ar[d] & H^*_{B_n}(\widetilde{Q}_{1^n}) \ar[r,"\sim"]\ar[d]\ar[l,"\sim"] & H^*_{T_n}(Q_{1^n})\ar[d,equal]\\
H^*_{T_n}(C^n) \ar[r,"a^*"]& H^*_{T_n}(\frakS_n\times C^n) & H^*_{T_n}(C^n)\ar[l,"a^*"']\ar[r,equal] & H^*_{T_n}(C^n)
\end{tikzcd}\]
All squares above are obviously commutative, except for the second one, which commutes by~\cite[Lemma A.17]{Min_CHAH2020}.
The proof of \cref{HofTor} in~\cite{Hei_CMSC2012} shows that the composition of upper horizontal maps coincides with the inclusion $H^*(\T_n)\simeq \bf{P}_n^{\frakS_n}\subset \bf{P}_n$.
On the other hand, lower horizontal row can be replaced with the identity map without breaking commutativity.
This concludes the proof.
\end{proof}

\begin{corr}\label{ind-vs-act}
Let $\lambda\in \Comp(n)$.
The following square commutes:
\[
\begin{tikzcd}
H^*(\FT_\lambda) \ar[d,hook,"co"] & H^*(\T_\lambda)\ar[l,"\sim"]\ar[d,hook]\\
H^*_{T_n}(\fk{S}_n/\fk{S}_\lambda \times C^n)  & H^*_{T_n}(C^n)\ar[l,"a^*"]
\end{tikzcd}
\]
where $a: \fk{S}_n/\fk{S}_\lambda \times C^n = \fk{S}^\lambda \times C^n \to C^n$ is the natural action.
\end{corr}
\begin{proof}
Consider the following diagram:
\[\begin{tikzcd}
H^*_{G_n}(Y_\lambda) \ar[d] & H^*_{P_\lambda}(\widetilde{Q}_\lambda)\ar[d]\ar[l,"\sim"']\ar[r,dash,"\sim"] & H^*_{G_\lambda}(Q_\lambda)\ar[d,hook,"co"]\ar[r,dash,"\sim"] & \bf{P}_n^{\frakS_\lambda}\ar[d,hook]\\
H^*_{T_n}(\fk{S}_n/\fk{S}_\lambda \times C^n) & H^*_{T_n}(C^n)\ar[r,equal]\ar[l,"a^*"] & H^*_{T_n}(C^n)\ar[r,dash,"\sim"] & \bf{P}_n
\end{tikzcd}\]
The first square commutes by~\cite[Lemma A.17]{Min_CHAH2020}, and the last one does by \cref{isovialoc}.
We are done.
\end{proof}

\begin{rmq}\label{rmk:purity}
Recall that $\T_1^n \simeq C^n\times \mathrm{B}T_n$.
In particular, its cohomology has pure weight filtration, and therefore so do $H^*(\T_n)$ and $H^*(\FT_\lambda)$.
\end{rmq}

\subsection{Steinberg varieties}
\begin{defi}
The fiber product
\[
Z_{\mu,\lambda}:=Y_\mu\times_{Q_n} Y_\lambda\subset Y_\mu\times Y_\lambda
\]
is called the \textit{partial Steinberg variety} of type $(\mu,\lambda)$.
\end{defi}
For each $\lambda\in \Comp(n)$, there is a natural map
$$
Y_\lambda=G_n\times_{P_\lambda}\widetilde{Q}_\lambda\to \mathscr{F}_\lambda=G_n/P_\lambda,\qquad (g,q)\mapsto gP_\lambda.
$$
Then the ambient variety $Y_\mu\times Y_\lambda$ comes equipped with a projection to the product of partial flag varieties:
\[
f_{\mu,\lambda}:Y_\mu\times Y_\lambda\to \mathscr{F}_\mu\times \mathscr{F}_\lambda.
\]
Stratification~\eqref{SchubertCells} induces the following stratification on $Z_{\mu,\lambda}$:
\[
Z_{\mu,\lambda}=\bigsqcup_{w\in \doubleS{\mu}{\lambda}} Z_{\mu,\lambda}^w:=\bigsqcup_{w\in \doubleS{\mu}{\lambda}} f_{\mu,\lambda}^{-1}(\Omega_w).
\]

Let us compute the fiber $f^{-1}(F_e,F_{w})$.
By definition, we have 
\[
f^{-1}(F_e,F_w) = \widetilde{Q}_\mu\cap w.\widetilde{Q}_\lambda \subset Q_n,
\]
where elements of $\fk{S}_d$ are identified with permutation matrices in $G_n$.
This subvariety is smooth by \cref{lm:flags-smooth}; therefore, each strata $Z_{\mu,\lambda}^w$ is smooth as well.

\begin{lm}\label{KLRLocInj}
Let $\bbk$ be a field of characteristic $0$.
Homology groups $H_*^{G_n}(Z_{\mu,\lambda},\bbk)$ are torsion-free as $H_{G_n}$-modules.
\end{lm}
\begin{proof}
Without loss of generality, we can assume that $\bbk = \bbQ$.
For each stratum $Z_{\mu,\lambda}^w$, we have
\[
H_*^{G_n}(Z_{\mu,\lambda}^w) = H_*^{G_n}\left( G_n \times_{P_{\mu\lambda}^w} \left( \widetilde{Q}_\mu\cap w.\widetilde{Q}_\lambda \right) \right) = H_*^{P_{\mu\lambda}^w} \left( \widetilde{Q}_\mu\cap w.\widetilde{Q}_\lambda \right) = H_*^{G_{\mu\lambda}^w}(Q_{\mu\cup w.\lambda}) = H_*(\T_{\mu\cup w.\lambda}).
\]
Here, we have used \cref{lm:flags-smooth}.
The homology $H_*(\T_{S})$ sits inside $\bf{P}_n$ by \cref{HofTor}, and is therefore a torsion-free $H_{G_n}$-module.
Moreover, it has pure weight filtration by \cref{rmk:purity}.

Let us choose a total order $\prec$ on $\doubleS{\mu}{\lambda}$ compatible with the orbit closure order, and define
\[
Z_{\mu,\lambda}^{\prec w} = \bigsqcup_{w'\prec w}Z_{\mu,\lambda}^{w'},\quad Z_{\mu,\lambda}^{\preccurlyeq w} = Z_{\mu,\lambda}^{\prec w} \sqcup Z_{\mu,\lambda}^w.
\]
Since each strata $Z_{\mu,\lambda}^w$ has pure Borel-Moore homology, the associated open-closed long exact sequences split into short exact sequences:
\[
0\to H_*^{G_n}(Z_{\mu,\lambda}^{\prec w}) \to H_*^{G_n}(Z_{\mu,\lambda}^{\preccurlyeq w}) \to H_*^{G_n}(Z_{\mu,\lambda}^{w}) \to 0,
\]
see e.g.~\cite[Lemma 4.9]{Min_CHAH2020}.
In particular, $H_*^{G_n}(Z_{\mu,\lambda})$ has a filtration with associated graded $\bigoplus_w H_*^{G_n}(Z_{\mu,\lambda}^w)$.
Since $H_*^{G_n}(Z_{\mu,\lambda}^w)$ is a torsion-free $H_{G_n}$-module for all $w$, the same holds for $H_*^{G_n}(Z_{\mu,\lambda})$.
\end{proof}
\section{Schur algebra of a smooth curve}\label{sec:schur}
In this section, as well as \cref{sec:KLRcurve}, we assume that $\bbk$ is a field of characteristic $0$.
\subsection{Schur algebras of curves}
Let $Y_n=\bigsqcup_{\lambda}Y_\lambda$, and consider the projection $\pi:Y_n\to Q_n$.
Denote $Z_n=Y_n\times_{Q_n} Y_n$; we have decomposition into connected components $Z_n = \bigsqcup_{\mu,\lambda} Z_{\mu,\lambda}$.
Let us apply the general construction from \cref{ConvSetup} to $\pi$.

\begin{defi}
The \textit{(torsion) Schur algebra} of $C$, denoted by $\Sc_n=\Sc_n^C$, is the convolution algebra $\cal A(\pi)=H^*_{G_n}(Z_n)$.
\end{defi}

One can easily check that the map $\pi$ is small, so that $\dim Z_{\mu,\lambda} = \dim Y_\lambda = \dim Q_n = n^2$ for any $\lambda$, $\mu$.
Thanks to our conventions in \cref{subs:Flag}, $\Sc_n^C$ is a graded algebra.

\begin{rmq}
Since we're only concerned with torsion sheaves in this article, we will omit the qualifier ``torsion'' from now on.
We still mention it in the definition, because one would like to eventually consider a similar algebra for coherent sheaves of positive rank.
\end{rmq}

We denote $\Sc_{\mu,\lambda} = H^*_{G_n}(Z_{\mu,\lambda})$.
By definition $\Sc_n=\bigoplus_{\mu,\lambda}\Sc_{\mu,\lambda}$, and the product in $\Sc_n$ decomposes into a direct sum of maps $\Sc_{\nu,\mu}\otimes \Sc_{\mu,\lambda}\to \Sc_{\nu,\lambda}$.
Moreover, $\Sc_n$ acts on the space
\begin{equation}
P_n = \bigoplus_{\lambda} H^*_{G_n}(Y_\lambda) = \bigoplus_{\lambda} \bf{P}_n^{\fk{S}_\lambda}
\end{equation}
by \cref{ConvModule}.
We call $P_n$ the \textit{polynomial representation} of $\Sc_n$.
This action decomposes into a direct sum of maps $\Sc_{\mu,\lambda}\otimes H^*_{G_n}(Y_\lambda)\to H^*_{G_n}(Y_\mu)$.

\subsection{Localized Schur algebra}
Let us apply equivariant localization to the algebra $\Sc_n$.
Under identifications in \cref{sec:TorSh}, the restriction of $\pi$ to $T_n$-fixed points $(Y_\lambda)^{T_n}\to (Q_n)^{T_n}$ equals to the projection $\fk S_n/\fk{S}_\lambda \times C^n\to C^n$; denote it $\pi_T$.
Let us also denote by $a_\lambda: \frakS_n/\frakS_{\lambda}\times C^n = \frakS^\lambda\times C^n\to C^n$ the natural map induced by action.
We also have
\[
(Z_n)^{T_n} = \bigsqcup_{\mu,\lambda} (Z_{\mu,\lambda})^{T_n} = \bigsqcup_{\mu,\lambda} \fk S_n/\fk{S}_\mu\times \fk S_n/\fk{S}_\lambda\times C^n.
\]
For each connected component of $(Y_\lambda)^{T_n}$, its normal bundle splits into the direct sum of $N_{C^n}\widetilde{Q}_\lambda$ and the tangent space to $\mathscr{F}_\lambda$.
We identify $(Y_\lambda)^{T_n}\simeq \coprod_{w\in \frakS_n/\frakS_\lambda} \{w\}\times C^n$ as above. Denote  by $\gamma_\lambda\in H^*_{T_n}(C^n)$ the Euler class of the normal bundle to $\{\Id\}\times C^n\subset Y_\lambda$.

The formula~\eqref{FlagQuotVBun} implies that $e(N_{C^n}\widetilde{Q}_\lambda)=\prod_{(i,j)\in I_\lambda} (x_i-x_j+\Delta_{ij})$.
Therefore,
\begin{equation}\label{NorBunFlag}
	\gamma_\lambda = e(N_{C^n}\widetilde{Q}_\lambda)e(T_{F_{\Id}}\mathscr{F}_\lambda) = \prod_{(i,j)\in I_\lambda} (x_i-x_j+\Delta_{ij})\prod_{(i,j)\in N_\lambda}(x_j-x_i).
\end{equation}
\cref{ConvLocAlg} provides us with an algebra homomorphism
\[
\Xi_n:\Sc_n\to (\cal A^{T_n}_{\gamma^{-1}})^{\fk{S}_n}=:\Sc_n^{\text{loc}},
\]
where $\gamma = \bigoplus_\lambda a_\lambda^*(\gamma_\lambda)$.
Furthermore, $\Xi_n$ is injective by \cref{KLRLocInj}.
As a vector space, $\Sc_n^{\text{loc}}$ is isomorphic to 
\[
\Sc_n^{\text{loc}} = \bigoplus_{\mu,\lambda} \Sc_{\mu,\lambda}^{\text{loc}} = \bigoplus_{\mu,\lambda}\left(\bfk[\fk{S}_n/\fk{S}_\mu]\otimes\bfk[\fk{S}_n/\fk{S}_\lambda]\otimes \left( H^*(C^n)[x_1,\ldots,x_n]\right)_{\loc}\right)^{\fk S_n}.
\]

Let us describe the multiplication in $\Sc_n^{\text{loc}}$ explicitly.
Pulling back along the quotient map $r_{\mu\lambda}:\fk{S}_n^2\to \fk S_n/\fk{S}_\mu\times \fk S_n/\fk{S}_\lambda$, we obtain the natural inclusion
\begin{equation}\label{SchurDesym}
\Sc_{\mu,\lambda}^{\text{loc}}\subset \left(\bfk[\fk S_n]^{\otimes 2}\otimes \left( H^*(C^n)[x_1,\ldots,x_n]\right)_{\loc}\right)^{\fk S_n}.
\end{equation}
We use the notations from \cref{ConvEx} with $X=C^n$ and $\Gamma = \fk{S}_n$ for the right-hand side.
The image of the inclusion above can be obtained by partial symmetrization.
Namely, for any $g\in\fk{S}_n$ consider the element $(\fk{S}_\mu, g\fk{S}_\lambda)\in \fk{S}_n/\fk{S}_\mu\times \fk{S}_n/\fk{S}_\lambda$.
Denote its stabilizer under the diagonal $\fk{S}_n$-action by $\Gamma_g^{\mu\lambda}:=\fk{S}_\mu\cap g\fk{S}_\lambda g^{-1}$.
Note that $\Gamma_g^{\mu\lambda}\simeq \fk{S}_{\nu}$ for some $\nu \in\Comp(n)$, and depends only on the image of $g$ in $\fk{S}_\mu\backslash \fk{S}_n/\fk{S}_\lambda$.
We have an inclusion
\[
\Gamma_g^{\mu\lambda}\hookrightarrow \fk{S}_\mu\times \fk{S}_\lambda,\qquad f\mapsto (f,g^{-1}fg).
\]

The algebra $\Sc_{\mu,\lambda}^{\text{loc}}$ is spanned by the elements
\[
\xi_{(g,x)}^{\mu\lambda}=\sum_{h\in \frakS_n/\Gamma_{g}^{\mu\lambda}}h\frakS_\mu\otimes hg\frakS_\lambda\otimes x^h,
\]
where $g\in \frakS_n$ and $x\in\left(H^*(C^n)[x_1,\ldots,x_n]\right)_{\loc}^{\Gamma_g^{\mu\lambda}}$.
Moreover, it is clear that $\xi_{(gg',x)}^{\mu\lambda}=\xi_{(g,x)}^{\mu\lambda}$ for $g'\in \frakS_\lambda$ and $\xi_{(g'g,g'x)}^{\mu\lambda}=\xi_{(g,x)}^{\mu\lambda}$ for $g'\in \frakS_\mu$.
Thus it is enough to consider $\xi_{(g,x)}^{\mu\lambda}$ with $g\in \doubleS{\mu}{\lambda}$.
Let us compute its image under the inclusion~\eqref{SchurDesym}:
\begin{align*}
\xi_{(g,x)}^{\mu\lambda} 
& 	= \sum_{h\in \frakS_n/\Gamma_{g}^{\mu\lambda}}h\frakS_\mu\otimes hg\frakS_\lambda\otimes x^h 
	\mapsto \frac{1}{|\Gamma_{g}^{\mu\lambda}|}\sum_{(h,h_1,h_2)\in \frakS_n\times \frakS_\mu\times\frakS_\lambda}hh_1\otimes hgh_2\otimes x^h\\
& 	= \frac{1}{|\Gamma_{g}^{\mu\lambda}|}\sum_{(h_1,h_2)\in \frakS_\mu\times\frakS_\lambda}\xi_{(h_1^{-1}gh_2,x^{h_1^{-1}})} 
	= \sum_{(h_1,h_2)\in (\fk{S}_\mu\times \fk{S}_\lambda)/\Gamma_g^{\mu\lambda}} \xi_{(h_1gh_2,x^{h_1})}.
\end{align*}
In what follows, we will abuse the notations and identify $\xi_{(g,x)}^{\lambda\mu}$ with its image under~\eqref{SchurDesym}.

Consider the following commutative diagram:
\[
\begin{tikzcd}
\fk{S}_n^2\times\fk{S}_n^2\ar[d,"r"] & \fk{S}_n^3\ar[l,"p"']\ar[r,"q"]\ar[d,"r"] & \fk{S}_n^2\ar[d,"r"]\\
(\fk{S}_n/\fk{S}_\lambda\times\fk{S}_n/\fk{S}_\mu)\times (\fk{S}_n/\fk{S}_\mu\times\fk{S}_n/\fk{S}_\nu) & \fk{S}_n/\fk{S}_\lambda\times\fk{S}_n/\fk{S}_\mu \times\fk{S}_n/\fk{S}_\nu\ar[l,"p"']\ar[r,"q"] & \fk{S}_n/\fk{S}_\lambda \times\fk{S}_n/\fk{S}_\nu
\end{tikzcd}
\]
We have $p^*r^* = r^*p^*$, and $q_*r^* =|\fk{S}_\mu| r^*q_*$.
Therefore, up to the factor $|\fk{S}_\mu|$, pullbacks along $r_{\lambda\mu}$ fit in the following commutative square, where horizontal maps are given by multiplication in the corresponding algebra:
\[
\begin{tikzcd}
\Sc_{\lambda,\mu}^{\text{loc}}\otimes \Sc_{\mu,\nu}^{\text{loc}}\ar[r]\ar[d,"r_{\lambda\mu}\otimes r_{\mu\nu}"] &\Sc_{\lambda,\nu}^{\text{loc}}\ar[d,"r_{\lambda\nu}"]\\
\cal A_{(\gamma_\mu^{-1})^\circ}^{\fk S_n}\otimes \cal A_{(\gamma_\mu^{-1})^\circ}^{\fk S_n}\ar[r]& \cal A_{(\gamma_\mu^{-1})^\circ}^{\fk S_n}
\end{tikzcd}
\]
In particular, using~\eqref{eq:mult-xi} we see that 
\begin{equation}\label{SchurLocComp}
\begin{aligned}
	|\Gamma_g^{\lambda\mu}||\Gamma_h^{\mu\nu}|\xi_{(g,x)}^{\lambda\mu}*\xi_{(h,y)}^{\mu\nu} & = \frac{1}{|\fk{S}_\mu|}\left( \sum_{(a_1,a_2)\in \fk{S}_\lambda\times \fk{S}_\mu} \xi_{(a_1ga_2,x^{a_1})} \right)*\left( \sum_{(b_1,b_2)\in \fk{S}_\mu\times \fk{S}_\nu} \xi_{(b_1hb_2,y^{b_1})} \right)\\
	& = \sum_{(a,b,c)\in \fk{S}_\lambda\times \fk{S}_\mu\times \fk{S}_\nu} \xi_{(agbhc,x^a(y\gamma^{-1}_\mu)^{agb})} \\
	& = \sum_{b\in \fk{S}_\mu} |\Gamma_{gbh}^{\lambda\nu}|\xi_{(gbh,x(y^b\gamma^{-1}_\mu)^g)}^{\lambda\nu}.
\end{aligned}
\end{equation}

\subsection{Generators}
Let us introduce some elements in $\Sc_n$, and compute their images under localization.

\paragraph{\textbf{Polynomials}}
Let $\lambda\in\Comp(n)$.
\cref{DiagEmb} provides us with a homomorphism of algebras $\delta_\lambda: H_{G_n}^*(Y_\lambda)\to \Sc_{\lambda,\lambda}\subset \Sc_n$.
For any $P\in\bf{P}_n^{\fk{S}_\lambda}\simeq H^*_{G_n}(Y_\lambda)$, we will identify $P$ with its image under $\delta_\lambda$ by abuse of notation.

\begin{lm}\label{TautElts}
For any $P\in \bf{P}_n^{\fk{S}_\lambda}$, we have $\Xi_n(P)=\xi^{\lambda\lambda}_{(1,\gamma_\lambda P)}$.
\end{lm}
\begin{proof}
We have inclusions of fixed point sets:
\[
\begin{tikzcd}
	\fk{S}_n/\fk{S}_\lambda\times C^n\ar[r,hook,"\Delta"]\ar[d,hook,"i_{Y_\lambda}"] & (\fk{S}_n/\fk{S}_\lambda)^2\times C^n\ar[d,hook,"i_{Z_{\lambda,\lambda}}"]\\
	Y_\lambda\ar[r,hook] & Z_{\lambda,\lambda}
\end{tikzcd}
\]
where the upper horizontal arrow is given by the diagonal embedding $\fk{S}_n/\fk{S}_\lambda\to (\fk{S}_n/\fk{S}_\lambda)^2$.
Let $p_1,p_2:(\fk{S}_n/\fk{S}_\lambda)^2\times C^n\to \fk{S}_n/\fk{S}_\lambda \times C^n$ be the two natural projections.
Applying \cref{LocPullPush}, we get:
\begin{align*}
\Xi_n(P) & = \left( p_1^*a_\lambda^*(\gamma_\lambda)\cdot p_2^*a_\lambda^*(\gamma_\lambda) \right) \cdot \Delta_*\left(a_\lambda^*(\gamma_\lambda)^{-1}\cdot i^*_{Y_\lambda}(P)\right) \\
& = \Delta_*\left( a_\lambda^*(\gamma_\lambda)\cdot a_\lambda^*(P) \right)
= \xi^{\lambda\lambda}_{(1,\gamma_\lambda P)},
\end{align*}
where we have used \cref{ind-vs-act} to replace $i_{Y_\lambda}^*$ by $a_\lambda^*$.
\end{proof}

\paragraph{\textbf{Splits and merges}}
For $\lambda,\lambda'\in\Comp(n)$, we say that $\lambda'$ \textit{subdivides} $\lambda$ if $\fk{S}_{\lambda'}\subset \fk{S}_{\lambda}$, and write $\lambda'\subset\lambda$.
In other words, $\lambda'$ is obtained from $\lambda$ by replacing each $\lambda_i$ with $\lambda_i^{(1)},\ldots,\lambda_i^{(k_i)}$, which sum up to $\lambda_i$.
For such pair of compositions, the closed embedding $\widetilde{Q}_{\lambda'}\subset \widetilde{Q}_{\lambda}$ induces a proper map $Y_{\lambda'}\to Y_{\lambda}$.
We therefore have closed embeddings
\begin{gather*}
\delta_{\lambda',\lambda}:Y_{\lambda'}=Y_{\lambda'}\times_{Y_\lambda}Y_\lambda\hookrightarrow Y_{\lambda'}\times_{Q_n} Y_{\lambda} = Z_{\lambda',\lambda},\\
\delta_{\lambda,\lambda'}:Y_{\lambda'}=Y_\lambda\times_{Y_\lambda}Y_{\lambda'}\hookrightarrow Y_{\lambda}\times_{Q_n} Y_{\lambda'} = Z_{\lambda,\lambda'}.	
\end{gather*}
In particular, let us define
\[
\rmS_\lambda^{\lambda'} = (\delta_{\lambda',\lambda})_*[Y_{\lambda'}]\in \Sc_{\lambda',\lambda},\qquad \rmM^\lambda_{\lambda'} = (\delta_{\lambda,\lambda'})_*[Y_{\lambda'}]\in \Sc_{\lambda,\lambda'}.
\]

\begin{defi}
	We call $\rmS_\lambda^{\lambda'}$ \textit{split}, and $\rmM^\lambda_{\lambda'}$ \textit{merge}.
\end{defi}
\begin{lm}\label{SMLoc}
	The images of splits and merges under localization are given by
	\[
	\Xi_n(\rmS_\lambda^{\lambda'}) = \xi^{\lambda'\lambda}_{(1,\gamma_\lambda)},\qquad \Xi_n(\rmM^\lambda_{\lambda'}) = \xi^{\lambda\lambda'}_{(1,\gamma_\lambda)}.
	\]
\end{lm}
\begin{proof}
Consider the inclusions of fixed points sets:
\[
\begin{tikzcd}
	\fk{S}_n/\fk{S}_{\lambda'}\times C^n\ar[r,hook,"\Delta"]\ar[d,hook,"i_{Y_{\lambda'}}"] & \fk{S}_n/\fk{S}_{\lambda'}\times\fk{S}_n/\fk{S}_\lambda\times C^n\ar[d,hook,"i_{Z_{\lambda',\lambda}}"]\\
	Y_{\lambda'}\ar[r,hook,"\delta_{\lambda',\lambda}"] & Z_{\lambda',\lambda}
\end{tikzcd}
\]
where the upper horizontal arrow sends $(g\fk{S}_{\lambda'},x)$ to $(g\fk{S}_{\lambda'},g\fk{S}_{\lambda},x)$.
Let $p_\natural:\fk{S}_n/\fk{S}_{\lambda'}\times\fk{S}_n/\fk{S}_\lambda\times C^n\to \fk{S}_n/\fk{S}_\natural \times C^n$, $\natural\in \{\lambda,\lambda'\}$ be the two natural projections.
Applying \cref{LocPullPush}, we get:
\begin{align*}
	\Xi_n(\rmS_\lambda^{\lambda'}) & = \left( p_{\lambda'}^*a_{\lambda'}^*(\gamma_{\lambda'})\cdot p_{\lambda}^*a_{\lambda}^*(\gamma_{\lambda}) \right) \cdot \Delta_*\left(a_{\lambda'}^*(\gamma_{\lambda'})^{-1}\cdot i^*_{Y_{\lambda'}}[Y_{\lambda'}]\right) \\
	& = \Delta_*\left( a_{\lambda'}^*(\gamma_{\lambda'})\cdot a_{\lambda'}^*[Y_{\lambda'}] \right)
 	= \xi^{\lambda'\lambda}_{(1,\gamma_\lambda)},
\end{align*}
where we have used \cref{ind-vs-act} to replace $i_{Y_{\lambda'}}^*$ by $a_{\lambda'}^*$.
The expression for $\Xi_n(\rmM^\lambda_{\lambda'})$ is obtained in an analogous fashion.
\end{proof}

\begin{lm}\label{SMAssoc}
Let $\lambda,\lambda',\lambda''\in\Comp(n)$ such that $\lambda''\subset \lambda'\subset \lambda$.
Then $\rmS_{\lambda'}^{\lambda''}\rmS_\lambda^{\lambda'} = \rmS_{\lambda}^{\lambda''}$, $\rmM^{\lambda}_{\lambda'}\rmM^{\lambda'}_{\lambda''} = \rmM^{\lambda}_{\lambda''}$.
\end{lm}

\begin{proof}
We will only prove the first equality, second being completely analogous.
Consider the following commutative diagram:
\[
\begin{tikzcd}
Y_{\lambda''}\times Y_{\lambda'}\ar[d,hook,"i"] & Y_{\lambda''}\ar[d,hook,"i"]\ar[l,"\Delta"']\ar[r,equal] & Y_{\lambda''}\ar[d,hook,"i"]\\
(Y_{\lambda''}\times_{Q_n} Y_{\lambda'})\times (Y_{\lambda'}\times_{Q_n} Y_{\lambda})\ar[d,hook] & Y_{\lambda''}\times_{Q_n} Y_{\lambda'}\times_{Q_n} Y_{\lambda}\ar[l,"p'"']\ar[r,"q"]\ar[d,hook] & Y_{\lambda''}\times_{Q_n} Y_{\lambda}\\
Y_{\lambda''}\times Y_{\lambda'}\times Y_{\lambda'}\times Y_{\lambda} & Y_{\lambda''}\times Y_{\lambda'}\times Y_{\lambda}\ar[l,"p"'] &
\end{tikzcd}
\]
We have $\Delta^* = \Delta^!_{p}$ by \cref{GysinProp:b}, and $i_*\Delta^!_{p} = (p')^!_pi_*$ by \cref{GysinProp:b2}.
As a consequence,
\[
\rmS_{\lambda'}^{\lambda''}\rmS_\lambda^{\lambda'} = q_*(p')^!_pi_*\left([Y_{\lambda''}]\boxtimes[Y_{\lambda'}]\right) = i_*\Delta^*\left([Y_{\lambda''}]\boxtimes[Y_{\lambda'}]\right) = i_*[Y_{\lambda''}] = \rmS_{\lambda}^{\lambda''},
\]
and we may conclude.
\end{proof}

\subsection{Diagrammatic presentation of $\Sc_n$}\label{subs:diag-Schur}
Let us identify compositions of operators defined above with certain cord diagrams.
Our strands are allowed to have multiplicities (i.e. non-negative integer labels), and we always read diagrams from bottom to top.

\paragraph{\textbf{Polynomials}}
We depict the polynomial operators as boxes on strands.
Namely, let
\[
P=P_{(1)}\otimes\ldots\otimes P_{(r)}\in \bf{P}_n^{\fk{S}_\lambda} = \bigotimes_i \bf{P}_{\lambda_i}^{\fk{S}_{\lambda_i}}.
\]
Then we draw $P$ as follows:
\[
\tikz[thick,xscale=.25,yscale=.25]{
	\ntxt{-4}{3}{$P=$}
	
	\draw (0,0) -- (0,2);
	\draw (0,4) -- (0,6);
	\opbox{-1.25}{2}{1.25}{4}{$P_{(1)}$}
	\ntxt{0}{-1.5}{$\lambda_1$}
	\ntxt{0}{7.5}{$\lambda_1$}
	
	\ntxt{4.5}{2}{$\ldots$}
	
	\draw (9,0) -- (9,2);
	\draw (9,4) -- (9,6);
	\opbox{7.75}{2}{10.25}{4}{$P_{(r)}$}
	\ntxt{9}{-1.5}{$\lambda_r$}
	\ntxt{9}{7.5}{$\lambda_r$}
	\ntxt{11}{3}{.}
}
\]

\paragraph{\textbf{Splits and merges}}
Take $\lambda = (\lambda_1,\ldots,\lambda_r)\in\Comp(n)$, and let $\lambda' = (\lambda_1,\ldots,\lambda_{k-1},\lambda_k^{(1)},\lambda_k^{(2)},\lambda_{k+1}\ldots, \lambda_r)$ for some $1\leq k\leq r$, where $\lambda_k^{(1)}+\lambda_k^{(2)}=\lambda_k$.
For such pair of compositions, we draw the corresponding split and merge as follows:
\begin{equation}\label{eq:pic-SM}
\begin{aligned}
	\tikz[thick,xscale=.25,yscale=.25]{
		\ntxt{-3}{2}{$\rmS^{\lambda'}_{\lambda}=$}
		
		\draw (0,0) -- (0,5);
		\ntxt{0}{-1.5}{$\lambda_1$}
		\ntxt{0}{6.5}{$\lambda_1$}
		
		\ntxt{4}{2}{$\ldots$}
		
		\draw (9,0) -- (9,1);
		\ntxt{9}{-1.5}{$\lambda_k$}
		\dsplit{6}{1}{12}{5}
		\ntxt{6}{6.5}{$\lambda_k^{(1)}$}
		\ntxt{12}{6.5}{$\lambda_k^{(2)}$}
		
		\ntxt{14}{2}{$\ldots$}
		
		\draw (18,0) -- (18,5);
		\ntxt{18}{-1.5}{$\lambda_r$}
		\ntxt{18}{6.5}{$\lambda_r$}
		
		\ntxt{19}{2}{,}
	}\qquad\qquad
	\tikz[thick,xscale=.25,yscale=.25]{
		\ntxt{-3}{2}{$\rmM_{\lambda'}^{\lambda}=$}
		
		\draw (0,0) -- (0,5);
		\ntxt{0}{-1.5}{$\lambda_1$}
		\ntxt{0}{6.5}{$\lambda_1$}
		
		\ntxt{4}{2}{$\ldots$}
		
		\draw (9,4) -- (9,5);
		\ntxt{9}{6.5}{$\lambda_k$}
		\dmerge{6}{0}{12}{4}
		\ntxt{6}{-1.5}{$\lambda_k^{(1)}$}
		\ntxt{12}{-1.5}{$\lambda_k^{(2)}$}
		
		\ntxt{14}{2}{$\ldots$}
		
		\draw (18,0) -- (18,5);
		\ntxt{18}{-1.5}{$\lambda_r$}
		\ntxt{18}{6.5}{$\lambda_r$}
		
		\ntxt{19}{2}{.}
	}
\end{aligned}
\end{equation}
We call such splits and merges \textit{elementary}.
\cref{SMAssoc} tells us that splits and merges are associative:
\begin{equation*}
\label{eq:assoc-SM} 
\tikz[thick,xscale=.25,yscale=.25]{
	\dsplit{0}{4}{4}{7}
	\draw (8,4) -- (8,7);
	\dsplit{2}{1}{8}{4}
	\draw (5,0) -- (5,1);
	\ntxt{0}{8}{$a$}
	\ntxt{4}{8}{$b$}
	\ntxt{8}{8}{$c$}
	\ntxt{5}{-1}{$a+b+c$}
	
	\ntxt{10}{3}{$=$}
	
	\dsplit{16}{4}{20}{7}
	\draw (12,4) -- (12,7);
	\dsplit{12}{1}{18}{4}
	\draw (15,0) -- (15,1);
	\ntxt{12}{8}{$a$}
	\ntxt{16}{8}{$b$}
	\ntxt{20}{8}{$c$}
	\ntxt{15}{-1}{$a+b+c$}
	
	\ntxt{21}{3}{,}
}\qquad\qquad
\tikz[thick,xscale=.25,yscale=.25]{
	\dmerge{0}{0}{4}{3}
	\draw (8,0) -- (8,3);
	\dmerge{2}{3}{8}{6}
	\draw (5,6) -- (5,7);
	\ntxt{0}{-1}{$a$}
	\ntxt{4}{-1}{$b$}
	\ntxt{8}{-1}{$c$}
	\ntxt{5}{8}{$a+b+c$}
	
	\ntxt{10}{3}{$=$}
	
	\dmerge{16}{0}{20}{3}
	\draw (12,0) -- (12,3);
	\dmerge{12}{3}{18}{6}
	\draw (15,6) -- (15,7);
	\ntxt{12}{-1}{$a$}
	\ntxt{16}{-1}{$b$}
	\ntxt{20}{-1}{$c$}
	\ntxt{15}{8}{$a+b+c$}
	
	\ntxt{21}{3}{.}
}
\end{equation*}
Moreover, for any $\mu\subset\lambda$ the corresponding split $\rmS^{\mu}_{\lambda}$ and merge $\rmM^{\lambda}_{\mu}$ can be written as a product of elementary ones.

\paragraph{\textbf{Crossings}}
Let $\lambda,\lambda'$ be as above, and let $\lambda'' = (\lambda_1,\ldots,\lambda_{k-1},\lambda_k^{(2)},\lambda_k^{(1)},\lambda_{k+1}\ldots, \lambda_r)$ be obtained from $\lambda'$ by permuting $\lambda_k^{(1)}$ with $\lambda_k^{(2)}$.
Consider the element $\rmR^{\lambda''}_{\lambda'}:=\rmS^{\lambda''}_{\lambda}\cdot \rmM^\lambda_{\lambda'}\in \Sc_{\lambda',\lambda''}$, which we will call an \textit{elementary permutation}.
Diagrammatically, we will depict it as a crossing:
\[
\tikz[thick,xscale=.25,yscale=.25,text height=1.2ex,text depth=.25ex]{
	\crosin{0}{0}{4}{7}
	\ntxt{0}{-1}{$a$}
	\ntxt{4}{-1}{$b$}
	\ntxt{0}{8}{$b$}
	\ntxt{4}{8}{$a$}
	
	\ntxt{6}{3.5}{$:=$}
	
	\dmerge{8}{0}{12}{3}
	\draw (10,3) -- (10,4);
	\dsplit{8}{4}{12}{7}
	\ntxt{8}{-1}{$a$}
	\ntxt{12}{-1}{$b$}
	\ntxt{8}{8}{$b$}
	\ntxt{12}{8}{$a$}
	
	\ntxt{13}{3.5}{.}
}
\]
More generally, let $\lambda,\mu\in\Comp(n)$, and assume that $\mu$ can be obtained from $\lambda$ by a permutation of components.
Let us pick such an element $w\in\fk{S}_r$, where $r$ is the number of components in $\lambda$; note that it is not necessarily unique.
Fix a presentation $w = s_{i_1}\ldots s_{i_l}$, where $s_i\in\fk{S}_r$ are transpositions, and $l$ is the length of $w$.
We then define $\rmR_\lambda^\mu(w)$ as the corresponding product of elementary permutations.

\begin{rmq}
Note that braid relations do not hold for elementary permutations $\rmR^{\lambda''}_{\lambda'}$.
In particular, the general definition of $\rmR_\lambda^\mu(w)$ heavily depends on the choice of presentation of $w$.
One could define these elements in a more canonical way using twisted bialgebra relations, but do not need this for our purposes.
However, this can be easily done for strands of multiplicity 1, see \cref{KLRisowreath}.
There, the elements $\tau_r$ of $\WR_n(H^*(C))$ are the ``canonical crossings'', which differ from the naive split-merge crossings above by a constant.
\end{rmq}

\subsection{Basis of $\Sc_n$}
Let $\lambda, \mu\in \Comp(n)$, and consider $\Sc_{\mu,\lambda} = H^{G_n}_*(Z_{\mu,\lambda})$.
Recall from the proof of \cref{KLRLocInj} that we have a stratification 
\[
Z_{\mu,\lambda} = \bigsqcup_{w\in \doubleS{\mu}{\lambda}}Z_{\mu,\lambda}^w,
\]
which induces a filtration $\{\Sc_{\mu,\lambda}^w\}$ on $\Sc_{\mu,\lambda}$ with associated graded $\bigoplus_w H_*^{G_n}(Z_{\mu,\lambda}^w)$.
We have $H_*^{G_n}(Z_{\mu,\lambda}^w) = \bf{P}_n^{\frakS_{\lambda'}}$, where $\lambda'\in\Comp(n)$ is such that $w\frakS_{\lambda'}w^{-1} = \Gamma_{w}^{\mu\lambda}$.
Since $Z_{\mu,\lambda}^w\simeq G_n\times_{P_{\mu\lambda}^w}\left( \widetilde{Q}_\mu\cap w.\widetilde{Q}_\lambda \right)$, the $T_n$-fixed points $\left( Z_{\mu,\lambda}^w \right)^{T_n}$ are given by $\frakS_n/\frakS_{\mu'}\times C^n$.
Note that 
\[
e(N_{\{1\}\times C^n}Z_{\mu,\lambda}^w) = e(T_{(F_e,F_w)}\Omega_w)e\left(N_{C^n}\!\left( \widetilde{Q}_\mu\cap w.\widetilde{Q}_\lambda \right)\right) = \prod_{(i,j)\in N_\mu\cup gN_\lambda}(x_j-x_i)\prod_{(i,j)\in I_\mu \cap gI_\lambda}(x_i-x_j+\Delta_{ij}).
\]
Let us denote $\beta_w=e(N_{\{1\}\times C^n}Z_{\mu,\lambda}^w)^{-1}\gamma_\lambda^w\gamma_\mu$.
Further, for each $\lambda\in \Comp(n)$ pick a basis $B_\lambda$ of $\bf{P}_n^{\fk{S}_\lambda}$.

\begin{lm}\label{basis-highest-terms}
Let $\{b_{w,P}: w\in \doubleS{\mu}{\lambda}, P\in B_{\lambda'}\}$ be a collection of elements in $\Sc_{\mu,\lambda}$.
Assume that
\[
\Xi_n(b_{w,P}) = \xi^{\mu\lambda}_{(w,\beta_w P)} + \sum_{w'\prec w} \xi_{(w',a_{w'})}
\]
for all $w$, $P$.
Then $\{b_{w,P}\}$ is a $\bbk$-basis of $\Sc_{\mu,\lambda}$.
\end{lm}
\begin{proof}
It is enough to show that these elements form a basis after passing to the associated graded $\bigoplus_w H_*^{G_n}(Z_{\mu,\lambda}^w)$.
An argument analogous to \cref{isovialoc} shows that the composition $H^{G_n}_*(Z_{\mu,\lambda}^w)\to H^{T_n}_*((Z_{\mu,\lambda}^w)^{T_n})\subset \Sc_{\mu,\lambda}^\loc$ is given by $P\mapsto \xi_{(w,P)}^{\mu\lambda}$.
The assumption on $b_{w,P}$ then implies that it is contained in $H_*^{G_n}(\overline{Z_{\mu,\lambda}^w})$.
Consider the following diagram:
\[\begin{tikzcd}
H_*^{G_n}(Z_{\mu,\lambda}^w)\ar[d,hook,"\Xi_n^w"] & H_*^{G_n}(\overline{Z_{\mu,\lambda}^w})\ar[l,"j^*"']\ar[r,"i_*"]\ar[d,hook,"\Xi_n^{\preccurlyeq w}"] & H_*^{G_n}(Z_{\mu,\lambda})\ar[d,hook,"\Xi_n"]\\
H_*^{T_n}((Z_{\mu,\lambda}^w)^{T_n})_\loc & H_*^{T_n}((\overline{Z_{\mu,\lambda}^w})^{T_n})_\loc \ar[l,"h"']\ar[r,hook] & \Sc_{\mu,\lambda}^\loc
\end{tikzcd}\]
We have $b_{w,P} = i_*b'$ for some $b'\in H_*^{G_n}(\overline{Z_{\mu,\lambda}^w})$, and the image of $b_{w,P}$ in the associated graded is given by $j^*(b')$.
On the other hand, the left square in the diagram above commutes, and by \cref{LocPullPush} the right square commutes up to Euler class, which equals precisely to $\xi^{\mu\lambda}_{1,\beta_w^{-1}}$.
In effect, the closed embedding $\Xi_n^w$ contributes $e(N_{\{1\}\times C^n}Z_{\mu,\lambda}^w)$, while $\Xi_n$ contributes
\[
e(N_{\{1\}\times C^n}Y_\mu)^{-1}e(N_{\{w\}\times C^n} Y_\lambda)^{-1} = \left(\gamma_\mu\gamma_\lambda^w\right)^{-1}.
\]
Thus
\begin{gather*}
\Xi_n^w\circ j^*(b') = h\circ\Xi_n^{\preccurlyeq w} (b') = \xi^{\mu\lambda}_{1,\beta_w^{-1}} h\circ \Xi_n\circ i_*(b') =  \xi^{\mu\lambda}_{(w,\beta_w^{-1}\beta_w P)} = \xi^{\mu\lambda}_{(w, P)}.
\end{gather*}
Since the localization map $\Xi_n$ is injective, we conclude that the image of $b_{w,P}$ in the associated graded is $P$.
Running over all $w\in \doubleS{\mu}{\lambda}$, $P\in B_{\lambda'}$ we obtain a basis of $\Sc_{\mu,\lambda}$.
\end{proof}

Let $g\in \doubleS{\mu}{\lambda}$.
As before, let $\mu'\in\Comp(n)$ be such that $\Gamma^{\mu\lambda}_g\simeq \fk{S}_{\mu'}$, and let $\lambda'$ be such that $\fk{S}_{\lambda'} = g^{-1}\fk{S}_{\mu'}g = \fk{S}_{\lambda}\cap g\fk{S}_{\mu}g^{-1}$.
Note that $g$ induces a permutation $w$ on the set of components of $\lambda'$, which transforms $\lambda'$ into $\mu'$.
Pick $P\in\bf{P}_n^{\fk{S}_{\lambda'}}$.
For each such pair $(g,P)$, we construct the following elements of $\Sc_{\mu,\lambda}$:
\begin{equation}\label{fla:basisSchur}
\Psi_g^P = \rmM_{\mu'}^\mu\rmR_{\lambda'}^{\mu'}(w)P\rmS_\lambda^{\lambda'},\qquad \Psi_g = \Psi_g^1 = \rmM_{\mu'}^\mu\rmR_{\lambda'}^{\mu'}(w)\rmS_\lambda^{\lambda'}.
\end{equation}

\begin{exe}
	Let $\lambda=(3,1)$, $\mu=(2,2)$, $g=(1,3,4,2)$, and $P=x_1^2x_2x_3$.
	Then $\lambda'=(1,2,1)$, $\mu'=(1,1,2)$, $w=(1,3,2)$, and we have
	
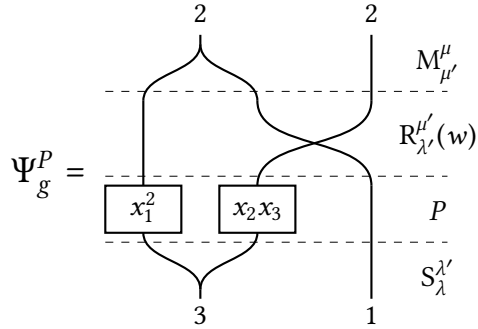
\[
	\tikz[thick,xscale=.25,yscale=.25,text height=1.2ex,text depth=.25ex]{
		\ntxt{-5}{5.5}{\Large $\Psi_g^P=$}
		\ntxt{3}{-2}{$3$}
		\ntxt{12}{-2}{$1$}
		\draw (3,-1) -- (3,-0.5);
		\dsplit{0}{-0.5}{6}{2.5}
		\ntxt{15.5}{0}{$\rmS_\lambda^{\lambda'}$}
		
		\draw[very thin,dashed] (-2,2) -- (16,2);
		
		\opbox{-2}{2.5}{2}{5}{$x_1^2$}
		\opbox{4}{2.5}{8}{5}{$x_2x_3$}
		\ntxt{15.5}{3.5}{$P$}
		
		\draw[very thin,dashed] (-2,5.5) -- (16,5.5);
		
		\draw (12,-1) -- (12,5);
		\crosin{6}{5}{12}{9.5}
		\draw (0,5) -- (0,9.5);
		\ntxt{15.5}{7.5}{$\rmR_{\lambda'}^{\mu'}(w)$}
		
		\draw[very thin,dashed] (-2,10) -- (16,10);
		
		\dmerge{0}{9.5}{6}{12.5}
		\draw (12,9.5) -- (12,13);
		\draw (3,12.5) -- (3,13);
		\ntxt{15.5}{11.5}{$\rmM_{\mu'}^\mu$}
		
		\ntxt{3}{14}{$2$}
		\ntxt{12}{14}{$2$}
	}
	\]
\end{exe}

\begin{prop}\label{lm:basis-Schur-zigzag}
The following set is a basis for $\Sc_{\mu,\lambda}$:
\[
\left\{\Psi^P_g:g\in \doubleS{\mu}{\lambda}, P\in B_{\lambda'}\right\}.
\]
\end{prop}
\begin{rmq}
Note that when $\mu=(n)$ is the trivial composition, we have $Z = Y_\lambda$, and this statement follows from the isomorphism~\eqref{FlagCoh}.
\end{rmq}
\begin{proof}
In order to simplify the notation, we will write $Z = Z_{\mu,\lambda}$ throughout the proof.
In light of \cref{basis-highest-terms}, we need to compute highest terms of $\Xi_n(\Psi_g^P)$.
From now on, we will denote the presence of lower terms by ellipsis.

We begin by computing elementary permutations.
Let $\nu,\lambda',\lambda''$ be as in the definition of $\rmR_{\lambda'}^{\lambda''}$; we write $\nu$ instead of $\lambda$ to avoid conflict of notation.
Note that the longest element in $\fk{S}_{\lambda''}\backslash \fk{S}_{\nu}/ \fk{S}_{\lambda'}\subset \doubleS{\lambda''}{\lambda'}$ is the permutation that exchanges the components $\nu_k^{(1)}$ and $\nu_k^{(2)}$; denote it by $s$.
Using formula~\eqref{SchurLocComp} and \cref{SMLoc}, we obtain:
\begin{align*}
\rmR_{\lambda'}^{\lambda''} & = \rmS_{\nu}^{\lambda''} \cdot \rmM^{\nu}_{\lambda'}
 = \xi^{\lambda''\nu}_{(1,\gamma_{\nu})}*\xi^{\nu\lambda'}_{(1,\gamma_{\nu})}
 = \frac{1}{|\fk{S}_{\lambda''}||\fk{S}_{\lambda'}|}\sum_{b\in\fk{S}_{\nu}} |\Gamma_b^{\lambda''\lambda'}| \xi^{\lambda''\lambda'}_{(b,\gamma_{\nu})}
 = \frac{|\Gamma_s^{\lambda''\lambda'}|}{|\fk{S}_{\lambda''}||\fk{S}_{\lambda'}|} \left(\sum_{b\in \frakS_{\lambda'}}\xi^{\lambda''\lambda'}_{(sb,\gamma_{\nu})}\right) +\ldots\\
& = \xi^{\lambda''\lambda'}_{(s,\gamma_{\nu})} +\ldots
\end{align*}

Next, consider general permutations.
Let $w\in \frakS_r$ be the permutation of components of $\lambda'$ defined by $g$, and fix a reduced presentation $w=s_l\ldots s_1$\footnote{here, $s_i$ does not stand for the transposition $(i,i+1)$, but rather for some simple transposition $(j,j+1)$, $1\leqslant j \leqslant r-1$}.
We write $w_i:= s_i\ldots s_1\in \frakS_r$.
Let $\lambda^i = w_i(\lambda')$, and $\nu^i$ the intermediate composition between $\lambda^{i-1}$ and $\lambda^i$, i.e. $\nu^i$ is such that we have $\rmR_{\lambda^{i-1}}^{\lambda^{i}}=\rmS^{\lambda^{i}}_{\nu^i}\cdot\rmM_{\lambda^{i-1}}^{\nu^i}$).
Note that $\lambda^l = \mu'$.
Let us write $M_j^i = [\tilde{\lambda}^{i}_{j-1}+1,\tilde{\lambda}^{i}_{j}]$, and $M^i_{jk} = M^i_j\times M^i_k$.
We also denote by $s_i$ the corresponding longest element in $\fk{S}_{\lambda^i}\backslash \fk{S}_{\nu^i}/ \fk{S}_{\lambda^{i-1}}$ as before.

\begin{lm}
Assume $s'_lb_{l-1}s'_{l-1}\ldots b_1s'_1\in \fk{S}_{\mu'}g\fk{S}_{\lambda'}$, where $s'_i\in \fk{S}_{\lambda^i}\backslash \fk{S}_{\nu^i}/ \fk{S}_{\lambda^{i-1}}$, and $b_i\in \fk{S}_{\lambda^i}$.
Then $s'_i = s_i$ for all $i$.
\end{lm}
\begin{proof}
The condition on $s'_i$'s can be rewritten as $g = b_ls'_l\ldots b_1s'_1$.
Suppose the equality $s'_i=s_i$ does not always hold.
Let $i\in [1,r]$ be the minimal index such that for some $k$ we have $w_k(i)\neq w_{k-1}(i)$ and $s_k\neq s'_k$.
Consider the smallest such $k$, and denote $\alpha = w_k(i)$.
There exists an index $m\in M^{k-1}_{w_{k-1}(i)}$ such that $s'_k(m)$ lies in $M^k_{w_{k-1}(i)}$; note that $w_{k-1}(i) = \alpha-1$.
Let $p_0 = (b_{k-1}s_{k-1}'\ldots b_1s_1')^{-1}(m)$, and consider the sequence $p_j = b_js_j'(p_{j-1})$. 
Let $a_j$ be such that $p_j\in M^j_{w_j(a_j)}$.

Pictorially, we draw the presentation $w = s_l\ldots s_1$ as a diagram on $r$ strands, going from bottom to top and numbered $1$ to $r$.
Since this presentation is reduced, each pair of strands intersects at most once.
Then $a_j$ tells us on which strand the image of $p_0$ is located after $j$-th crossing.
Depending of $s'_j$, at each crossing we either swap the strand or not.
The condition $g = b_ls'_l\ldots b_1s'_1$ tells us that as we traverse the diagram from bottom to top, we should end up on the $i$-th strand.
The minimality of $i$ implies that we can only change the strand on intersections with strands $i,\ldots, r$.
However, for $i_0>i$ the $i_0$-th strand has to intersect $i$-th strand first before intersecting $a_k$-th strand.
Therefore even if we change strands, the new strand cannot intersect $i$-th strand again, so that $g(p_0) = p_r\notin g(M^0_{i})$.
We have arrived at a contradiction.
\end{proof}

In particular, the highest term of $\rmR_{\lambda'}^{\mu'}(w)$ must be contained in the product of highest terms of elementary permutations.
By definition, conjugation by $s_i$ sends $\fk{S}_{\lambda^{i-1}}$ to $\fk{S}_{\lambda^i}$.
As a consequence $s_{i+1}bs_{i}$ defines the same class in $\fk{S}_{\lambda^{i+1}}\backslash \fk{S}_{n}/\fk{S}_{\lambda^{i-1}}$ for any $b\in \fk{S}_{\lambda^i}$.
Moreover, we have $|\Gamma^{\lambda^{i+1}\lambda^i}_{s_{i+1}}|=|\Gamma^{\lambda^{i+1}\lambda^{i-1}}_{s_{i+1}bs_{i}}|=|\fk{S}_{\lambda^{i+1}}|$, so that all coefficients in~\eqref{SchurLocComp} cancel out.
We get
\begin{gather*}
\rmR_{\lambda'}^{\mu'}(w)
 = \rmR_{\lambda^{l-1}}^{\mu'}\ldots \rmR_{\lambda'}^{\lambda^1}
 = \xi^{\mu'\lambda^{l-1}}_{(s_l,\gamma_{\nu^l})}*\ldots*\xi^{\lambda^1\lambda'}_{(s_1,\gamma_{\nu^1})} + \ldots
  = \xi^{\mu'\lambda'}_{(g,E)} + \ldots;\\
  E = \gamma_{\nu^l}(\gamma_{\nu^{l-1}}\gamma^{-1}_{\lambda^{l-1}})^{s_l}\ldots (\gamma_{\nu^{1}}\gamma^{-1}_{\lambda^{1}})^{s_l\ldots s_2}.
\end{gather*}

Denote $\zeta_0 = \gamma_{\lambda'}$, and $\zeta_i = (\gamma_{\nu^i}^{-1}\gamma_{\lambda^i})\zeta_{i-1}^{s_i}$ for $i>0$.
Recall that $r$ is the number of components in $\mu'$ and let $A_\mu$ be the subset of $[1,r]^2$ such that we have $N_\mu = \bigsqcup_{(j,k)\in A_\mu}M^l_{jk}$.
By analogy, we define $N_{\mu,i} = \bigsqcup_{(j,k)\in A_\mu}M^i_{jk}$, and $I_{\mu,i} = N_{\mu,i} \cup \left(\bigsqcup_j M^i_{jj}\setminus \{(j,j):1\leq j\leq n\}\right)$.

\begin{lm}
Let $g_i = s_i\ldots s_1\in \fk{S}_n$.
We have 
$$\zeta_i=\prod_{(p,q)\in g_iN_{\lambda}\cup N_{\mu,i}}(x_q-x_p) \prod_{(p,q)\in g_iI_{\lambda}\cap I_{\mu,i}}(x_p-x_q+\Delta_{pq}).
$$
\end{lm}
\begin{proof}
For $i=0$, the claim follows from the definition of $\gamma_{\lambda'}$.
Let $i>0$ and proceed by induction.
Suppose $s_i = (t, t+1)$; then we have
\begin{equation*}
	I_{\nu^i} = I_{\lambda^i} \sqcup M^i_{t+1,t},\qquad N_{\lambda^i} = N_{\nu^i} \sqcup M^i_{t,t+1}.
\end{equation*}
Looking at the formula for $\zeta_i$, we thus need to prove the following equalities between subsets in $[1,n]\times [1,n]$:
\begin{gather*}
g_iI_{\lambda}\cap s_iI_{\mu,i-1} = (g_iI_{\lambda}\cap I_{\mu,i}) \sqcup M^i_{t+1,t},\\
g_iN_{\lambda}\cup N_{\mu,i} = (g_iN_{\lambda}\cup s_iN_{\mu,i-1}) \sqcup M^i_{t,t+1}.
\end{gather*}
We will only prove the second identity; the first one can be obtained analogously by passing to complementary index sets and substituting $(i,j)\mapsto (j,i)$.

It is easy to check that $N_{\mu,i}\setminus s_iN_{\mu,i-1} = M^i_{t,t+1}$, and conversely $s_iN_{\mu,i-1}\setminus N_{\mu,i} = M^i_{t+1,t}$.
Therefore
\begin{equation}\label{boxesincexc}
g_iN_{\lambda}\cup N_{\mu,i} = g_iN_{\lambda}\cup \left( (s_iN_{\mu,i-1}\setminus M^i_{t+1,t}) \sqcup M^i_{t,t+1} \right).
\end{equation}
Assume that $t = g_{i-1}(k_1)$, $t+1 = g_{i-1}(k_2)$; note that we automatically have $k_1 < k_2$.
We cannot have crossings between the strands which split off the same thick strand.
Thus $k_1$ and $k_2$ lie in different components of $\lambda$, and $N_\lambda \supset M^0_{k_1,k_2}$, $N_\lambda \cap M^0_{k_2,k_1} = \emptyset$ by definition of $N_\lambda$.
Applying $g_i$, we get $M^i_{t+1,t}\subset g_iN_{\lambda}$ and $M^i_{t,t+1}\cap g_iN_{\lambda} = \emptyset$, which together with \eqref{boxesincexc} implies the desired identity.
\end{proof}

An immediate consequence of the lemma above is that
\[
\gamma_{\mu'}^{-1}E(\gamma_{\lambda'}^{-1})^g = \zeta_l^{-1} = \beta_g(\gamma_\lambda \gamma_\mu)^{-1}.
\]
This formula allows us to compute the highest term of $\Psi_g^P$:
\begin{align*}
\Xi_n(\Psi_g^P) & = \Xi_n(\rmM_{\mu'}^\mu\rmR_{\lambda'}^{\mu'}(w)P\rmS_\lambda^{\lambda'}) = \xi^{\mu\mu'}_{(1,\gamma_\mu)}*(\xi^{\mu'\lambda'}_{(g,E)} + \ldots)*\xi^{\lambda'\lambda}_{(1,P\gamma_\lambda)}\\
& = \xi^{\mu\mu'}_{(1,\gamma_\mu)}*\left( \frac{|\Gamma^{\mu'\lambda}_g||\fk{S}_{\lambda'}|}{|\Gamma^{\lambda'\lambda}_1||\Gamma^{\mu'\lambda'}_g|}\xi^{\mu'\lambda}_{(g,E(P\gamma_\lambda\gamma_{\lambda'}^{-1})^g)} +\ldots \right) 
= \frac{|\Gamma^{\mu\lambda}_g||\fk{S}_{\mu'}|}{|\Gamma^{\mu\mu'}_1||\Gamma^{\mu'\lambda}_g|} \xi^{\mu\lambda}_{(g,P^g \gamma_\mu\gamma_{\mu'}^{-1}E(\gamma_\lambda\gamma_{\lambda'}^{-1})^g)} + \ldots\\
& = \xi^{\mu\lambda}_{(g,P^g \beta_g)} +\ldots
\end{align*}
Substituting $P\rightsquigarrow P^{g^{-1}}$, we may conclude by \cref{basis-highest-terms}.
\end{proof}

\begin{corr}
The Schur algebra $\Sc_n$ is generated by polynomials and elementary splits and merges.
\end{corr}

\subsection{Polynomial representation}
The localized Schur algebra $\Sc_n^\loc$ admits an action on
\[
P_n^\loc := \bigoplus_{\lambda} \left( H^*(C)^{\otimes n}(x_1,\ldots,x_n) \right)^{\fk{S}_\lambda}
\]
by \cref{ConvModule}.
Similarly to~\eqref{SchurLocComp}, using formula~\eqref{ExAct} one shows that
\begin{equation}\label{SchurLocRep}
\xi_{(g,x)}^{\mu\lambda}.y =\frac{1}{|\Gamma_g^{\mu\lambda}|} \sum_{a\in\fk{S}_\mu} \left( x(y\gamma_\lambda^{-1})^g \right)^a.
\end{equation}

By \cref{ConvLocMod}, we have a commutative square
\begin{equation}\label{SchurPolLoc}
	\begin{tikzcd}
		\Sc_n\ar[d]\ar[r] & \End P_n\ar[d]\\
		\Sc_n^\loc\ar[r] & \End P_n^\loc
	\end{tikzcd}
\end{equation}

\begin{prop}\label{SchurPolRepFaith}
Let $\lambda$, $\lambda'$ be as in the definition of elementary splits and merges.
The algebra $\Sc_n$ has a faithful representation on $P_n$ such that
\begin{itemize}
	\item polynomials $P\in \bf{P}_n^{\fk{S}_\lambda}$ act by multiplication on $H^*_{G_n}(Y_\lambda)\simeq \bf{P}_n^{\fk{S}_\lambda}$, 
	\item the split $\rmS_\lambda^{\lambda'}$ acts by the natural inclusion of rings $\bf{P}_n^{\fk{S}_\lambda}\to \bf{P}_n^{\fk{S}_{\lambda'}}$,
	\item the merge $\rmM^\lambda_{\lambda'}$ acts by the following operator:
	\[
	P\mapsto \sum_{a\in \fk{S}_\lambda/\fk{S}_{\lambda'}} \left(y\prod_{(i,j)\in N_{\lambda'}^\lambda}\left(1+\frac{\Delta_{ij}}{x_j-x_i}\right)\right)^a, \quad N^\lambda_{\lambda'}  = [\tilde{\lambda}_{k-1}+1,\tilde{\lambda}_{k-1}+\lambda_k^{(1)}]\times[\tilde{\lambda}_{k-1}+\lambda_k^{(1)}+1,\tilde{\lambda}_{k}].
	\]
\end{itemize}
\end{prop}
\begin{proof}
The vertical maps in diagram~\eqref{SchurPolLoc} are injective by \cref{KLRLocInj} and Thom isomorphism.
Moreover, the restriction of $\Sc_n^\loc\to \End P_n^\loc$ to $\Sc_{\lambda,\mu}^\loc$ is nothing else than the pullback $r_{\lambda\mu}^*$, and is therefore injective as well.
The faithfulness of the polynomial representation $\Sc_n \to \End P_n$ follows.

Let us compute the action of generators by applying $\Xi_n$, and using formula~\eqref{SchurLocRep}:
\begin{align*}
P.y & = \xi_{(1,P\gamma_\lambda)}^{\lambda\lambda}.y = \frac{1}{|\fk{S}_\lambda|}\sum_{a\in\fk{S}_\lambda} \left( P\gamma_\lambda y\gamma_\lambda^{-1} \right)^a = Py;\\
\rmS_\lambda^{\lambda'}.y & = \xi_{(1,\gamma_\lambda)}^{\lambda'\lambda}.y = \frac{1}{|\fk{S}_{\lambda'}|}\sum_{a\in\fk{S}_{\lambda'}} \left( \gamma_\lambda y\gamma_\lambda^{-1} \right)^a = y; \\
\rmM^\lambda_{\lambda'}.y & = \xi_{(1,\gamma_\lambda)}^{\lambda\lambda'}.y = \frac{1}{|\fk{S}_{\lambda'}|}\sum_{a\in\fk{S}_\lambda} \left( \gamma_\lambda y\gamma_{\lambda'}^{-1} \right)^a = \sum_{a\in \fk{S}_\lambda/\fk{S}_{\lambda'}} \left(\frac{\prod_{(i,j)\in I_\lambda\setminus I_{\lambda'}}(x_i-x_j+\Delta_{ij})}{\prod_{(i,j)\in N_{\lambda'}\setminus N_{\lambda}}(x_j-x_i)}.y\right)^a.
\end{align*}
We conclude by observing that $N_{\lambda'}\setminus N_\lambda = N_{\lambda'}^\lambda$, and $(i,j)\in I_\lambda\setminus I_{\lambda'}$ if and only if $(j,i)\in N_{\lambda'}^\lambda$.
\end{proof}

\begin{rmq}
As in~\cite{Prz_QSAC2019}, this polynomial representation can be realized inside a tensor power of cohomological Hall algebra of torsion sheaves on $C$.
We do not pursue this point of view here.
\end{rmq}

\begin{exe}
Consider the case $n=2$.
Denote the inclusion $\bf{P}_2^{\fk{S}_2}\subset \bf{P}_2$ by $f_2$.
We have only one possible split and merge respectively; we write $\rmS = \rmS_2^{(1,1)}$, $\rmM = \rmM_{(1,1)}^2$, and omit labels on strands.
Using \cref{lm:basis-Schur-zigzag}, one can check that $\Sc_2$ is generated by $\rmS$, $\rmM$ and polynomial subalgebras $\bf{P}_2^{\fk{S}_2}$, $\bf{P}_2$, subject to the following relations:
\begin{equation*}
	\tikz[thick,xscale=.25,yscale=.25]{
		\draw (1.5,0) -- (1.5,0.5);
		\dsplit{0}{0.5}{3}{2.5}
		\opbox{-0.5}{2.5}{3.5}{4.5}{$Q$}
		\dmerge{0}{4.5}{3}{6.5}
		\draw (1.5,6.5) --(1.5,7);
		\ntxt{5}{3.5}{$=$}
		\opbox{6.5}{2}{20.5}{5}{$Q+s_1(Q)-\frac{\Delta_{12}(Q-s_1(Q))}{x_1-x_2}$}
		\draw (13,5) -- (13,6.5);
		\draw (13,2) -- (13,0.5);
		\ntxt{21.5}{3.5}{,}
	}\qquad
	\tikz[thick,xscale=.25,yscale=.25]{
		\draw (2,0) -- (2,1);
		\opbox{1}{1}{3}{3}{$P$}
		\dsplit{0.5}{3}{3.5}{7}
		
		\ntxt{4.75}{3.5}{$=$}
		
		\draw (8,0) -- (8,1);
		\dsplit{6.5}{1}{9.5}{4}
		\opbox{6}{4}{10}{6}{$f_2(P)$}
		\draw (6.5,6) -- (6.5,7);
		\draw (9.5,6) -- (9.5,7);
			
		\ntxt{11}{3.5}{,}
	}\qquad
	\tikz[thick,xscale=.25,yscale=.25]{
		\dmerge{0.5}{0}{3.5}{4}
		\opbox{1}{4}{3}{6}{$P$}
		\draw (2,6) -- (2,7);
		
		\ntxt{4.75}{3.5}{$=$}
		
		\draw (6.5,0) -- (6.5,1);
		\draw (9.5,0) -- (9.5,1);
		\opbox{6}{1}{10}{3}{$f_2(P)$}
		\dmerge{6.5}{3}{9.5}{6}
		\draw (8,6) -- (8,7);
		
		\ntxt{11}{3.5}{,}
		
		\ntxt{18}{3.5}{$P\in \bf{P}_2^{\fk{S}_2}, Q\in \bf{P}_2$.}
	}
\end{equation*}
\end{exe}

\begin{rmq}\label{rmq:SchurP1}
When $C=\bbP^1$, we have $\bf{P}_n=\bfk[x_1,\ldots,x_n,c_1,\ldots,c_n]/(c_1^2,\dots,c_n^2)$ and $\Delta_{ij} = c_i+c_j$.
For instance, $\rmM\in\Sc_2^{\bbP^1}$ acts on the polynomial representation by $(1+s_1)-\frac{(c_1+c_2)(1-s_1)}{x_1-x_2}$.
\end{rmq}
\section{KLR algebra of a smooth curve}\label{sec:KLRcurve}
In this section we study a subalgebra of $\Sc_n$, which admits a simpler description.

\subsection{Demazure operators}
Let us recall the definition and basic properties of Demazure operators.

\begin{defi}For $r\in [1,n-1]$, denote by $\partial_r$ the {\it Demazure operator} 
\[
\partial_r\colon \bfk[x_1,\ldots,x_n]\to \bfk[x_1,\ldots,x_n], \qquad
P\mapsto (P-s_r(P))/(x_r-x_{r+1}).
\]
\end{defi}
Note that we have $\partial_r(P)=0$ if and only if $s_r(P)=P$.
In particular, a polynomial $P$ is symmetric if and only if it is annihilated by all $\partial_r$ for $r\in [1,n-1]$.

The following relations are well-known. 
\begin{lm}
\label{lem:rel-Demaz}
We have 
\[
\begin{array}{rcll}
	\partial_r^2&=&0& \mbox{\rm for } r\in [1,n-1],\\
	\partial_r\partial_t&=&\partial_t\partial_r&  \mbox{\rm for } r,t\in [1,n-1],~|r-t|>1,\\
	\partial_r\partial_{r+1}\partial_r&=&\partial_{r+1}\partial_r\partial_{r+1} &\mbox{\rm for } r\in [1,n-2].\\
\end{array}
\]
\end{lm}

For each $w\in\frakS_n$, fix a reduced expression $w=s_{k_1}\ldots s_{k_r}$ and define $\partial_w=\partial_{k_1}\ldots\partial_{k_r}$.
Since Demazure operators satisfy braid relations, this definition is independent of the choice of a reduced expression.
Moreover, the square-zero relation implies that we have $\partial_{k_1}\ldots\partial_{k_r}=0$ if $s_{k_1}\ldots s_{k_r}$ is not a reduced expression.

Let $w_{0,n}$ be the longest element in $\frakS_n$.

\begin{lm}
\label{lem:Dem-image-in-sym}
For any $P\in \bfk[x_1,\ldots,x_n]$, the polynomial $\partial_{w_{0,n}}(P)$ is symmetric. 
\end{lm}
\begin{proof}
Since we have $\ell(s_rw_{0,n})<\ell(w_{0,n})$ for each $r\in [1,n-1]$, we get $\partial_r\partial_{w_{0,n}}=0$.
Then the polynomial $\partial_{w_{0,n}}(P)$ is symmetric because for each $r\in [1,n-1]$ we have $\partial_r(\partial_{w_{0,n}}(P))=0$.
\end{proof}

\begin{rmq}
\label{rem:image-Dem}
The lemma above shows that the image of $\partial_{w_{0,n}}$ is contained in symmetric polynomials.
This inclusion is in fact an equality.
Indeed, take an arbitrary polynomial $Q\in \bfk[x_1,\ldots,x_n]$ such that $\partial_{w_{0,n}}(Q)=1$, for example $Q=x_1^{n-1}x_2^{n-2}\ldots x_{n-2}^2x_{n-1}$.
Since Demazure operators commute with multiplication by symmetric polynomials, we have $\partial_{w_{0,n}}(PQ)=P\partial_{w_{0,n}}(Q)=P$ for any symmetric polynomial $P\in \bfk[x_1,\ldots,x_n]^{\frakS_n}$.
\end{rmq}

\begin{defi}
\label{def:wa,b}
For positive integers $a$, $b$ with $a+b=n$ consider the permutation 
$w_{0,a,b} \in\frakS_n$ given by
\[
w_{0,a,b}(i)=
\begin{cases}
	i+b & \mbox{ if } 1\leqslant i\leqslant a,\\
	i-a & \mbox{ if } a< i\leqslant n.
\end{cases}
\]
\end{defi}

\begin{lm}
\label{lem:Dem-ab-sym}
For any $P\in\bfk[x_1,\ldots,x_n]^{\frakS_a\times\frakS_b}$, we have $\partial_{w_{0,a,b}}(P)\in\bfk[x_1,\ldots,x_n]^{\frakS_n}$.
\end{lm}
\begin{proof}
Abusing the notation, let us write $w_{0,a}$ and $w_{0,b}$ for the images of $w_{0,a}\in\frakS_a$ and $w_{0,b}\in \frakS_b$ under the inclusion $\frakS_a\times \frakS_b\subset \frakS_n$.
We have $w_{0,n}=w_{0,a,b}w_{0,a}w_{0,b}$.
	
It is enough to prove the statement for $P$ of the form $P=QR$, where $Q$ is a symmetric polynomial on $x_1,\ldots,x_a$ and $R$ is a symmetric polynomial on $x_{a+1},\ldots,x_{n}$.
Moreover, by \cref{rem:image-Dem} we can find $Q_0\in \bfk[x_1,\ldots,x_a]$ and $R_0\in \bfk[x_{a+1},\ldots,x_n]$ such that $Q=\partial_{w_{0,a}}(Q_0)$ and $R=\partial_{w_{0,b}}(R_0)$.
Then we have 
\[
\partial_{w_{0,a,b}}(P)=\partial_{w_{0,a,b}}[\partial_{w_{0,a}}(Q_0)\partial_{w_{0,b}}(R_0)]=\partial_{w_{0,n}}(Q_0R_0).
\]
This polynomial is symmetric by \cref{lem:Dem-image-in-sym}.
\end{proof}

\begin{lm}
\label{lem:Dem-ab-formula}
For any $P\in \bfk[x_1,\ldots,x_n]^{\frakS_a\times\frakS_b}$, we have 
\[
\partial_{w_{0,a,b}}(P)=\sum_{w\in\frakS_n/(\frakS_a\times \frakS_b)}w\left(\frac{P}{\prod_{1\leqslant i\leqslant a}\prod_{a+1\leqslant j\leqslant n}(x_i-x_j)}\right)
\]
\end{lm}
\begin{proof}
Let us consider $\partial_{w_{0,a,b}}$ as a linear map
\[
\partial_{w_{0,a,b}}\colon \bfk(x_1,\ldots,x_n)^{\frakS_a\times \frakS_b}\to \bfk(x_1,\ldots,x_n)^{\frakS_n}.
\]
We can write it as a sum
\[
\partial_{w_{0,a,b}}=\sum_{w\in\frakS_n/(\frakS_a\times \frakS_b)}Q_w w,
\]
where $Q_w\in \bfk(x_1,\ldots,x_n)$.
We need to show that for each $w\in\frakS_n/(\frakS_a\times \frakS_b)$, we have 
\[
Q_w = w\left(\frac{1}{\prod_{1\leqslant i\leqslant a}\prod_{a+1\leqslant j\leqslant n}(x_i-x_j)}\right).
\]
	
By \cref{lem:Dem-ab-sym}, we have $Q_{\Id}\in \bbk[x_1,\ldots, x_n]^{\frakS_a\times \frakS_b}$ and $Q_w=w(Q_{\Id})$.
So, to complete the proof it remains to show that
\[
Q_{w_{0,a,b}}=w_{0,a,b}\left(\frac{1}{\prod_{1\leqslant i\leqslant a}\prod_{a+1\leqslant j\leqslant n}(x_i-x_j)}\right)=\frac{1}{\prod_{b+1\leqslant i\leqslant n}\prod_{1\leqslant j\leqslant b}(x_i-x_j)}.
\] 
Take a reduced decomposition $w_{0,a,b}=s_{k_1}\ldots s_{k_{ab}}$, and write 
\[
\partial_{w_{0,a,b}}=\left(\frac{1}{x_{k_1}-x_{k_1+1}}-\frac{s_{k_1}}{x_{k_1}-x_{k_1+1}}\right)\ldots \left(\frac{1}{x_{k_{ab}}-x_{k_{ab}+1}}-\frac{s_{k_{ab}}}{x_{k_{ab}}-x_{k_{ab}+1}}\right).
\]
The only way to get a term with permutation belonging to the class $w_{0,a,b}(\frakS_a\times \frakS_b)$ in this product is to take the second term in each bracket.
More precisely, when we write
\[
\left(\frac{s_{k_1}}{x_{k_1+1}-x_{k_1}}\right)\ldots \left(\frac{s_{k_{ab}}}{x_{k_{ab}+1}-x_{k_{ab}}}\right)
\]
and move all $s_i$'s to the right, we get 
\[
\prod_{t=1}^r\left( \frac{1}{x_{i_t}-x_{j_t}}\right)\cdot w_{0,a,b},
\]
where 
\[
i_t=s_{k_1}s_{k_2}\ldots s_{k_{t-1}}(k_t+1),\qquad j_t=s_{k_1}s_{k_2}\ldots s_{k_{t-1}}(k_t).
\]
Furthermore, for each $(i,j)\in [b+1,n]\times [1,b]$ there exists a unique index $t\in[1,ab]$ such that 
\[
s_{k_{t-1}}\ldots s_{k_2}s_{k_{1}}(i)>s_{k_{t-1}}\ldots s_{k_2}s_{k_{1}}(j),\qquad  s_{k_{t}}\ldots s_{k_2}s_{k_{1}}(i)<s_{k_{t}}\ldots s_{k_2}s_{k_{1}}(j).
\] 
For this $t$ we have $i=i_t$ and $j=j_t$, since the decomposition of $w_{0,a,b}$ is reduced.
Therefore
\[
\prod_{t=1}^r\left( \frac{1}{x_{i_t}-x_{j_t}}\right)=\frac{1}{\prod_{b+1\leqslant i\leqslant n}\prod_{1\leqslant j\leqslant b}(x_i-x_j)},
\]
and we may conclude.
\end{proof}

\begin{rmq}
Applying \cref{lem:Dem-ab-formula}, we can rewrite the action of merge operator $\rmM^\lambda_{\lambda'}\in \Sc_n$ on the polynomial representation (see \cref{SchurPolRepFaith}) as follows:
$$
P\mapsto \partial_{w_{0,a,b}}\left(P \prod_{(i,j)\in N_{\lambda'}^\lambda}\left(x_i-x_j-\Delta_{ij}\right)\right),
$$
where $a=\lambda_k^{(1)}$, $b=\lambda_k^{(2)}$ and the Demazure operator $\partial_{w_{0,a,b}}$ is applied to the variables in positions ${\lambda}_{k-1}+1, {\lambda}_{k-1}+2,\ldots,{\lambda}_{k}$.
\end{rmq}

\subsection{Affinized symmetric algebras}
Let $F = \bigoplus_i F_i$ be a $\bbZ_{\geqslant 0}$-graded unital finite dimensional $\bfk$-algebra.
Further, let $\sigma: F\otimes F\to\bfk$ be a non-degenerate graded pairing, such that $\sigma(fg,h)=\sigma(f,gh)$ and $\sigma(f,g)=\sigma(g,f)$ for any $f,g,h\in F$.
This makes $(F,\sigma)$ into a symmetric Frobenius algebra.
An example of such algebra is given by the cohomology ring $H^*(X,\bfk)$ of any smooth projective variety $X$.

Let $m\colon F\otimes F\to F$ be the product in $F$, and $\Delta:F\to F\otimes F$ be its dual with respect to $\sigma$.
The tensor product $F^{\otimes n}$ has a natural $\frakS_n$-action.
For any $1\leqslant i < j\leqslant n$, consider the $\bfk$-linear map 
$$
\iota_{i,j}\colon F^{\otimes 2}\to F^{\otimes n}, \quad f\otimes g\mapsto 1\otimes \ldots \otimes 1 \otimes f\otimes 1\otimes\ldots  \otimes 1\otimes g \otimes 1 \otimes\ldots \otimes 1,
$$ 
where $f$ and $g$ appears at the $i$-th and $j$-th position respectively.
Set $\Delta_{i,j}:=\iota_{i,j}(\Delta(1))\in F^{\otimes n}$.

Let $s_i = (i,i+1)$, $1\leq i< n$ be elementary transpositions in $\frakS_n$.
We will denote the image of $s_i$ in the group algebra $\bfk\frakS_n$ by $\tau_i$, and more generally for any $w\in\fk{S}_n$ we denote its image by $\tau_w$.
The following algebra is defined in \cite[Definition~3.2]{KM_AZAI2019} (see also \cite[Definition~3.1]{Sav_AWPA2018} for a version with non-symmetric $F$).

\begin{defi}\label{def:wreath}
	The \textit{affinized symmetric algebra} $\WR_n=\WR_n(F)$ of rank $n$ is the quotient of the free product
	\[
	\bfk[x_1,\ldots,x_n]\star F^{\otimes n} \star \bfk\frakS_n
	\]
	by the following relations:
	\begin{gather*}
		x_if = fx_i,\quad\tau_if = s_i(f)\tau_i \quad\text{  for all }f\in F^{\otimes n},\\
		\tau_ix_j = x_{s_i(j)}\tau_i-(\delta_{i,j}-\delta_{i+1,j})\Delta_{i,i+1},
	\end{gather*}
	where $\delta_{i,j}$ is the Kronecker symbol.
\end{defi}

\begin{rmq}
	The algebra in the definition above differs from the algebra in~\cite[Definition~3.2]{KM_AZAI2019} by the sign in the last relation.
	However, we could eliminate this difference if we replace $x_i$ by $-x_i$.
\end{rmq}

For any $f\in F$ and $1\leqslant r\leqslant n$, denote by $f_r$ the image of $1^{\otimes r-1}\otimes f\otimes 1^{\otimes n-r}\in F^{\otimes n}$ in $\WR_n$.
While it is not obvious that the natural map $F^{\otimes n}\to \WR_n$ is injective, this follows from the lemma below.

\begin{lm}[{\cite[Theorem 3.8]{KM_AZAI2019}}]\label{lem:basis-wreath}
	Let $B_F$ be a basis of $F$.
	The affinized symmetric algebra $\WR_n$ has the following basis:
	$$
	\left\{\tau_wx_1^{a_1}x_2^{a_2}\ldots x_n^{a_n}(f^{(1)})_1(f^{(2)})_2\ldots (f^{(n)})_n:~w\in \frakS_n,a_r\in \bbN,f^{(i)}\in B_F\right\}.
	$$
\end{lm}

Let us introduce a grading on $\WR_n$ by setting
\[
\deg \tau = 0,\qquad \deg x_i = 2,\qquad \deg f = \deg_F f.
\]
This makes $\WR_n$ into a graded algebra.
We write $\WR_n = \bigoplus_i \WR_n[i]$, where $\WR_n[i]$ is the subspace of degree $i$.

\begin{corr}\label{WreathDim}
	We have the following formula for the graded dimension of $\WR_n$:
	\[
	\sum_i t^i\dim\WR_n[i] = n!\left(\frac{P_t(F)}{1-t^2}\right)^n,
	\]
	where $P_t(F) = \sum_i t^i\dim F_i$ is the graded dimension of $F$.
\end{corr}

Let us describe a faithful representation of $\WR_n$.
The vector space ${\bf P}_n(F):=\bfk[x_1,\ldots,x_n]\otimes F^{\otimes n}$ admits several natural $\frakS_n$-actions.
First, there is an action permuting $x_i$'s and leaving $F^{\otimes n}$ intact; denote the operators on ${\bf P}_n(F)$ induced by the elementary transpositions by $s^X_1,\ldots, s^X_{n-1}$ .
Conversely, there is an action permuting components of $F^{\otimes n}$ without touching $x$'s; denote the operators on ${\bf P}_n(F)$ induced by the elementary transpositions by $s^f_1,\ldots, s^f_{n-1}$.
Set also $s_k=s^X_ks^f_k$, the operators $s_1,\cdots,s_{n-1}$ correspond to the diagonal $\frakS_n$-action, exchanging simultaneously $x_i$'s and the components of $F^{\otimes n}$.

\begin{lm}
	\label{lm:polrep-wreath}
	The algebra $\WR_n$ has a faithful representation in ${\bf P}_n(F)$ such that
	\begin{itemize}
		\item $x_i$ acts by multiplication by $x_i\in {\bf P}_n(F)$;
		\item $f\in F^{\otimes n}$ acts by multiplication by $f\in {\bf P}_n(F)$;
		\item $\tau_i$ acts by $s_i-\Delta_{i,i+1}\partial^X_i$, where $\partial^X_i:=\frac{1-s^X_i}{x_{i}-x_{i+1}}$ is the Demazure operator on $\bfk[x_1,\ldots,x_n]$.
	\end{itemize}
\end{lm}
\begin{proof}
	The formulas above yield a representation of $\WR_n$ by~\cite[Lemma~3.7]{KM_AZAI2019}.
	The faithfulness follows from the proof of~\cite[Theorem 3.8]{KM_AZAI2019}.
	Indeed, it is shown there that the basis of $\WR_n$ in \cref{lem:basis-wreath} act on ${\bf P}_n(F)$ by linearly independent operators. 
\end{proof}

\subsection{KLR algebras of curves}
Let $n\in \bbN_+$, and let $1^n$ be the partition of $n$ into $1$'s.
Then $\Sc_{1^n,1^n}\subset \Sc_n$ is a subalgebra; we denote it by $\cal R_n=\cal R_n^C$ and refer to it as the \textit{(torsion) KLR algebra of $C$} of degree $n$.
More explicitly, $\cal R_n$ is the convolution algebra $\cal A(\pi')=H_{G_n}^*(Z_n)$, where $\pi':Y_{1^n}\to Q_n$ is the restriction of $\pi$ (see \cref{sec:schur}) to $Y_{1^n}$.

\begin{nota}
In order to unclutter the notation, we will write $Z_n$ instead of $Z_{1^n,1^n}$ for the correspondence $Y_{1^n}\times_{Q_n}Y_{1^n}$ throughout this section.
\end{nota}

For any $1\leq i\leq n-1$, let us consider the ordered partition $\sigma_i$ of $n$ with $i$-th term equal to $2$, and other terms equal to $1$:
\[
\sigma_i = (1,\ldots,\underbrace{2}_{\text{$i$-th}},\ldots,1).
\]
We have a natural map $Y_{1^n}\to Y_{\sigma_i}$, induced by the $B_n$-equivariant embedding $\widetilde{Q}_{1^n}\subset \widetilde{Q}_{\sigma_i}$.
Consider the following correspondences:
\[
Z^i_n := Y_{1^n}\times_{Y_{\sigma_i}} Y_{1^n}, \qquad Z^{\sigma_i}_n:=\overline{Z^i_n\setminus Y_{1^n}}\subset Z_n.
\]

Let us denote $\tau_i:=[Z_n^{\sigma_i}]\in \cal R_n^C$.
Unlike for the full Schur algebra, we can use the notations from \cref{ConvEx} without any adjustments for $\cal R_n$.
We set $X = C^n$, $\Gamma = \fk{S}_n$, and $\gamma = \gamma_{1^n}$.
\cref{TautElts} then implies that $\Xi_n(P) = \widetilde{\xi}_{(1,P)}$ for any $P\in\bf{P}_n\simeq H^*_{G_n}(Y_{1^n})$.

\begin{prop}\label{CrossElts}
We have 
\[
\Xi_n(\tau_i)=\widetilde{\xi}_{1,\frac{\Delta_{i,i+1}}{x_{i+1}-x_i}}+\widetilde{\xi}_{s_i,1+\frac{\Delta_{i,i+1}}{x_{i+1}-x_i}}.
\]
\end{prop}
\begin{proof}
By definition of $Z^{\sigma_i}_n$, we have $Z^i_n = Y_{1^n}\cup Z^{\sigma_i}_n$.
In particular,
\[
\Xi_n(\tau_i) = \Xi_n([Z^i_n]) - \Xi_n([Y_{1^n}]) = \Xi_n([Z^i_n]) - \widetilde{\xi}_{(1,1)}.
\]
Since $Z^i_n = Y_{1^n}\times_{Y_{\sigma_i}} Y_{1^n}$, we have $[Z^i_n] = \rmS_{\sigma_i}^{1^n}\rmM^{\sigma_i}_{1^n}$ inside $\Sc_n$.
Using \cref{SMLoc} and formula~\eqref{SchurLocComp}, we obtain
\[
\Xi_n([Z_n^i]) = \xi^{1^n,\sigma_i}_{(1,\gamma_{\sigma_i})}*\xi^{\sigma_i,1^n}_{(1,\gamma_{\sigma_i})} = \xi_{(1,\gamma_{\sigma_i})} + \xi_{(s_i,\gamma_{\sigma_i})}.
\]
Since $\gamma_{\sigma_i}/\gamma_{1^n} = (x_{i+1}-x_i+\Delta_{i,i+1})/(x_{i+1}-x_i)$, we conclude that
\[
\Xi_n(\tau_i) = \widetilde{\xi}_{(1,\gamma_{\sigma_i}/\gamma_{1^n}-1)} + \widetilde{\xi}_{(s_i,\gamma_{\sigma_i}/\gamma_{1^n})} = \widetilde{\xi}_{1,\frac{\Delta_{i,i+1}}{x_{i+1}-x_i}}+\widetilde{\xi}_{s_i,1+\frac{\Delta_{i,i+1}}{x_{i+1}-x_i}}.
\]
\end{proof}
\begin{rmq}
It is possible to make this computation directly, without appealing to the results of \cref{sec:schur}.
For this, one can first show that $Z_n^{\sigma_i}$ is smooth (in effect, it is isomorphic to a certain blowup of $Y_{1^n}$), and then use \cref{LocPullPush}.
\end{rmq}

The algebra $\cal R_n$ acts on $\bf{P}_n = H^G_*(Y_{1^n})$ by \cref{ConvModule}.
This is a subrepresentation of $\Sc_n \curvearrowright P_n$, restricted to $\cal R_n\subset\Sc_n$.
$\bf{P}_n$ can be identified with a subspace in $(\bf{P}_n)_\loc$, on which $\cal R_n^\loc$ acts as in \cref{ConvEx}.
Under this identification, we have
\begin{align*}
\Xi_n(P).h &= \widetilde{\xi}_{(1,P)} . h = Ph\gamma/\gamma = Ph;\\
\Xi_n(\tau_i).h &= \left( \widetilde{\xi}_{1,\frac{\Delta_{i,i+1}}{x_{i+1}-x_i}}+\widetilde{\xi}_{s_i,1+\frac{\Delta_{i,i+1}}{x_{i+1}-x_i}} \right).h
=\frac{\Delta_{i,i+1}}{x_{i+1}-x_i}h+\left(1+\frac{\Delta_{i,i+1}}{x_{i+1}-x_i}\right)h^{s_i}\left(\frac{x_i-x_{i+1}+\Delta_{i,i+1}}{x_i-x_{i+1}-\Delta_{i,i+1}}\right)\\
&=\frac{\Delta_{i,i+1}}{x_{i+1}-x_i}h+\left(1-\frac{\Delta_{i,i+1}}{x_{i+1}-x_i}\right)h^{s_i}
= h^{s_i} + \frac{\Delta_{i,i+1}}{x_{i+1}-x_i}(h-h^{s_i})\\
& = (s_i-\Delta_{i,i+1}\partial_i)h,
\end{align*}
where we have used the fact that
\[
\frac{\gamma_{1^n}}{\gamma_{1^n}^{s_i}} = \frac{\gamma_{1^n}/\gamma_{\sigma_i}}{(\gamma_{1^n}/\gamma_{\sigma_i})^{s_i}} = -\frac{x_i-x_{i+1}+\Delta_{i,i+1}}{x_{i+1}-x_i+\Delta_{i,i+1}}.
\]

\begin{prop}\label{KLRisowreath}
We have an isomorphism of algebras $\cal{R}_n^C\simeq \WR_n(H^*(C))$.
\end{prop}
\begin{proof}
It follows from the proof of \cref{CrossElts} that $\tau_i = R_{1^n}^{1^n}(s_i)-1$ in notations of \cref{subs:diag-Schur}.
In particular, \cref{lm:basis-Schur-zigzag} implies that the algebra $\cal R_n$ is generated by polynomial operators together with $\tau_i$, $1\leq i\leq n-1$.
Both $\cal{R}_n^C$ and $\WR_n(H^*(C))$ act on ${\bf P}_n(F)$, and the action of polynomial operators and $\tau_i$'s is given by the same formulas.
Since the polynomial representation ${\bf P}_n(F)$ is faithful for both algebras by \cref{lm:polrep-wreath} and \cref{SchurPolRepFaith}, we deduce the desired isomorphism.
\end{proof}	

\subsection{Affine zigzag algebra}
Let us write out $\cal R^C_n$ for $C=\bbP^1$.
In this case $H^*(C)\simeq\bfk[c]/c^2$, and $\Delta_{i,j}=c_i+c_j$.

\begin{defi}
The \emph{affine zigzag algebra} $\Z_n$ is the $\bfk$-algebra generated by elements $x_r$, $c_r$, $1\leq r\leq n$ and $\tau_k$, $1\leq k < n$ modulo the following relations:
\begin{gather*}
	x_rx_t=x_tx_r,\qquad x_rc_t=c_tx_r,\qquad c_rc_t=c_tc_r,\qquad c_r^2=0;\\
	\tau_k^2=1,\qquad \tau_k\tau_{k+1}\tau_k=\tau_{k+1}\tau_k\tau_{k+1},\qquad \tau_k\tau_l=\tau_l\tau_k \mbox{ if }|l-k|>1;\\
	\tau_kc_k=c_{k+1}\tau_k,\qquad\tau_kc_r=c_r\tau_k \mbox{ if }r\ne k,k+1;\\
	\tau_kx_k-x_{k+1}\tau_k=-c_k-c_{k+1}=x_{k}\tau_k-\tau_kx_{k+1},\qquad \tau_kx_r=x_r\tau_k \mbox{ if }r\ne k,k+1.
\end{gather*}
\end{defi}

\begin{corr}\label{corr:KLRP1=zigzag}
	We have an isomorphism of algebras $\cal{R}_n^{\bbP^1}\simeq \Z_n$.
\end{corr}

Consider the following truncated polynomial ring:
\[
{\bf P}_n=\bfk[x_1,\ldots,x_n,c_1,\ldots,c_n]/(c_1^2,\ldots,c_n^2).
\]
The following lemma is a special case of \cref{lm:polrep-wreath,lem:basis-wreath}.
\begin{lm}
\label{lm:polrep-zigzag}
The affine zigzag algebra $\Z_n$ has the following basis:
\[
\left\{\tau_wx_1^{a_1}x_2^{a_2}\ldots x_n^{a_n}c_1^{b_1}c_2^{b_2}\ldots c_n^{b_n};~w\in \frakS_n,a_r\in \bbN,b_r\in\{0,1\}\right\}.
\]
Furthermore, the algebra $\Z_n$ has a faithful representation in ${\bf P}_n$ such that
\begin{itemize}
	\item $x_r\in \Z_n$ acts by multiplication by $x_r\in {\bf P}_n$, 
	\item $c_r\in \Z_n$ acts by multiplication by $c_r\in {\bf P}_n$,
	\item $\tau_k$ acts by $s_k-(c_k+c_{k+1})\partial_k$, where $\partial_k=\frac{1-s_k}{x_{k}-x_{k+1}}$ is the Demazure operator. 
\end{itemize}
\end{lm}

\begin{rmq}
	While the operator $\partial_k$ is not well-defined on ${\bf P}_n$, the operator $(c_k+c_{k+1})\partial_k$ is.
	We could also write $\partial^X_k$ instead of $\partial_k$ as in \cref{lm:polrep-wreath}.
\end{rmq}

\subsection{Other examples}
For $C=\bbC$ or $\bbC^*$, we have $Q_n\simeq \fk{gl}_n$ and $Q_n\simeq GL_n$ respectively, equipped with the adjoint action of $GL_n$.
We therefore recover Grothendieck-Springer resolution and its multiplicative version.
Furthermore, let $C=E$ be an elliptic curve, and $\ona{Bun}_{GL_n}^{0,ss}$ the stack of semistable $GL_n$-bundles of degree $0$ on $E$.
We have an equivalence of stacks $\T_n\simeq \ona{Bun}_{GL_n}^{0,ss}$ essentially due to Atiyah~\cite{Ati_VBEC1957,FMW_PBEC1998}.
Since it is compatible with embeddings of sheaves\footnote{that is, it is an equivalence of stacks, and not just of associated (algebraic) stacks in groupoids}, our map $\FT_n\to \T_n$ produces the same Steinberg variety that appears in the context of elliptic Springer theory~\cite{BN_EST2015} (for $G=GL_n$).
\section{Integral version for $\bbP^1$}\label{sec:P1int}
In this section we adapt some of the considerations above to homology with integral coefficients, when $C = \bbP^1$ is the projective line.

\subsection{Equivariant homology and localization}
The following analogue of \cref{GvsTW} holds for cohomology with integer coefficients.
\begin{prop}[{\cite[Theorem 2.10]{HS_TAEC2009}}]
Let $G=GL_n$, $T\subset G$ a maximal torus and $X$ a $G$-variety.
If $H^*_T(X,\bbZ)$ is torsion-free as an $H_T$-module, then we have an isomorphism $H_G^*(X,\bbZ)\simeq H_T^*(X,\bbZ)^{\fk{S}_n}$.
\end{prop}
In particular, we have $H_G^*(pt) = \bbZ[x_1,\ldots,x_n]^{\fk{S}_n}$.
As for localization, \cref{LocThm} holds over any coefficient ring $\bfk$ without $\bbZ$-torsion, in particular for $\bfk = \bbZ$.

\subsection{Integral homology of $\T_n(\bbP^1)$}
The proof of \cref{HofTor} uses the decomposition theorem for perverse sheaves in an essential way, therefore only works for cohomology with coefficients in $\bbQ$.
We expect that it remains true if we replace $\bbQ$ by $\bbZ$.
Here, we prove an analogous claim for the projective line.
Let us denote the base curve of $\T_n$ by superscript; thus, here we study $\T_n^{\bbP^1}$.
We also denote by $\T_n^x\subset \T_n^C$ the substack of sheaves supported on $x\in C$.
Note that we have 
\[
\T_n^\bbC \simeq [\fk{gl}_n/GL_n], \quad \T_n^x\simeq [\cal N_n/GL_n],
\]
where $\cal N_n\subset GL_n$ is the nilpotent cone.
\begin{prop}\label{prop:CohP1Z}
Let $C=\bbP^1$.
We have 
\[
H^*(\T_n,\bbZ)\simeq \ona{Sym}^n\left( H^*(\bbP^1,\bbZ)[x] \right) = \left( \bbZ[x_1,\ldots,x_n,c_1,\ldots,c_n]/(c_1^2,\ldots,c_n^2) \right)^{\fk{S}_n},
\]
where $\deg x_i=\deg c_i=2$.
\end{prop}
\begin{proof}
We will drop the coefficient ring from notations, and write $H^*(-) = H^*(-,\bbZ)$ throughout the proof for brevity.
Let us decompose $\bbP^1 = \bbC \sqcup \{\infty\}$, where $\infty\in\bbP^1$ is the point at infinity (or any other point).
This induces a stratification $S^n\bbP^1 = \bigsqcup S_i$, where $S_i \simeq S^{n-i} \bbC$ is the locally closed subvariety, consisting of tuples with $i$ occurrences of $\infty$.
Taking preimages under the support map, we get
\begin{equation}\label{eq:stratP1}
\T_n = \bigsqcup_i \T_n^{i\infty},\qquad \T_n^{i\infty} = \ona{supp}^{-1}(S_i).
\end{equation}
Note that we have isomorphisms of moduli spaces $\T_n^{i\infty}\simeq \T_i^\infty\times \T_{n-i}^\bbC$.
In particular, $\T_n^{i\infty} \simeq [(\cal N_i\times \fk{gl}_{n-i})/G_n^{(i)}] \simeq [Q_n^{(i)}/G_n]$, where $G_n^{(i)} = GL_i\times GL_{n-i}\subset G_n$, and $Q_n^{(i)}$ can be seen either as $G_n\times_{G_n^{(i)}}(\cal N_i\times \fk{gl}_{n-i})$, or a locally closed subvariety of $Q_n$ with prescribed supports.

Recall that $Q_n^{T_n} = (\bbP^1)^n$.
We have $Q_n^{T_n}\cap Q_n^{(i)} = \fk{S}_n\times_{\fk{S}_{(i,n-i)}}(\{\infty\}^i\times \bbC^{n-i})$, where $\fk{S}_n$ acts on $(\bbP^1)^n$ by permuting the factors.
Note that $\cal N_i\times \fk{gl}_{n-i}$ retracts to $\{\infty\}^i\times \bbC^{n-i}$.
Since equivariant cohomology is homotopy invariant, the pullback map $H^*_{T_n}(\cal N_i\times \fk{gl}_{n-i})\to H^*_{T_n}(\{\infty\}^i\times \bbC^{n-i})$ is an isomorphism.
In particular, we see that $H^*(\T_n^{i\infty}) = H^*_{G_n}(Q_n^{(i)}) = H^*_{G_n^{(i)}}(\cal N_i\times \fk{gl}_{n-i})$ is even.
Therefore the stratification~\eqref{eq:stratP1} defines a filtration on $H^*_{G_n}(Q_n)$ with associated graded $\bigoplus_i H^*_{G_n}(Q_n^{(i)})$; the same holds if we replace $G_n$ by $T_n$.

Let us introduce an auxillary grading $\bf{P}_n = \bigoplus_k \bf{P}_n^{(k)}$ by the degree of polynomials in $c_i$'s.
Since $c_i^2=0$ for any $i$, we see that $\bf{P}_n^{(k)} = 0$ for $k>n$.
Moreover, it is clear that
\[
\bf{P}_n^{(k)} = \bigoplus_{1\leq i_1<\ldots < i_k\leq n} \bbZ[x_1,\ldots, x_n]c_{i_1}\ldots c_{i_k}.
\]
We write $\bf{P}_n^{\leq k} = \bigoplus_{i=1}^k \bf{P}_n^{(i)}$; in particular, $\bf{P}_n^{\leq n} = \bf{P}_n$.

Consider the intersection $X_i = Q_n^{T_n}\cap Q_n^{(i)}$.
By the computations above, we have $X_i = \fk{S}_n\times_{\fk{S}_{(k,n-k)}}(\{\infty\}^k\times \bbC^{n-k})$ and $\overline{X}_i = \bigcup_{\sigma\in\fk{S}_n}\sigma(\{\infty\}^k\times (\bbP^1)^{n-k})$.
Therefore $H^*_{T_n}(\overline{X}_k) = \bf{P}_n^{\leq k}$ for any $k$, and an easy induction argument identifies the short exact sequence $H^*_{T_n}(\overline{X}_k)\to H^*_{T_n}(\overline{X}_{k-1})\to H^*_{T_n}(X_k)$ with $\bf{P}_n^{\leq k}\to \bf{P}_n^{\leq k-1}\to \bf{P}_n^{(k)}$.
With this substitution, localization to $T_n$-fixed points gives us the following commutative diagrams:
\[
\begin{tikzcd}
H_{T_n}^*(\overline{Q_n^{(k)}})\ar[r,hook]\ar[d] & H_{T_n}^*(\overline{Q_n^{(k-1)}})\ar[r,two heads]\ar[d] & H_{T_n}^*(Q_n^{(k)})\ar[d,equal]\\
\bf{P}_n^{\leq k}\ar[r,hook] & \bf{P}_n^{\leq k-1}\ar[r,two heads] & \bf{P}_n^{(k)}
\end{tikzcd}
\]
By an iterated application of 5-lemma for $k$ descending from $n$ to $1$, pullback along the inclusion $(\bbP^1)^n\hookrightarrow Q_n$ induces an identification
\[
H^*_{T_n}(Q_n) = \bf{P}_n = \bbZ[x_1,\ldots,x_n,c_1,\ldots,c_n]/(c_1^2,\ldots,c_n^2).
\]
Applying \cref{GvsTW}, we get $H^*(\T_n) = H^*_{T_n}(Q_n)^{\fk{S}_n}$, and so we may conclude.
\end{proof}

\begin{rmq}
Note that the retractions in the proof above do not come from the global $\bb G_m$-action on $\bbP^1$, since for any such action either $0$ or $\infty\in \bbP^1$ is a repellent point.
\end{rmq}

Applying universal coefficients, we obtain the following corollary.
\begin{corr}
For any ring $\bbk$, we have
\[
H^*(\T_n,\bbk) \simeq \left( \bbk[x_1,\ldots,x_n,c_1,\ldots,c_n] \right)^{\frakS_n}.
\]
\end{corr}

Let $\lambda\in \Comp(n)$.
Since $\FT_\lambda\to \T_\lambda$ is a stack vector bundle, formula \eqref{FlagCoh} shows that
\[
H^*(\FT_\lambda,\bbZ)\simeq \left( \bbZ[x_1,\ldots,x_n,c_1,\ldots,c_n]/(c_1^2,\ldots,c_n^2) \right)^{\fk{S}_\lambda}.
\]

Consider the composition
\[
[(\bbP^1)^n/T_n] \to [Q_n/T_n] \to [Q_n/G_n],
\]
where the first map is obtained from closed embedding $(\bbP^1)^n\hookrightarrow Q_n$, and the second map is given by restricting $G_n$-action on $Q_n$ to $T_n$.
This composition can be identified with the direct sum map
\[
\ona{\bigoplus} : (\T_1)^n\to \T_n ,\qquad (\cal F_1,\ldots,\cal F_n)\mapsto \cal F_1\oplus\ldots\oplus\cal F_n.
\]
Therefore, the isomorphism in \cref{prop:CohP1Z} can be regarded as being induced from pullback along $\bigoplus$.

\subsection{Tautological subring}
Recall that we have the universal sheaf $\cal E$ over $\T_n\times \bbP^1$.
Applying Künneth decomposition to the Chern classes of $\cal E$, we write
\[
c_i(\cal E) = c_{i,0}\otimes 1 + c_{i,1}\otimes p,
\]
where $c_{i,j}\in H^{2(i-j)}(\T_n,\bbZ)$, and $p\in H^2(\bbP^1,\bbZ)$ is the class of a point.

\begin{exe}\label{ex:n=1}
For $n=1$, we have $H^*(\T_1,\bbZ) = \bbZ[x,c]/c^2$ and $\cal E = x\Ocal_\Delta$.
Note that under our identifications, $\Delta = c+p$.
The total Chern class of $\cal E$ is given by
\begin{equation}\label{eq:totChern1pt}
c(\cal E) = c(x)/c(x\Ocal(-\Delta)) = \frac{1+x}{1+x-\Delta} = 1+(c+p)+\sum_{i\geq 1} (-x)^{i-1}((2ic-x)p-xc).
\end{equation}
This allows us to express the Künneth components as follows:
\[
c_{i,0} = (-x)^{i-1}c,\qquad c_{i,1} = (-x)^{i-2}(2(i-1)c-x).
\]
In particular, $c = c_{1,0}$ and $x = 2c_{1,0}- c_{2,1}$.
\end{exe}

\begin{defi}
The tautological ring $TH^*(\T_n,\bbZ)$ is the subring of $H^*(\T_n,\bbZ)$ generated by classes $c_{i,0}$, $c_{i,1}$, $i\in \bbZ_{>0}$.
\end{defi}

If we work with $\bbQ$-coefficients, this definition is superfluous, since Künneth-Chern classes generate $H^*(\T_n,\bbQ)$ as a ring; see~\cite{Hei_CMSC2012}.
The following example shows that this is not the case over $\bbZ$.

\begin{exe}\label{ex:d2nonsurj}
Let $n=2$.
By the universal property of $\cal E$, its pullback under $\bigoplus$ is isomorphic to the direct sum of universal sheaves $\cal E_1\oplus\cal E_2$.
Since the total Chern class is functorial under pullbacks, we have $c(\cal E) = c(\cal E_1)c(\cal E_2)$.
Using the formula \eqref{eq:totChern1pt}, we obtain
\begin{align*}
c(\cal E)	& =  \left( 1+ (c_1+p) + (2c_1p-x_1(c_1+p)) + ((c_1+p)x_1^2-4c_1px_1) +\cdots \right)\\
			&\hspace{1cm}\times \left( 1+ (c_2+p) + (2c_2p-x_2(c_2+p)) + ((c_2+p)x_2^2-4c_2px_2) +\cdots \right)\\
			&= \left(1 + (c_1+c_2) + (c_1c_2-c_1x_1-c_2x_2) + \ldots\right)  \\
			&\hspace{1cm}+ p\left(2+(3c_1+3c_2-x_1-x_2)+(x_1^2+x_2^2+4(c_1c_2-c_1x_1-c_2x_2)-(c_1+c_2)(x_1+x_2))+\ldots\right).
\end{align*}
As a consequence, we get the following expressions for the first few Künneth-Chern classes:
\begin{gather*}
c_{1,0} = c_1+c_2, \qquad c_{2,0} = c_1c_2-c_1x_1-c_2x_2,\\
c_{2,1} = 3c_{1,0} - (x_1+x_2), \qquad c_{3,1} = x_1^2+x_2^2 -(x_1+x_2)c_{1,0} + 4c_{2,0}.
\end{gather*}
By definition, $TH^4(\T_2,\bbZ)$ is spanned as a $\bbZ$-module by $c_{1,0}^2$, $c_{2,1}^2$, $c_{1,0}c_{2,1}$, $c_{2,0}$, $c_{3,1}$.
Using the formulae above, it is easy to check that this sublattice does not contain either $c_1c_2$ or $x_1x_2$; however, we have
\[
2c_1c_2 = c_{1,0}^2,\qquad 2x_1x_2 = 6c_{1,0}^2 - 5c_{1,0}c_{2,1} + c_{2,1}^2 - c_{3,1} + 4c_{2,0}.
\]
\end{exe} 

A similar computation shows that Künneth-Chern classes fail to generate $H^*(\T_n^C,\bbZ)$ for any smooth projective curve $C$, assuming that an analogue of \cref{prop:CohP1Z} holds.

\subsection{Homology of Steinberg is torsion-free}\label{subs:P1anyring}
The proof of \cref{KLRLocInj} works verbatim for $C=\bbP^1$ over any ring $\bbk$, except we need to replace purity considerations for splitting long exact sequences by parity of homology groups.
The localization theorem can still be applied by \cref{rmq:locthmcoef}, since classes $\gamma_\lambda$ are not zero divisors over any $\bbk$ by formula~\eqref{NorBunFlag}.
In particular, the localization map $\Xi_n$ remains injective.

The rest of \cref{sec:schur,sec:KLRcurve}, notably the proof of \cref{lm:basis-Schur-zigzag}, goes through for any $\bbk$ without changes.
Note that even though some denominators appear in intermediate computations (essentially because of coefficients in~\eqref{SchurLocComp}), all modules under consideration are free, so one can perform computations over $\bbZ\subset\bbQ$, obtain a result valid over $\bbZ$, and then change the base ring.
One could get rid of denominators altogether, but it would render the notations even more cumbersome.

In particular, we obtain that \cref{corr:KLRP1=zigzag} holds over any $\bfk$.
\section{KLR algebras of quivers}\label{sec:KLRquivers}

\subsection{KLR algebra}
Let $\Gamma=(I,H)$ be a quiver without loops, where $I$ stands for the set of vertices, and $H$ for the set of arrows.
It comes equipped with source and target maps $s,t:H\to I$.
For any $i,j\in I$, let $h_{ij}$ be the number of arrows from $i$ to $j$, and define
\[
Q_{ij}(u,v)=
\left\{\begin{array}{ll}
0& \mbox{ if }i=j,\\
(u-v)^{h_{ij}}(v-u)^{h_{ji}}& \mbox{ otherwise}.
\end{array}\right.
\]

Let $\alpha=(n_i)_{i\in I}\in \bbZ_{\geqslant 0}^I$ be a dimension vector.
It can be alternatively written as a sum $\alpha=\sum_{i\in I}n_i\alpha_i$, where $\alpha_i = (\delta_{ij})_{j\in I}$, $i\in I$.
For a dimension vector $\alpha$, we define $|\alpha|:=\sum_{i\in I}n_i$ and
\[
I^\alpha=\left\{\ui=(i_1,i_2,\ldots,i_{|\alpha|})\in I^{|\alpha|} : \sum_{r=1}^{|\alpha|}\alpha_{i_r}=\alpha\right\}.
\]
We will also make use of the following refinement of $I^\alpha$:
\[
I^{(\alpha)} = \left\{\ui=(i_1^{(a_1)},i_2^{(a_2)},\ldots, i_k^{(a_k)}) : a_r\in\bbZ_{>0},\sum_{r=1}^{|\alpha|}a_r\alpha_{i_r}=\alpha\right\},
\]
where we treat $i^{(a)}$ as a formal symbol corresponding to divided powers (see \cref{subs:divided-power}).
To any $\ui\in I^{(a)}$ we can associate $\overline{\ui} = (i_1^{a_1},i_2^{a_2},\ldots, i_k^{a_k}) \in I^{\alpha}$, obtained by replacing each divided power $i_r^{(a_r)}$ with $a_r$ consecutive copies of $i_r$.
For each $\ui=(i_1^{(a_1)},i_2^{(a_2)},\ldots, i_k^{(a_k)})\in I^{(\alpha)}$, let $\frakS_\ui$ be the subgroup of $\frakS_{|\alpha|}$ associated to $(a_1,\ldots,a_k)\in\Comp(|\alpha|)$:
\[
\frakS_\ui=\frakS_{a_1}\times \frakS_{a_2}\times\ldots\times\frakS_{a_k}\subset \frakS_{|\alpha|}.
\]

\begin{defi}
A \emph{KLR diagram} of weight $\alpha$ is a planar diagram containing $|\alpha|$ strands such that:
\begin{itemize}
\item the strands connect $|\alpha|$ points on one horizontal line with $|\alpha|$ points on another horizontal line, each strand goes from bottom to top;
\item each strand is labeled by an element of $I$;
\item for each $i\in I$, there are $n_i$ strands with label $i$;
\item two strands are allowed to cross, and there are no triple-crossings;
\item a piece of a strand is allowed to carry a dot, a dot cannot collide with a crossing.
\end{itemize}
\end{defi}
We consider KLR diagrams modulo isotopies.
In particular, a dot is allowed to move freely along the strand, as long as it doesn't slide past a crossing. 

\begin{defi}
The \textit{KLR algebra} $R(\alpha)$ is the $\bfk$-algebra generated by KLR diagrams of weight $\alpha$ modulo the local relations below:
\begin{equation}\label{rel:KLR1}
\begin{aligned}
\begin{tikzpicture}[scale=0.6]
	\draw[thick](-4,0) +(-1,-1) -- +(1,1) node[below,at start]
	{$i$}; \draw[thick](-4,0) +(1,-1) -- +(-1,1) node[below,at
	start] {$j$}; \fill (-4.5,.5) circle (5pt);
	\node at (-2,0){=}; \draw[thick](0,0) +(-1,-1) -- +(1,1)
	node[below,at start] {$i$}; \draw[thick](0,0) +(1,-1) --
	+(-1,1) node[below,at start] {$j$}; \fill (.5,-.5) circle (5pt);
	\node at (4,0){if $i\neq j$,};
\end{tikzpicture}
\end{aligned}
\end{equation}

\begin{equation}\label{rel:KLR2}
\begin{aligned}
\begin{tikzpicture}[scale=0.6,thick]
	\draw[thick](-4,0) +(-1,-1) -- +(1,1) node[below,at start]
	{$i$}; \draw[thick](-4,0) +(1,-1) -- +(-1,1) node[below,at
	start] {$i$}; \fill (-4.5,.5) circle (5pt);
	\node at (-2,0){$=$}; \draw[thick](0,0) +(-1,-1) -- +(1,1)
	node[below,at start] {$i$}; \draw[thick](0,0) +(1,-1) --
	+(-1,1) node[below,at start] {$i$}; \fill (.5,-.5) circle (5pt);
	\node at (2,0){\color{red}$-$}; \draw[thick](4,0) +(-1,-1) -- +(-1,1)
	node[below,at start] {$i$}; \draw[thick](4,0) +(0,-1) --
	+(0,1) node[below,at start] {$i$}; \node at (4.5,-0.25){,};
\end{tikzpicture}
\qquad
\begin{tikzpicture}[scale=0.6,thick]
	\draw[thick](-4,0) +(-1,-1) -- +(1,1) node[below,at start]
	{$i$}; \draw[thick](-4,0) +(1,-1) -- +(-1,1) node[below,at
	start] {$i$}; \fill (-4.5,-.5) circle (5pt);
	\node at (-2,0){$=$}; \draw[thick](0,0) +(-1,-1) -- +(1,1)
	node[below,at start] {$i$}; \draw[thick](0,0) +(1,-1) --
	+(-1,1) node[below,at start] {$i$}; \fill (.5,.5) circle (5pt);
	\node at (2,0){\color{red}$-$}; \draw[thick](4,0) +(-1,-1) -- +(-1,1)
	node[below,at start] {$i$}; \draw[thick](4,0) +(0,-1) --
	+(0,1) node[below,at start] {$i$}; \node at (4.5,-0.25){,};
\end{tikzpicture}
\end{aligned}
\end{equation}

\begin{equation}\label{rel:KLR3}
\begin{aligned}
\begin{tikzpicture}[scale=0.6, thick]  
	\draw (2,0)  +(0,-1) .. controls (3.6,0) ..  +(0,1)  node[below,at start]{$i$};
	\draw (3.6,0)  +(0,-1) .. controls (2,0) ..  +(0,1)  node[below,at start]{$j$} ; 
	\node at (4.5,0){$=$}; 
	\draw (7.8,0) +(0,-1) -- +(0,1) node[below,at start]{$j$};
	\draw (7,0) +(0,-1) -- +(0,1) node[below,at start]{$i$};
	\node[inner xsep=10pt,fill=white,draw,inner ysep=7pt] at (7.4,0) {$Q_{ij}(y_1,y_2)$};
	\node at (11,0) {if $i\neq j$,};
\end{tikzpicture}
\qquad
\begin{tikzpicture}[scale=0.6, thick]  
	\draw (-4,0)  +(0,-1) .. controls (-2.4,0) ..  +(0,1) node[below,at start]{$i$}; 
	\draw (-2.4,0)  +(0,-1) .. controls (-4,0) ..  +(0,1)  node[below,at start]{$i$}; 
	\node at (-1.5,0){$=0,$};
\end{tikzpicture} 	
\end{aligned}
\end{equation}

\begin{equation}\label{rel:KLR4}
\begin{aligned}
\begin{tikzpicture}[thick,scale=0.6]
	\draw (-3,0) +(1,-1) -- +(-1,1) node[below,at start]{$k$}; 
	\draw (-3,0) +(-1,-1) -- +(1,1) node[below,at start]{$i$};
	\draw (-3,0) +(0,-1) .. controls (-4,0) ..  +(0,1) node[below,at
	start]{$j$}; \node at (-1,0) {=}; 
	\draw (1,0) +(1,-1) -- +(-1,1) node[below,at start]{$k$}; 
	\draw (1,0) +(-1,-1) -- +(1,1) node[below,at start]{$i$}; 
	\draw (1,0) +(0,-1) .. controls (2,0) ..  +(0,1) node[below,at start]{$j$}; \node at (5,0)
	{unless $i=k\neq j$,};
\end{tikzpicture}	
\end{aligned}
\end{equation}

\begin{equation}\label{rel:KLR5}
\begin{aligned}
\begin{tikzpicture}[thick,scale=0.6]
	\draw (-3,0) +(1,-1) -- +(-1,1) node[below,at start]{$i$}; 
	\draw (-3,0) +(-1,-1) -- +(1,1) node[below,at start]{$i$}; 
	\draw (-3,0) +(0,-1) .. controls (-4,0) ..  +(0,1) node[below,at start]{$j$}; \node at (-1,0) {=}; 
	\draw (1,0) +(1,-1) -- +(-1,1) node[below,at start]{$i$}; 
	\draw (1,0) +(-1,-1) -- +(1,1) node[below,at start]{$i$}; 
	\draw (1,0) +(0,-1) .. controls (2,0) ..  +(0,1) node[below,at start]{$j$}; \node at (3,0){\color{red}$-$};        
	\draw (6.2,0)+(1,-1) -- +(1,1) node[below,at start]{$i$}; 
	\draw (6.2,0)+(-1,-1) -- +(-1,1) node[below,at start]{$i$}; 
	\draw (6.2,0)+(0,-1) -- +(0,1) node[below,at start]{$j$};
	\node[inner ysep=8pt,inner xsep=5pt,fill=white,draw,scale=.6] at (6.2,0){$\displaystyle \frac{Q_{ij}(y_3,y_2)-Q_{ij}(y_1,y_2)}{y_3-y_1}$};
	\node at (10.5,0) {if $i\neq j$.};
\end{tikzpicture}		
\end{aligned}
\end{equation}
The multiplication is given by vertical concatenation; we impose the concatenation of strands with different labels to be zero.
\end{defi}

\begin{rmq}
	We will be chiefly interested in KLR algebras for the Kronecker quiver $\Gamma = (1\rightrightarrows 0)$. 
	In this case we have
	\[
	Q_{01}(u,v)=Q_{10}(u,v)=(u-v)^2,\qquad \frac{Q_{01}(y_3,y_2)-Q_{01}(y_1,y_2)}{y_3-y_1}=y_1-2y_2+y_3.
	\]
\end{rmq}

For each $\ui\in I^\alpha$ we have an idempotent $1_\ui$, given by a diagram consisting of $|\alpha|$ vertical strands, with $r$-th strand labeled by $i_r$ for any $r$.
The algebra $R(\alpha)$ is clearly generated by these idempotents, together with single crossings and dots.
In what follows, we will denote the crossing between $r$-th and $(r+1)$-th strand by $\psi_r$, and the diagram with a single dot on $r$-th strand by $y_r$.
More precisely, we have 
\begin{equation*}
\begin{aligned}
	\tikz[thick,xscale=.25,yscale=.25]{
	\ntxt{-3}{1.5}{$y_r1_\ui=$}
	\draw (0,0) -- (0,4);
	\ntxt{0}{-1.5}{$i_1$}
	\ntxt{4}{1.5}{$\ldots$}
	\draw (8,0) -- (8,4);
	\strdot{8}{2}
	\ntxt{8}{-1.5}{$i_r$}
	\ntxt{12}{1.5}{$\ldots$}
	\draw (16,0) -- (16,4);
	\ntxt{16}{-1.5}{$i_{|\alpha|}$}	
	\ntxt{17.5}{1.5}{,}
	}
\qquad\quad
	\tikz[thick,xscale=.25,yscale=.25]{
	\ntxt{-3}{1.5}{$\psi_r1_\ui=$}
	\draw (0,0) -- (0,4);
	\ntxt{0}{-1.5}{$i_1$}
	\ntxt{4}{1.5}{$\ldots$}
	\draw (8,0) -- (12,4);
	\draw (8,4) -- (12,0);
	\ntxt{8}{-1.5}{$i_r$}
	\ntxt{12}{-1.5}{$i_{r+1}$}
	\ntxt{16}{1.5}{$\ldots$}
	\draw (20,0) -- (20,4);
	\ntxt{20}{-1.5}{$i_{|\alpha|}$}	
	\ntxt{21.5}{1.5}{,}
	}
\end{aligned}
\end{equation*}	
and
\[
y_r=\sum_{\ui\in I^\alpha}y_r1_\ui,\qquad\qquad \psi_r=\sum_{\ui\in I^\alpha}\psi_r1_\ui.
\]

For example, relations~(\ref{rel:KLR1}-\ref{rel:KLR2}) take the following form:
\[
y_r\psi_r = \psi_r y_{r+1} - \sum_{\ui\in I^\alpha,\text{ } i_r=i_{r+1}} 1_\ui,\qquad \psi_r y_r = y_{r+1} \psi_r - \sum_{\ui\in I^\alpha,\text{ } i_r=i_{r+1}} 1_\ui.
\]

\subsection{Polynomial representation of $R(\alpha)$}
Let $m=|\alpha|$, and define $\Pol_m=\bfk[y_1,\ldots,y_m]$.
Let further $\Pol_\alpha$ be the direct sum of $I^\alpha$-worth copies of $\Pol_m$.
We write $\Pol_\alpha=\bigoplus_{\ui\in I^\alpha}\Pol_m 1_\ui$, where $1_\ui$ is the idempotent projecting to the $\ui$-th copy.

\begin{lm}[{\cite[\S 3.2.2]{Rou_2A2008}}]
\label{lem:polrep-KLR}
The algebra $R(\alpha)$ has a faithful representation on the vector space $\Pol_\alpha$, such that $1_\ui\in R(\alpha)$ acts by the projector $1_\ui$, $y_r\in R(\alpha)$ acts by multiplication by $y_r$, and for any $f\in \Pol_m$ we have
\begin{equation}
\label{eq:polrep-KLR}
\psi_r\cdot f1_\ui=
\begin{cases}
-\partial_r(f)1_\ui &\mbox{ if }i_r=i_{r+1},\\
P_{i_r,i_{r+1}}(y_r,y_{r+1})s_r(f)1_{s_r(\ui)}& \mbox{ else.}\\
\end{cases}
\end{equation}
Here $P_{ij}(u,v)=(u-v)^{h_{ij}}$, and $\partial_r=\frac{1-s_r}{y_r-y_{r+1}}$ is the Demazure operator.
\end{lm}

\subsection{Geometric construction of KLR algebras}\label{subs:geom-KLR}
Fix a dimension vector $\alpha=\sum_{i\in I}n_i\alpha_i$ with $|\alpha|=m$.
Let $V$ be an $I$-graded complex vector space of dimension $\alpha$, that is a complex vector space with decomposition $V=\bigoplus_{i\in I}V_i$, such that $\dim V_i=n_i$.
Consider the variety $E_\alpha=\bigoplus_{h\in H}{\rm Hom}(V_{s(h)},V_{t(h)})$, on which we have a natural action of $G_\alpha=\prod_{i\in I}GL(V_i)$.
For $\ui=(i_1^{(a_1)},i_2^{(a_2)},\ldots, i_k^{(a_k)})\in I^{(\alpha)}$, let $\bfF_\ui$ be the variety of flags in $V$
$$
\phi=(\{0\}=V^0\subset V^1\subset\cdots\subset V^k=V),
$$
which are homogeneous with respect to the decomposition $V=\bigoplus_{i\in I}V_i$, and for each $1\leq r\leq k$ the graded dimension of $V^r/V^{r-1}$ is equal to $a_r\alpha_{i_r}$.
Further, let $\widetilde{\bfF}_{\ui}$ be the following variety of pairs:
\[
\widetilde{\bfF}_{\ui} = \left\{ (x,\phi)\in E_\alpha\times \bfF_\ui : x(V^r)\subset V^r, 0\leq r\leq k \right\}.
\]
Analogously to~\cref{subs:Flag}, we have an isomorphism $H^*_{G_\alpha}(\bfF_\ui)\simeq\Pol_m^{\frakS_\ui}$, where for each $r\in[1;k]$ the elements $y_{a_1+\ldots+a_{r-1}+1},y_{a_1+\ldots+a_{r-1}+2},\ldots,y_{a_1+\ldots+a_{r}}$ are the Chern roots of the vector bundle $V^r/V^{r-1}$.
Since $\widetilde{\bfF}_{\ui}$ is a vector bundle over ${\bfF}_{\ui}$, we also have $H^*_{G_\alpha}(\widetilde\bfF_\ui)\simeq\Pol_m^{\frakS_\ui}$.

We also denote $\bfF_\alpha=\coprod_{\ui\in I^\alpha}\bfF_\ui$, $\widetilde \bfF_\alpha=\coprod_{\ui\in I^\alpha}\widetilde \bfF_\ui$.
Let $\pi_\alpha:\widetilde{\bfF}_{\alpha}\to E_\alpha$ be the natural projection, that is $\pi_\alpha(x,\phi) = x$, and consider the corresponding fiber product $\bfZ_\alpha = \widetilde{\bfF}_{\alpha}\times_{E_\alpha} \widetilde{\bfF}_{\alpha}$.
We have
\begin{equation*}
\bfZ_{\alpha}=\coprod_{\ui,\uj\in I^\alpha}\bfZ_{\ui,\uj} = \coprod_{\ui,\uj\in I^\alpha} \widetilde \bfF_\ui\times_{E_\alpha} \widetilde \bfF_\uj.
\end{equation*}
In other words, $\bfZ_{\ui,\uj}$ is the variety of triples $(x,\phi_1,\phi_2)\in E_\alpha\times \bfF_{\ui}\times \bfF_{\uj}$, such that $x$ preserves both $\phi_1$ and $\phi_2$.

\begin{rmq}
For now, we only consider $\ui\in I^\alpha$; however, the definition of $\bfZ_{\ui,\uj}$ makes sense for $\ui,\uj\in I^{(\alpha)}$ as well.
We will make use of these more general varieties in \cref{subs:geom-div-powers}.
\end{rmq}

By \cref{ConvProduct,ConvModule}, we have an algebra structure on $\cal A(\pi_\alpha) = H^*_{G_\alpha}(\bfZ_\alpha,\bfk)$, and an action of it on $H^*_{G_\alpha}(\widetilde \bfF_\alpha,\bfk)$.
The following statement is proved in~\cite{VV_CBK2011,Rou_2A2008} for $\bfk$ a field of characteristic zero, and in~\cite{Mak_CBKA2015} for an arbitrary ring $\bfk$ of finite global dimension.
\begin{prop}
\label{prop:geom-KLR}
The KLR algebra $R(\alpha)$ is isomorphic to the convolution algebra $H^*_{G_\alpha}(\bfZ_\alpha)$.
Moreover, the representation $\Pol_\alpha$ of $R(\alpha)$ is isomorphic to the representation $H^*_{G_\alpha}(\widetilde \bfF_\alpha)$ of $H^*_{G_\alpha}(\bfZ_\alpha)$.
\end{prop}
\begin{rmq}
\label{rem:geom-KLR-idemp}
The idempotents $1_\ui\in R(\alpha)$ correspond to different connected components.
Namely, we have
\[
H^*_{G_\alpha}(\bfZ_{\ui,\uj}) \simeq 1_\ui R(\alpha) 1_\uj,\qquad H^*_{G_\alpha}(\widetilde \bfF_\ui)\simeq \Pol_m1_\ui.
\]
\end{rmq}

\subsection{Divided powers}\label{subs:divided-power}
The algebra $R(n\alpha_i)$ is known as the \emph{nil-Hecke algebra} of rank $n$.
Let $w=s_{i_1}\ldots s_{i_r}$ be a reduced decomposition of $w\in\fk{S}_n$.
Relation~\eqref{rel:KLR4} implies that the product $\psi_{i_1}\ldots\psi_{i_r}$ is independent of the decomposition above.
We denote this element of $R(n\alpha_i)$ by $\psi_w$.

Let $w_{0,n}$ be the longest element in $\fk{S}_n$.
Define $y_{0,n}:=y_n^{n-1}y_{n-1}^{n-2}\ldots y_3^2y^{\vphantom{1}}_2$, and $1_{i^{(n)}}:=\psi_{w_{0,n}}y_{0,n}\in R(n\alpha_i)$.
It is easy to check that $\psi_{w_{0,n}}y_{0,n}\psi_{w_{0,n}} = \psi_{w_{0,n}}$, which implies that the element $1_{i^{(n)}}$ is an idempotent.
We call $1_{i^{(n)}}$ the \textit{divided difference idempotent}.
Note that it acts on $\Pol_{n\alpha_i}\simeq \Pol_n$ as the projector to symmetric polynomials.

Let $\alpha\in \bbZ_{\geqslant 0}^I$ be as in~\cref{subs:geom-KLR}. 
For each $\ui=(i_1^{(a_1)},i_2^{(a_2)},\ldots, i_k^{(a_k)})\in I^{(\alpha)}$, let $w_{0,\ui}$ be the maximal length element in $\frakS_\ui$, and define the following elements in $R(a_1\alpha_{i_1})\otimes \ldots \otimes R(a_k\alpha_{i_k})\subset R(\alpha)$:
\[
1_\ui:=1_{i_1^{(a_1)}}\otimes 1_{i_2^{(a_2)}}\otimes\ldots \otimes 1_{i_r^{(a_k)}},\qquad y_\ui:=y_{0,a_1}\otimes y_{0,a_2}\otimes\ldots \otimes y_{0,a_k}.
\]
The definitions in nil-Hecke algebra imply that $y_\ui=y_\ui 1_{\overline\ui}$ and $1_\ui=\psi_{w_{0,\ui}} y_\ui$.

In graphical calculus we draw divided power idempotents as boxes.
For example, the idempotent $1_\ui$ with $\ui=(i_1^{(a_1)},i_2^{(a_2)},\ldots, i_k^{(a_k)})$ is depicted as follows (we have $a_r$ strands with label $i_r$):
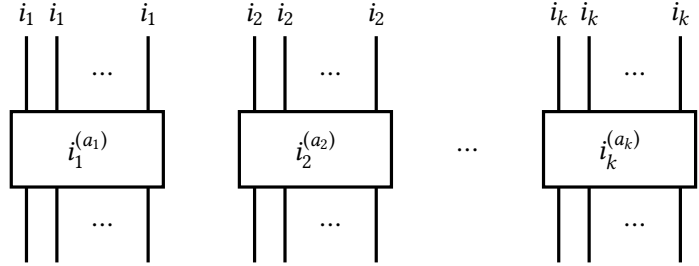
\[
\tikz[very thick,baseline={([yshift=+.6ex]current bounding box.center)}]{

\draw (0.2,0) -- (0.2,1);
\draw (0.6,0) -- (0.6,1);
\node at (1.2,0.5){$\ldots$};
\draw (1.8,0) -- (1.8,1);
    
\draw (0.2,2) -- (0.2,3) node[above]{$i_1$};
\draw (0.6,2) -- (0.6,3) node[above]{$i_1$};
\node at (1.2,2.5){$\ldots$};
\draw (1.8,2) -- (1.8,3) node[above]{$i_1$};    
    
\draw (0,1) rectangle (2,2);
\node at (1,1.5) {$i_1^{(a_1)}$};

\draw (3.2,0) -- (3.2,1);
\draw (3.6,0) -- (3.6,1);
\node at (4.2,0.5){$\ldots$};
\draw (4.8,0) -- (4.8,1);
    
\draw (3.2,2) -- (3.2,3) node[above]{$i_2$};
\draw (3.6,2) -- (3.6,3) node[above]{$i_2$};
\node at (4.2,2.5){$\ldots$};
\draw (4.8,2) -- (4.8,3) node[above]{$i_2$};    
    
\draw (3,1) rectangle (5,2);
\node at (4,1.5) {$i_2^{(a_2)}$};

\node at (6,1.5){$\ldots$};

\draw (7.2,0) -- (7.2,1);
\draw (7.6,0) -- (7.6,1);
\node at (8.2,0.5){$\ldots$};
\draw (8.8,0) -- (8.8,1);
    
\draw (7.2,2) -- (7.2,3) node[above]{$i_k$};
\draw (7.6,2) -- (7.6,3) node[above]{$i_k$};
\node at (8.2,2.5){$\ldots$};
\draw (8.8,2) -- (8.8,3) node[above]{$i_k$};    
    
\draw (7,1) rectangle (9,2);
\node at (8,1.5) {$i_k^{(a_k)}$};       
}   
\]

\subsection{Basis}
\label{subs:basis-KLR}
For each $w\in\frakS_{m}$ fix a reduced decomposition $w=s_{r_1}\ldots s_{r_k}$, and let
\[
\psi_w1_\ui=\psi_{r_1}\ldots \psi_{r_k}1_\ui,\qquad \psi_w=\sum_{\ui\in I^{\alpha}}\psi_w1_\ui.
\]
Unlike the case of nil-Hecke algebra, this definition of $\psi_w 1_\ui$ does depend on the choice of the decomposition.
Note that we allow to choose different reduced decompositions of the same $w$ for different $\ui$.

For any $\ui=(i_1,i_2,\ldots,i_m)\in I^m$ and $w\in\frakS_m$, set $w(\ui)=(i_{w^{-1}(1)},i_{w^{-1}(2)},\ldots,i_{w^{-1}(m)})$.
For any $\ui,\uj\in I^{\alpha}$, let $\doubleS{\uj}{\ui}=\{w\in \frakS_{m}: w(\ui)=\uj\}$.
More generally if $\ui,\uj\in I^{(\alpha)}$, define $\doubleS{\uj}{\ui}$ to be the set of shortest representatives of cosets in $\frakS_\uj\backslash\doubleS{\overline\uj}{\overline\ui}/\frakS_\ui$.
The following lemma is proved in~\cite[Theorem~3.7]{Rou_2A2008}, see also~\cite[Theorem~2.5]{KL_DACQ2009}.
\begin{lm}
\label{lem:basis-KLR}
For any $\ui,\uj\in I^{\alpha}$, each of the following two sets forms a basis of $1_\uj R(\alpha) 1_\ui$:
\[
\{y_1^{a_1}y_2^{a_2}\ldots y_n^{a_n}\psi_w1_\ui;~w\in \doubleS{\uj}{\ui},a_r\in \bbZ_{\geqslant 0}\},\qquad \{\psi_w y_1^{a_1}y_2^{a_2}\ldots y_n^{a_n}1_\ui;~w\in \doubleS{\uj}{\ui},a_r\in \bbZ_{\geqslant 0}\}.
\]
\end{lm}

\begin{rmq}\label{rmq:basis-KLR}
As we explained above, the definition of $\psi_w$ depends on some choices.
However, the following vector subspaces of $R(\alpha)$ always remain the same:
\[
R(\alpha)^{\leqslant w}=\bigoplus_{w'\leqslant w}\Pol_\alpha \psi_{w'},\qquad R(\alpha)^{<w}=\bigoplus_{w'< w}\Pol_\alpha \psi_{w'}.
\]
Moreover, they are stable by multiplication by elements of $\Pol_\alpha$ on the right and on the left.
The image of $\psi_w$ in $R(\alpha)/R(\alpha)^{<w}$ is also independent of our choices. 
\end{rmq}

Above we described some bases in $1_\uj R(\alpha)1_\ui$ for $\ui,\uj\in I^\alpha$.
However, we will need a version of \cref{lem:basis-KLR} which allows $\ui$, $\uj$ to lie in $I^{(\alpha)}$.
Let us begin with some preparations.
For each $\ui=(i_1^{(a_1)},i_2^{(a_2)},\ldots, i_k^{(a_k)})\in I^{(\alpha)}$ and $x\in\frakS_k$, let
\[
x^{-1}(\ui)=\left(i_{x(1)}^{\left(a_{x(1)}\right)},i_{x(2)}^{\left(a_{x(2)}\right)},\ldots, i_{x(k)}^{\left(a_{x(k)}\right)}\right).
\]
\begin{defi}
\label{def:adapted-perm}
Let $\ui,\uj\in I^{(\alpha)}$, $\ui=(i_1^{(a_1)},i_2^{(a_2)},\ldots, i_k^{(a_k)})$, $\uj=(j_1^{(b_1)},j_2^{(b_2)},\ldots, j_k^{(b_k)})$.
We say that $\uj$ is a \textit{permutation} of $\ui$ if there exists $x\in \frakS_k$ such that $x(\ui)=\uj$.
Each such $x$ induces a permutation $w\in \doubleS{\uj}{\ui}$.

Note that each reduced decomposition $x=s_{r_1}s_{r_2}\ldots s_{r_t}$ of $x$ induces a decomposition $w=\hat s_{r_1}\hat s_{r_2}\ldots \hat s_{r_t}$ of $w$.
We say that a reduced decomposition of $w$ is \textit{adapted} to $\ui,\uj$ and $x$ if it is a refinement of the decomposition $w=\hat s_{r_1}\hat s_{r_2}\ldots \hat s_{r_t}$ for some reduced decomposition $x=s_{r_1}s_{r_2}\ldots s_{r_t}$.
\end{defi}

\begin{lm}\label{lem:divided-some-eq}\leavevmode
\begin{lmlist}
	\item For $x\in R(\alpha)$ and $\ui\in I^{(\alpha)}$, we have $1_\ui x=x$ if and only if the image of the action of $x$ on $\Pol_\alpha$ is contained in $\Pol_m^{\frakS_\ui} 1_{\overline\ui}$;\label{lem:divided-some-eq:a}
	\item let $\ui,\uj\in I^{(\alpha)}$ be such that $\uj$ is a permutation of $\ui$ in the sense of \cref{def:adapted-perm}.
	For $w\in \doubleS{\uj}{\ui}$ induced by some $x\in \frakS_k$ with $x(\ui)=\uj$, define the operator $\psi_w$ using a reduced decomposition adapted to $\ui,\uj$ and $x$.
	Then we have $\psi_w 1_\ui=1_\uj \psi_w 1_\ui$;\label{lem:divided-some-eq:b}
	\item if $s_r\in \frakS_\ui$, then $\psi_r Q 1_\ui=-\partial_r(Q)1_\ui$ as elements in $R(\alpha)$ for any $Q\in \Pol_m$.\label{lem:divided-some-eq:c}
\end{lmlist}
\end{lm}
\begin{proof}
Part (a) follows from the fact that $1_\ui$ acts on $\Pol_\alpha$ as a projector to $\Pol_m^{\frakS_\ui}1_{\overline\ui}$.

For (b), it is enough to prove the statement for $\ui=(i_1^{(a_1)},i_2^{(a_2)})$, $\uj=(i_2^{(a_2)},i_1^{(a_1)})$, and $w$ the unique non-trivial permutation in $\doubleS{\uj}{\ui}$.
If $i_1\ne i_2$, then we clearly have $1_\uj\psi_w=\psi_w 1_\ui$ by relations~\eqref{rel:KLR1} and~\eqref{rel:KLR4}.
If $i_1=i_2$, then it suffices to show that the Demazure operator $\partial_w$ sends $\Pol_{a_1+a_2}^{\frakS_{a_1}\times \frakS_{a_2}}$ to $\Pol_{a_1+a_2}^{\frakS_{a_2}\times \frakS_{a_1}}$.
This follows from \cref{lem:Dem-ab-sym}, which says that $\partial_w$ always sends $\Pol_{a_1+a_2}^{\frakS_{a_1}\times \frakS_{a_2}}$ to $\Pol_{a_1+a_2}^{\frakS_{a_1+a_2}}$.

In order to check (c), we act by $\psi_r Q1_\ui$ on some $P\in \Pol_\alpha$.
First of all, $1_\ui\cdot P$ is of the form $R1_{\overline\ui}$ for some $R\in \Pol^{\frakS_\ui}$.
Since $s_r\in \frakS_\ui$, we have $s_r(R)=R$.
In particular, the operator $\partial_r$ commutes with multiplication by $R$.
Therefore
\[
\psi_r Q1_\ui\cdot P=\psi_r\cdot Q R1_{\overline\ui}=-\partial_r(QR)1_{\overline\ui}=-\partial_r(Q)R1_{\overline\ui}=-\partial_r(Q)1_\ui\cdot P.
\]
We conclude by faithfulness of the polynomial representation.
\end{proof}

For $\uj,\uj'\in I^{(\alpha)}$, we say that $\uj'$ is a \textit{split} of $\uj$ if we have $\overline\uj=\overline\uj'$ and $\frakS_{\uj'}\subset \frakS_{\uj}$.
In this case, let $w_{0,\uj,\uj'}=w_{0,\uj}w_{0,\uj'}^{-1}$ be the longest element in $\frakS_\uj\cap \doubleS{\overline\uj}{\uj'}$.
For any $\ui,\uj\in I^{(\alpha)}$ and $w\in \doubleS{\uj}{\ui}$, there exist unique $\ui',\uj'\in I^{(\alpha)}$ such that $\ui'$ is a split of $\ui$, $\uj'$ is a split of $\uj$, $\uj'$ is a permutation of $\ui'$ and $\frakS_{\uj'}=w\frakS_{\ui}w^{-1}\cap \frakS_\uj$, $\frakS_{\ui'}=w^{-1}\frakS_{\uj}w\cap \frakS_\ui$ (compare this to the notation in~\eqref{fla:basisSchur}).
Fix a basis $B_{\ui'}$ of $\Pol_m^{\frakS_{\ui'}}$; note that $B_{\uj'}=w(B_{\ui'})$ is a basis of $\Pol_m^{\frakS_{\uj'}}$.

\begin{lm}
\label{lem:basis-KLR-divided}
For each $\ui,\uj\in I^{(\alpha)}$, each of the following two sets forms a basis of $1_\uj R(n\delta) 1_\ui$:
\begin{equation}\label{eq:KLR-divided-bases}
\{\psi_{w_{0,\uj,\uj'}}P\psi_w  1_\ui;~w\in {^\uj}\frakS^\ui, P\in B_{\uj'}\},\qquad \{\psi_{w_{0,\uj,\uj'}}\psi_w P  1_\ui;~w\in {^\uj}\frakS^\ui, P\in B_{\ui'}\}.
\end{equation}
\end{lm}
\begin{proof}
We concentrate on the first set for now.
Let us first prove that it forms a basis under the following additional assumptions on the choice of reduced decompositions:
\begin{itemize}
	\item each element $w\in \doubleS{\overline{\uj}}{\overline{\ui}}$ can be written in a unique way as $w=w'x$, where $w'\in \doubleS{\overline{\uj}}{\ui}$ and $x\in \frakS_\ui$.
	We assume that the reduced expressions are chosen in such a way that $\psi_w1_\ui=\psi_{w'}\psi_x1_\ui$;
	\item each element $w\in \doubleS{\overline{\uj}}{\ui}$ can be written in a unique was as $w=xw'$, where $w'\in \doubleS{\uj}{\ui}$ and $x\in \frakS_\uj$.
	We assume that the reduced expressions are chosen in such a way that $\psi_w1_\ui=\psi_x\psi_{w'}1_\ui$;
	\item we assume additionally that the reduced representation of each $w\in \doubleS{\uj}{\ui}$ is adapted to $\ui',\uj'$ and $x$ in the sense of \cref{def:adapted-perm}, where $x$ is the permutation with $x(\ui')=\uj'$ that induces $w$.
\end{itemize}

It is clear from \cref{lem:basis-KLR} that the set
\begin{equation}\label{eq:span-proof}
\{1_\uj P\psi_w 1_\ui:w\in \doubleS{\overline{\uj}}{\overline{\ui}}, P\in \Pol_m\}
\end{equation}
spans $1_\uj R(\alpha) 1_\ui$.
We are going to reduce the number of generators in \eqref{eq:span-proof}.
First, we can assume that each $w$ lies in $\doubleS{\overline{\uj}}{\ui}$, because if $w$ is not minimal in $w\frakS_\ui$, then we have $\psi_w 1_\ui=0$ by the first assumption. 
Let us write $w\in \doubleS{\overline{\uj}}{\ui}$ as $w=xw'$, where $w'\in \doubleS{\uj}{\ui}$ and $x\in \frakS_\uj$ as in the second assumption.  
There exists a polynomial $Q\in \Pol_m$ such that we have the following chain of equalities
\begin{align*}
1_\uj P\psi_w 1_\ui&=\psi_{w_{0,\uj}}y_\uj P\psi_x\psi_{w'} 1_\ui\\
&=\psi_{w_{0,\uj}}Q\psi_{w'}1_\ui\\
&=\psi_{w_{0,\uj,\uj'}}\psi_{w_{0,\uj'}}Q\psi_{w'}1_\ui\\
&=\psi_{w_{0,\uj,\uj'}}(-1)^{\ell(w_{0,\uj'})}\partial_{w_{0,\uj'}}(Q)\psi_{w'}1_\ui.
\end{align*}
The first and the third equalities follow from the definitions of $1_\uj$ and $\psi_{w_{0,\uj,\uj'}}$ respectively.
For the second equality, we use relations~\eqref{rel:KLR1} and~\eqref{rel:KLR2} in order to move polynomial in the expression $y_\uj P\psi_x$ past $\psi$.
Since $\psi_{w_{0,\uj}}\psi_r=0$ for each $r$ with $s_r\in \frakS_\uj$, the additional terms coming from~\eqref{rel:KLR2} will disappear, and therefore $\psi_{w_{0,\uj}}y_\uj P\psi_x=\psi_{w_{0,\uj}}Q$ for some $Q\in \Pol_m$. 
Finally, let us justify the fourth equality.
First, \cref{lem:divided-some-eq:b} implies that
\[
\psi_{w'}1_\ui=\psi_{w'} 1_{\ui'} 1_\ui=1_{\uj'}\psi_{w'} 1_{\ui'} 1_\ui=1_{\uj'}\psi_{w'} 1_\ui.
\]
Second, by \cref{lem:divided-some-eq:c} we have 
\[
(\psi_{w_{0,\uj'}}Q1_{\uj'})\psi_{w'}1_\ui = (-1)^{\ell(w_{0,\uj'})}(\partial_{w_{0,\uj'}}(Q)1_{\uj'}) \psi_{w'}1_\ui.
\]

All in all, this shows that the first set in~\eqref{eq:KLR-divided-bases} spans $1_\uj R(n\delta) 1_\ui$.
It remains to check linear independence.
Consider an element $\psi_{w_{0,\uj,\uj'}}P\psi_w 1_\ui$ from this set.
Applying relations~(\ref{rel:KLR1}-\ref{rel:KLR2}), we get
$$
\psi_{w_{0,\uj,\uj'}}P\psi_w 1_\ui\in Q\psi_{w_{0,\uj,\uj'}}\psi_w 1_\ui+R(\alpha)^{<{w_{0,\uj,\uj'}w}},
$$
where we use notations of \cref{rmq:basis-KLR}, and $Q=w_{0,\uj,\uj'}(P)$.
Linear independence therefore follows from \cref{lem:basis-KLR}.

Now, let us prove the claim without additional assumptions on the reduced decompositions.
Consider a partial order on the basis obtained above, defined as follows:
\[
\psi_{w_{0,\uj,\uj'}}P_1\psi_{w_1}  1_\ui < \psi_{w_{0,\uj,\uj'}}P_2\psi_{w_2} 1_\ui \quad  \Leftrightarrow \quad l(w_1)<l(w_2).
\]

First, note that  $1_{\overline\uj}\psi_{w_{0,\uj,\uj'}}$ is independent of the choice of the reduced decomposition of $w_{0,\uj,\uj'}$ because its diagram contains only crossings of strands with the same label.
Assume that we have made some other choice of reduced decompositions.
Let us write $\psi_w$ for the operator defined with respect to the previous choice of decompositions, and $\psi'_w$ with respect to the new one.
We have
$$
\psi_{w_{0,\uj,\uj'}}P\psi'_w  1_\ui=\psi_{w_{0,\uj,\uj'}}P\psi_w  1_\ui+\ldots,
$$
where ellipses stand for lower terms with respect to the order introduced above.
We have thus deduced that the first set in~\eqref{eq:KLR-divided-bases} forms a basis for an arbitrary choice of reduced decompositions.
Finally, in a similar fashion we have
\[
\psi_{w_{0,\uj,\uj'}}\psi_w P  1_\ui=\psi_{w_{0,\uj,\uj'}}w(P)\psi_w  1_\ui + \ldots,
\]
so that the second set in~\eqref{eq:KLR-divided-bases} is a basis as well.
\end{proof}

\subsection{Geometric construction of divided powers}\label{subs:geom-div-powers}
Let us consider the following divided power versions of the KLR algebra $R(\alpha)$ and related geometric objects:
\[
\widehat R(\alpha)=\bigoplus_{\ui,\uj\in I^{(\alpha)}}1_\ui R(\alpha) 1_\uj,\qquad \widetilde{\bfF}_{(\alpha)} = \coprod_{\ui\in I^{(\alpha)}} \widetilde{\bfF}_{\ui},\qquad \bfZ_{(\alpha)}=\coprod_{\ui,\uj\in I^{(\alpha)}}\bfZ_{\ui,\uj}.
\]
Similarly to $\Pol_\alpha$, let us also consider the vector space $\Pol_{(\alpha)}=\bigoplus_{\ui\in I^{(\alpha)}}\Pol_m^{\frakS_\ui}1_\ui$, where $1_\ui$ is the projector to the direct summand labeled by $\ui$.
For any $\ui,\uj\in I^{(\alpha)}$, each element of $1_{\overline\ui}R(\alpha)1_{\overline\uj}$ yields a linear map $\Pol_m1_{\overline\uj}\to \Pol_m1_{\overline\ui}$ by \cref{lem:polrep-KLR}.
In particular, each element of $1_{\ui}R(\alpha)1_{\uj}\subset 1_{\overline\ui}R(\alpha)1_{\overline\uj}$ yields a linear map $\Pol_m^{\frakS_\uj}1_{\overline\uj}\to \Pol_m^{\frakS_\ui}1_{\overline\ui}$ by \cref{lem:divided-some-eq:a}.
This defines an action of $\widehat R(\alpha)$ on $\Pol_{(\alpha)}$.
Moreover, since $R(\alpha)$ acts faithfully on $\Pol_\alpha$, the representation $\Pol_{(\alpha)}$ of $\widehat R(\alpha)$ is faithful as well.

On the other hand, $H^*_{G_\alpha}(\bfZ_{(\alpha)})$ is a convolution algebra, which acts on $H^*_{G_\alpha}(\widetilde \bfF_{(\alpha)})$.
We have an identification of vector spaces
\[
H^*_{G_\alpha}(\widetilde \bfF_{(\alpha)})=\bigoplus_{\ui\in I^{(\alpha)}}H^*_{G_\alpha}(\widetilde \bfF_{\ui})=\bigoplus_{\ui \in I^{(\alpha)}}\Pol_m^{\frakS_\ui}1_\ui=\Pol_{(\alpha)}.
\]
We will upgrade this to an isomorphism of $\widehat R(\alpha)$-modules in \cref{prop:geom-KLR-div}.

Note that for any $\ui\in I^{(\alpha)}$, we have a closed embedding
\[
\widetilde{\bfF}_{\overline\ui} = \widetilde{\bfF}_{\ui}\times_{\widetilde{\bfF}_{\ui}}\widetilde{\bfF}_{\overline\ui} \hookrightarrow \widetilde{\bfF}_{\ui}\times_{E_\alpha}\widetilde{\bfF}_{\overline\ui} = \bfZ_{\ui,\overline\ui}.
\]
Consider the corresponding classes in the algebra $H^*_{G_\alpha}(\bfZ_{(\alpha)})$: 
$$
z_{\ui,\overline\ui}=[\widetilde{\bfF}_{\ui}\times_{\widetilde{\bfF}_{\ui}}\widetilde{\bfF}_{\overline\ui}]\in H^*_{G_\alpha}(\bfZ_{\ui,\overline\ui})\subset H^*_{G_\alpha}(\bfZ_{(\alpha)}),\qquad z_{\overline\ui,\ui}=[\widetilde{\bfF}_{\overline\ui}\times_{\widetilde{\bfF}_{\ui}}\widetilde{\bfF}_{\ui}]\in H^*_{G_\alpha}(\bfZ_{\overline\ui,\ui})\subset H^*_{G_\alpha}(\bfZ_{(\alpha)}).
$$

\begin{lm}
The map
$$
H^*_{G_\alpha}(\bfZ_{\ui,\uj})\to H^*_{G_\alpha}(\bfZ_{\overline\ui,\overline\uj}),\qquad x\mapsto z_{\overline\ui,\ui}x z_{\uj,\overline\uj} 
$$
is injective.
\end{lm}
\begin{proof}
By the definition of convolution product, this map is given by the following correspondence:
\[
\bfZ_{\ui,\uj}\xleftarrow{p}\widetilde{\bfF}_{\overline\ui} \times_{\widetilde{\bfF}_{\ui}} \bfZ_{\ui,\uj} \times_{\widetilde{\bfF}_{\uj}} \widetilde{\bfF}_{\overline\uj}\xrightarrow{q}\bfZ_{\overline\ui,\overline\uj},\qquad z_{\overline\ui,\ui}x z_{\uj,\overline\uj} = q_*p^!(x).
\]
Since $\bfZ_{\ui,\uj} = \widetilde{\bfF}_{\ui}\times_{E_\alpha} \widetilde{\bfF}_{\uj}$, it is clear that the map $q$ is an isomorphism.
Note that $\widetilde{\bfF}_{\overline\ui}\to \widetilde{\bfF}_{\ui}$ is a locally trivial fibration in partial flag varieties for any $\ui$.
In particular, $p$ is a locally trivial fibration in products of partial flag varieties, and it is straightforward to verify that $p^! = p^*$.
Moreover, $p^*$ is injective by an iterated application of projective bundle theorem.

Putting everything together, the map $x\mapsto z_{\overline\ui,\ui}x z_{\uj,\overline\uj}$ is identified with an injective map $p^*$.
\end{proof}

\begin{corr}
The representation $H^*_{G_\alpha}(\widetilde \bfF_{(\alpha)})$ of $H^*_{G_\alpha}(\bfZ_{(\alpha)})$ is faithful. 
\end{corr}
\begin{proof}
Suppose that $x\in H^*_{G_\alpha}(\bfZ_{\ui,\uj})$ acts on $H^*_{G_\alpha}(\widetilde \bfF_{(\alpha)})$ by zero.
Then the element $z_{\overline\ui,\ui}x z_{\uj,\overline\uj}\in \bfZ_{\overline\ui,\overline\uj}$ acts on $H^*_{G_\alpha}(\widetilde \bfF_{\alpha})$ by zero.
Since the action of $R(\alpha)$ on $H^*_{G_\alpha}(\widetilde \bfF_{\alpha})$ is faithful, we have $z_{\overline\ui,\ui}x z_{\uj,\overline\uj}=0$, and the lemma above implies that $x=0$.
\end{proof}

Recall that $w_{0,\ui}$ is the maximal length element in $\frakS_\ui$.
We view the polynomial $y_\ui$, defined in \cref{subs:divided-power}, as an element of $H^*_{G_\alpha}(\bfZ_{\overline\ui,\overline\ui})$.
\begin{lm}\label{lem:action-z}
\begin{lmlist}
	\item \label{lem:action-z:a} The element $z_{\ui,\overline\ui}$ acts on $\Pol_{(\alpha)}$ by 
	\[
	z_{\ui,\overline\ui}: \Pol_m1_{\overline\ui}\to \Pol_m^{\frakS_\ui}1_{\ui},\quad P1_{\overline\ui}\mapsto (-1)^{\ell(w_{0,\ui})}\partial_{w_{0,\ui}}(P)1_{\overline\ui};
	\]
	\item \label{lem:action-z:b} the element $z_{\overline\ui,\ui}$ acts on $\Pol_{(\alpha)}$ by 
	\[
	z_{\overline\ui,\ui}:\Pol_m^{\frakS_\ui}1_{\ui}\to \Pol_m1_{\overline\ui},\quad P1_{\overline\ui}\mapsto P1_{\ui}.
	\]
\end{lmlist}
\end{lm}
\begin{proof}
See~\cite[Theorem~4.7]{Prz_QSAC2019}; the proof there is similar to our \cref{SchurPolRepFaith}.
\end{proof}

\begin{corr}\label{cor:eq-z}
\begin{corrlist}
	\item We have $z_{\ui,\overline\ui}y_\ui z_{\overline\ui,\ui}=1$ in $H^*_{G_\alpha}(\bfZ_{\ui,\ui})$;\label{cor:eq-z:a}
	\item we have $z_{\overline\ui,\ui}z_{\ui,\overline\ui}y_\ui=1_\ui$ in $1_{\overline\ui}R(\alpha)1_{\overline\ui}=H^*_{G_\alpha}(\bfZ_{\overline\ui,\overline\ui})$. \label{cor:eq-z:b}
\end{corrlist}	
\end{corr}
\begin{proof}
It suffices to check these equalities on polynomial representations, where they follow from \cref{lem:action-z}. 
\end{proof}

The following statement is the divided power version of \cref{prop:geom-KLR}.
\begin{prop}\label{prop:geom-KLR-div}
\begin{proplist}
	\item There exists an isomorphism of algebras $H^*_{G_\alpha}(\bfZ_{(\alpha)})\simeq \widehat R(\alpha)$;\label{prop:geom-KLR-div:a}
	\item this isomorphism restricts to $H^*_{G_\alpha}(\bfZ_{\ui,\uj})\simeq 1_\ui R(\alpha) 1_\uj$ for each $\ui,\uj\in I^{(\alpha)}$;\label{prop:geom-KLR-div:b}
	\item the $H^*_{G_\alpha}(\bfZ_{(\alpha)})$-action on $H^*_{G_\alpha}(\widetilde \bfF_{(\alpha)})$ gets identified with the $\widehat R(\alpha)$-action on $\Pol_{(\alpha)}$.\label{prop:geom-KLR-div:c}
\end{proplist}		
\end{prop}
\begin{proof}
For each $\ui,\uj\in I^{(\alpha)}$, consider the maps
\begin{gather*}
p_{\ui,\uj}\colon H^*_{G_\alpha}(\bfZ_{\ui,\uj})\to H^*_{G_\alpha}(\bfZ_{\overline\ui,\overline\uj}), \qquad x\mapsto  z_{\overline\ui,\ui}x z_{\uj,\overline\uj} y_\uj,\\
q_{\ui,\uj}\colon H^*_{G_\alpha}(\bfZ_{\overline\ui,\overline\uj})\to H^*_{G_\alpha}(\bfZ_{\ui,\uj}), \qquad x\mapsto  z_{\ui,\overline\ui}y_\ui xz_{\overline\uj,\uj}.
\end{gather*}
It follows from \cref{cor:eq-z:a} that $q_{\ui,\uj}\circ p_{\ui,\uj}$ is the identity of $H^*_{G_\alpha}(\bfZ_{\ui,\uj})$.
This allows to identify $H^*_{G_\alpha}(\bfZ_{\ui,\uj})$ with a vector subspace of $H^*_{G_\alpha}(\bfZ_{\overline\ui,\overline\uj})$.
Moreover, $p_{\ui,\uj}\circ q_{\ui,\uj}$ is the projector to the image of $H^*_{G_\alpha}(\bfZ_{\ui,\uj})$ in $H^*_{G_\alpha}(\bfZ_{\overline\ui,\overline\uj})$.
Let us describe this image. 

Let us identify $H^*_{G_\alpha}(\bfZ_{\overline\ui,\overline\uj})$ with $1_{\overline\ui}R(\alpha)1_{\overline\uj}$.
By \cref{cor:eq-z:b}, the map $p_{\ui,\uj}\circ q_{\ui,\uj}$ becomes
\begin{equation}
\label{eq:pq}
1_{\overline\ui}R(\alpha)1_{\overline\uj}\to 1_{\overline\ui}R(\alpha)1_{\overline\uj},\qquad x\mapsto 1_\ui x 1_\uj.
\end{equation}
This shows that under the identification above $H^*_{G_\alpha}(\bfZ_{\ui,\uj})$ coincides with $1_{\ui}R(\alpha)1_{\uj}$ as vector subspaces in $1_{\overline\ui}R(\alpha)1_{\overline\uj}$. 
Summing over $\ui,\uj\in I^{(\alpha)}$ yields an isomorphism of vector spaces $H^*_{G_\alpha}(\bfZ_{(\alpha)})\simeq \widehat R(\alpha)$.

Using \cref{cor:eq-z:a}, we see that for any $\uh,\ui,\uj\in I^{(\alpha)}$ and $x\in H^*_{G_\alpha}(\bfZ_{\uh,\ui})$, $y\in H^*_{G_\alpha}(\bfZ_{\ui,\uj})$ we have $p_{\uh,\ui}(x)\cdot p_{\ui,\uj}(y)=p_{\uh,\uj}(x\cdot y)$.
Therefore $H^*_{G_\alpha}(\bfZ_{(\alpha)})\simeq \widehat R(\alpha)$ is an isomorphism of algebras, which proves (a) and (b).

Let us verify (c).
Suppose $x\in H^*_{G_\alpha}(\bfZ_{\ui,\uj})$ acts on $H^*_{G_\alpha}(\widetilde \bfF_{(\alpha)})\simeq\Pol_{(\alpha)}$ by 
$$
\Pol_m^{\frakS_\uj}1_\uj\to \Pol_m^{\frakS_\ui}1_\ui,\qquad P 1_\uj\mapsto L(P)1_\ui,
$$
where $L:\Pol_m^{\frakS_\uj}\to \Pol_m^{\frakS_\ui}$ is a linear map.
Then by \cref{lem:action-z}, the element $p_{\ui,\uj}(x)$ acts on $\Pol_\alpha$ by
$$
\Pol_m1_{\overline\uj}\to \Pol_m1_{\overline\ui}, \qquad P1_{\overline\ui}\mapsto L(\partial_{w_{0,\uj}}(y_\uj P))1_{\overline\ui}.
$$ 
We see that this action agrees with the action of $\widehat R(\alpha)$ on $\Pol_{(\alpha)}$.
\end{proof}
\section{Semi-cuspidal category of the Kronecker quiver}\label{sec:semicusp}
In this section, we establish a link between KLR algebras for the Kronecker quiver and Schur algebras for $\bbP^1$.
We will assume that either $\bbk$ is a field or $\bbk = \bbZ$, unless otherwise stated.

\subsection{Kronecker quiver}
From now on, let $\Gamma=(1\rightrightarrows 0)$ be the Kronecker quiver, and denote by $\delta$ the dimension vector $\alpha_0+\alpha_1$.

Take $\alpha=n_0\alpha_0+n_1\alpha_1$, and $m = |\alpha| = n_0+n_1$.
Let us introduce a more convenient notation for polynomial variables in the KLR algebra $R(\alpha)$.
Let $\ui\in I^\alpha$; it can be thought of as a sequence consisting of $n_0$ zeroes and $n_1$ ones.
For each $1\leq r\leq m$, let $k_r$ be the number of $r'\in[1,r]$ such that $i_{r'} = i_r$.
Inside $1_\ui R(\alpha) 1_\ui$, we then write $y_r = u_{k_r}$ if $i_r = 0$, and $y_r = v_{k_r}$ if $i_r = 1$.

\begin{exe}
Let $\alpha=2\alpha_0+3\alpha_1$, $\ui=(0,1,1,0,1)$. Then we have
\begin{gather*}
u_11_\ui=y_11_\ui, \quad u_21_\ui=y_41_\ui,\\
v_11_\ui=y_21_\ui,\quad v_21_\ui=y_31_\ui,\quad v_31_\ui=y_51_\ui.
\end{gather*}
\end{exe}
In particular, for any non-negative integer $n$ denote
\[
\Poll_n=\bfk[u_1,\ldots,u_n,v_1,\dots,v_n].
\]
Using the new notation, we can write $\Pol_{n\delta}=\bigoplus_{\ui\in I^{n\delta}}\Poll_n1_\ui$.

For each composition $\lambda=(\lambda_1,\ldots, \lambda_k)\in\Comp(n)$, consider the elements
\[
\ui_\lambda=(0^{(\lambda_1)},1^{(\lambda_1)},\ldots, 0^{(\lambda_k)},1^{(\lambda_k)})\in I^{(n\delta)},\qquad \uj_\lambda=\overline\ui_\lambda=(0^{\lambda_1}1^{\lambda_1}\ldots 0^{\lambda_k}1^{\lambda_k})\in I^\alpha.
\]
Let $\ui_0 = \uj_0 = \uj_{1^n} = (0101\ldots01)\in I^{n\delta}$.
In order to unburden the notation, we will write $e_\lambda = 1_{\ui_\lambda}$, $e_0 = 1_{\ui_0}$, and $e = \sum_{\lambda\in \Comp(n)}e_\lambda$.
By \cref{lem:polrep-KLR} the algebra $R(n\delta)$ acts faithfully on $\Pol_{n\delta}$.
This implies that we have a faithful representation of $e R(n\delta) e$ on $e \Pol_{n\delta}$.
Since $e_\lambda \Pol_{n\delta}\simeq \Poll_n^{(\frakS_\lambda)^2}$, we obtain that $e R(n\delta)e$ has a faithful representation on $\bigoplus_{\lambda\in\Comp(n)} \Poll_n^{(\frakS_\lambda)^2}$. 

\subsection{Semi-cuspidal modules}
\begin{defi}
We say that a sequence $\ui=(i_1,i_2,\ldots,i_{2n})\in I^{n\delta}$ is \emph{non-cuspidal} if there exists an index $r\in [1;2n]$ such that $(i_1,i_2,\ldots,i_r)$ contains more $1$'s than $0$'s.
Denote by $I^{n\delta}_{\nc}$ the set of all non-cuspidal sequences in $I^{n\delta}$, and write $1_{\nc}=\sum_{\ui\in I^{n\delta}_{\nc}}1_\ui$.
\end{defi}

Given a $\bfk$-algebra $A$, let $\mod(A)$ denote the category of finitely generated $A$-modules.

\begin{defi}[{\cite[\S 2.6]{KM_SKAA2017}}]
We say that an $R(n\delta)$-module $M$ is \emph{semi-cuspidal} if $1_{\ui}M=0$ for each $\ui\in I^{n\delta}_{\nc}$.
In other words, $M$ is semi-cuspidal if it is annihilated by $1_{\nc}$.
\end{defi}

Denote by $\cusp(R(n\delta))\subset \mod(R(n\delta))$ the full subcategory of semi-cuspidal $R(n\delta)$-modules.
We clearly have $\cusp(R(n\delta))=\mod(C(n\delta))$, where $C(n\delta)$ is the quotient algebra $R(n\delta)/R(n\delta)1_{\nc}R(n\delta)$.

\begin{lm}
Let $\bfk$ be a field.
The $C(n\delta)$-module $C(n\delta)e$ is a projective generator in the category $\mod(C(n\delta))$.
If $\bfk$ has characteristic zero, then $C(n\delta)e_0$ is a projective generator of $\mod(C(n\delta))$ as well.
\end{lm}
\begin{proof}
The second claim follows from \cite[Lemma~6.22]{KM_SKAA2017}.

Let us prove the first claim.
It is equivalent to the fact that for each simple module $L\in\mod(C(n\delta))$ we can find $\nu\in\Comp(n)$ and a surjection $C(n\delta)e_\nu\to L$.
In other words, we need to show that for each simple module $L\in C(n\delta)$ we can find $\nu\in \Comp(n)$ such that $e_\nu L\ne 0$.
However, this follows from~\cite[Theorem~5.5.4]{KM_ISD2017}. 

Let us provide some additional explanation about the given reference.
The algebra $\mathscr S_n$ is defined in \cite[\S 4.3]{KM_ISD2017} as a quotient of $R(n\delta)$ by the annihilator of some semi-cuspidal $R(n\delta)$-module.
In particular, we a get a chain of surjections $R(n\delta)\to C(n\delta)\to \mathscr S_n$, and inclusions of categories $\mod(\mathscr S_n)\subset \mod(C(n\delta))\subset \mod(R(n\delta))$.
For each $\nu\in \Comp(n)$, \cite[\S 5.3]{KM_ISD2017} constructs an $\mathscr S_n$-module $Z^\nu$ and \cite[Theorem~5.5.4~(iii)]{KM_ISD2017} shows that $Z=\bigoplus_{\nu\in \Comp(n)}Z^\nu$ is a projective generator in $\mod(\mathscr S_n)$.
On the other hand, the proof of~\cite[Theorem~5.5.4]{KM_ISD2017} shows that there is a surjection of $R(n\delta)$-modules from $R(n\delta)e_\nu$ (denoted by $I_\nu^n \Gamma_\nu$ in \cite{KM_ISD2017}) to $Z^\nu$.
This proves that for each simple module $L\in \mod(\mathscr S_n)$ there exists $\nu\in\Comp(n)$ such that the $R(n\delta)$-module $R(n\delta)e_\nu$ surjects to $L$, in particular we have $1_\nu L\ne 0$. 

In order to complete the proof, we have to show that each simple module $L\in \mod(C(n\delta))$ factors through the quotient $C(n\delta)\to \mathscr S_n$.
In other words, we have to show that the categories $\mod(C(n\delta))$ and $\mod(\mathscr S_n)$ have the same number of simple modules.
By~\cite[Theorem~6]{KM_ISD2017}, the number of simple modules in $\mod(\mathscr S_n)$ is equal to the number of partitions of $n$.
On the other hand, by~\cite[Theorem~2]{KM_SKAA2017} the number of simple modules in $\mod(C(n\delta))$ is the same.
\end{proof}

\begin{corr}
For $\bfk$ a field, the algebra $C(n\delta)$ is Morita equivalent to $e C(n\delta)e$.
Moreover, if $\bfk$ has characteristic zero, $C(n\delta)$ is Morita equivalent to $e_0C(n\delta)e_0$.
\end{corr}

\subsection{Thick calculus in $e R(n\delta)e$}
\label{subs:thickcals-KLR}
In this section we construct some special element in the algebra $e R(n\delta)e$.

Let us introduce some diagrammatic abbreviations.
First, we write
\[
\tikz[thick,xscale=.3,yscale=.3]{
\draw (0,-1) --(0,7);
\ntxt{0}{8}{$a$}

\ntxt{2}{3}{$=$}

\draw (4.5,-1) -- (4.5,1.5);
\draw (6,-1) -- (6,1.5);
\draw (9.5,-1) -- (9.5,1.5);
\ntxt{7.75}{0}{$\ldots$}
\opbox{4}{1.5}{10}{4.5}{$0^{(a)}$}
\draw (4.5,4.5) -- (4.5,7);
\draw (6,4.5) -- (6,7);
\draw (9.5,4.5) -- (9.5,7);
\ntxt{7.75}{6.5}{$\ldots$}
\ntxt{4.5}{8}{$0$}
\ntxt{6}{8}{$0$}
\ntxt{9.5}{8}{$0$}

\draw (11.5,-1) -- (11.5,1.5);
\draw (13,-1) -- (13,1.5);
\draw (16.5,-1) -- (16.5,1.5);
\ntxt{14.75}{0}{$\ldots$}
\opbox{11}{1.5}{17}{4.5}{$1^{(a)}$}
\draw (11.5,4.5) -- (11.5,7);
\draw (13,4.5) -- (13,7);
\draw (16.5,4.5) -- (16.5,7);
\ntxt{14.75}{6.5}{$\ldots$}
\ntxt{11.5}{8}{$1$}
\ntxt{13}{8}{$1$}
\ntxt{16.5}{8}{$1$}

\ntxt{18}{3}{.}
}
\]
In particular, for $\lambda=(\lambda_1,\ldots,\lambda_k)\in \Comp(n)$, we draw the idempotent $e_\lambda$ as $k$ parallel vertical lines with labels $\lambda_1,\ldots,\lambda_k$.
Moreover, a strand with label $a$ is allowed to carry a polynomial $P\in \Poll_a^{(\frakS_a)^2}$.
In fact, it would make sense to allow polynomials from $\Poll_a$; however, the presence of an idempotent allows us to replace any polynomial by a symmetric one.

Next, we write
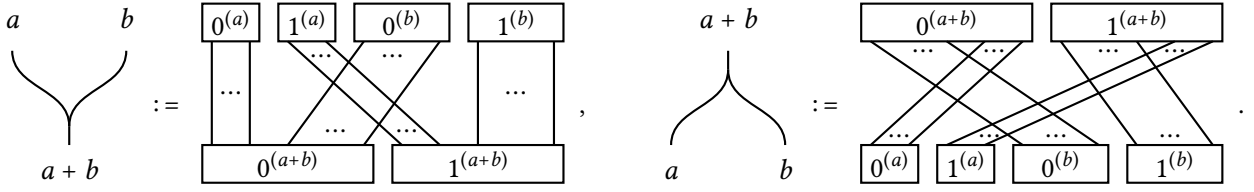
\begin{equation*}
\tikz[thick,xscale=.25,yscale=.25]{
	\draw (3,0) -- (3,1);
	\ntxt{3}{-1.5}{$a+b$}
	\dsplit{0}{1}{6}{5}
	\ntxt{0}{6.5}{$a$}
	\ntxt{6}{6.5}{$b$}
	
	\ntxt{8}{2}{$:=$}
	
	\opbox{10}{-2}{19}{0}{$0^{(a+b)}$}
	\opbox{20}{-2}{29}{0}{$1^{(a+b)}$}
	
	\opbox{10}{5.5}{13}{7.5}{$0^{(a)}$}
	\opbox{14}{5.5}{17}{7.5}{$1^{(a)}$}
	\opbox{18}{5.5}{23}{7.5}{$0^{(b)}$}
	\opbox{24}{5.5}{29}{7.5}{$1^{(b)}$}
	
	\draw (10.5,0) -- (10.5,5.5);
	\draw (12.5,0) -- (12.5,5.5);
	\ntxt{11.5}{3}{$\ldots$}
	
	\draw (14.5,0) -- (18.5,5.5);
	\draw (18.5,0) -- (22.5,5.5);
	\ntxt{17}{1}{$\ldots$}
	\ntxt{20}{5}{$\ldots$}
	
	\draw (20.5,0) -- (14.5,5.5);
	\draw (22.5,0) -- (16.5,5.5);
	\ntxt{20.75}{1}{$\ldots$}
	\ntxt{16.25}{5}{$\ldots$}
	
	\draw (24.5,0) -- (24.5,5.5);
	\draw (28.5,0) -- (28.5,5.5);
	\ntxt{26.5}{3}{$\ldots$}
	
	\ntxt{30}{2}{,}
}	\qquad
\tikz[thick,xscale=.25,yscale=.25]{
	\draw (3,4) -- (3,5);
	\ntxt{3}{6.5}{$a+b$}
	\dmerge{0}{0}{6}{4}
	\ntxt{0}{-1.5}{$a$}
	\ntxt{6}{-1.5}{$b$}
	
	\ntxt{8}{2}{$:=$}
	
	\opbox{10}{5.5}{19}{7.5}{$0^{(a+b)}$}
	\opbox{20}{5.5}{29}{7.5}{$1^{(a+b)}$}
	
	\opbox{10}{-2}{13}{0}{$0^{(a)}$}
	\opbox{14}{-2}{17}{0}{$1^{(a)}$}
	\opbox{18}{-2}{23}{0}{$0^{(b)}$}
	\opbox{24}{-2}{29}{0}{$1^{(b)}$}
	
	\draw (10.5,0) -- (16.5,5.5);
	\draw (12.5,0) -- (18.5,5.5);
	\ntxt{12}{0.75}{$\ldots$}
	\ntxt{17}{5.25}{$\ldots$}
	
	\draw (14.5,0) -- (26.5,5.5);
	\draw (16.5,0) -- (28.5,5.5);
	\ntxt{16.75}{0.75}{$\ldots$}
	\ntxt{26.5}{5.25}{$\ldots$}
	
	\draw (18.5,0) -- (10.5,5.5);
	\draw (22.5,0) -- (14.5,5.5);
	\ntxt{20.25}{0.75}{$\ldots$}
	\ntxt{13.25}{5.25}{$\ldots$}
	
	\draw (24.5,0) -- (20.5,5.5);
	\draw (28.5,0) -- (24.5,5.5);
	\ntxt{26.25}{0.75}{$\ldots$}
	\ntxt{23}{5.25}{$\ldots$}
	
	\ntxt{30}{2}{.}
}
\end{equation*}
Assume that $\lambda,\mu\in \Comp(n)$ are such that $\mu$ is a split of $\lambda$ at $k$-th place.
Then, similarly to \cref{subs:diag-Schur}, we define elements $\rmS_\lambda^\mu\in e_\mu R(n\delta)e_\lambda$ and $\rmM_\mu^\lambda\in e_\lambda R(n\delta)e_\mu$ by \eqref{eq:pic-SM}, but using the diagrammatic calculus defined above for $e R(n\delta)e$ instead of the analogous calculus for curve Schur algebra.
It is easy to check that the elementary splits and merges above are associative as in~\eqref{eq:assoc-SM}.
This allows us to extend the definitions of $\rmS_\lambda^\mu$ and $\rmM_\mu^\lambda$ to any $\lambda,\mu$ with $\frakS_\mu\subset \frakS_\lambda$.

\begin{rmq}
It is not the case that we can write any element of $e R(n\delta)e$ as a linear combination of diagrams containing splits, merges and symmetric polynomials.
However, we will see in \cref{rem:span-gCg-thick} that this holds for $e C(n\delta)e$.
Moreover, it can be shown that $e R(n\delta)e$ is an idempotent truncation of the quiver Schur algebra and the diagrams introduced above are nothing else than the diagrams in quiver Schur algebra (replacing the label $a$ by $a\delta$).
However, here we allow only labels of the form $a\delta$, while quiver Schur algebras allow more general labels of the form $a_0\alpha_0+a_1\alpha_1$.
This is the reason why our thick calculus does not have enough diagrams to represent every element in $e R(n\delta)e$. 
\end{rmq}

Let $\lambda,\mu\in \Comp(n)$ be such that $\frakS_\mu\subset \frakS_\lambda$. 
Let us give a geometric description of the operators $\rmS_\lambda^\mu$, $\rmM_\mu^\lambda$.
We have an obvious projection $\widetilde{\bfF}_{\ui_\mu}\to \widetilde{\bfF}_{\ui_\lambda}$, obtained by forgetting some components of the flag.
This allows us to define the following correspondences:
\[
\bfZ^{\rm S}_{\ui_\mu,\ui_\lambda} := \widetilde{\bfF}_{\ui_\mu}\times_{\widetilde{\bfF}_{\ui_\lambda}} \widetilde{\bfF}_{\ui_\lambda} \subset \widetilde{\bfF}_{\ui_\mu}\times_{E_\alpha} \widetilde{\bfF}_{\ui_\lambda} = \bfZ_{\ui_\mu,\ui_\lambda},\qquad \bfZ^{\rm M}_{\ui_\lambda,\ui_\mu} := \widetilde{\bfF}_{\ui_\lambda}\times_{\widetilde{\bfF}_{\ui_\lambda}} \widetilde{\bfF}_{\ui_\mu} \subset \widetilde{\bfF}_{\ui_\lambda}\times_{E_\alpha} \widetilde{\bfF}_{\ui_\mu} = \bfZ_{\ui_\lambda,\ui_\mu}.
\]

\begin{lm}\label{lm:geomSMcusp}
Under the identification in \cref{prop:geom-KLR-div}, we have $\rmS_\lambda^\mu=[\bfZ^{\rm S}_{\ui_\mu,\ui_\lambda}]$ and $\rmM_\mu^\lambda=[\bfZ^{\rm M}_{\ui_\lambda,\ui_\mu}]$.
\end{lm}
\begin{proof}
It suffices to check these equalities on the faithful representation $\Pol_{(n\delta)}$ of $\widehat R(n\delta)$.
The actions of $\rmS_\lambda^\mu$ and $\rmM_\mu^\lambda$ can be easily obtained from \cref{lem:polrep-KLR}.
On the other hand, the actions of $[\bfZ^{\rm S}_{\ui_\mu,\ui_\lambda}]$ and $[\bfZ^{\rm M}_{\ui_\lambda,\ui_\mu}]$ were computed in larger generality (for quiver Schur algebras) in~\cite[Theorem~4.7]{Prz_QSAC2019}; see also~\cite[Proposition~3.4]{SW_QSAF2014}. 

By \cite[Theorem~4.7(b)]{Prz_QSAC2019}, the element $[\bfZ^{\rm M}_{\ui_\lambda,\ui_\mu}]$ acts by
\[
\Pol_{2n}^{\frakS_{\ui_\mu}}1_{\ui_\mu}\to \Pol_{2n}^{\frakS_{\ui_\lambda}}1_{\ui_\lambda}, \qquad P1_{\ui_\mu}\mapsto P1_{\ui_\lambda},
\]
which coincides with the action of $\rmM_\mu^\lambda$.

For $\rmS_\lambda^\mu$, assume $\lambda=(a+b)$ and $\mu=(a,b)$; the general case is proved in the same way, but requires more complicated notation.
By \cite[Theorem~4.7(a)]{Prz_QSAC2019}, the element $[\bfZ^{\rm S}_{\ui_\mu,\ui_\lambda}]$ acts on $\Pol_{(n\delta)}$ by
\[
\Pol_{2n}^{\frakS_{\ui_\lambda}}1_{\ui_\lambda}\to \Pol_{2n}^{\frakS_{\ui_\mu}}1_{\ui_\mu}, \qquad P1_{\ui_\mu}\mapsto \sum_{w_u,w_v\in \fk{S}_n/(\fk{S}_a\times \fk{S}_b)}w_uw_v\left( P\prod_{i=1}^a\prod_{j=a+1}^{a+b}\frac{(v_i-u_j)^2}{(u_i-u_j)(v_i-v_j)} \right),
\]
where $w_u$ permutes $u_i$'s, and $w_v$ permutes $v_i$'s.
Let $\partial^u_{w_{0,a,b}}$ be the composition of Demazure operators as in \cref{lem:Dem-ab-sym} acting on variables $u_1,\ldots,u_n$, and define $\partial^v_{w_{0,a,b}}$ analogously.
Applying \cref{lem:Dem-ab-formula}, we have
\[
\sum_{w_u,w_v\in \fk{S}_n/(\fk{S}_a\times \fk{S}_b)}w_uw_v\left( P\prod_{i=1}^a\prod_{j=a+1}^{a+b}\frac{(v_i-u_j)^2}{(u_i-u_j)(v_i-v_j)} \right)
= \partial^u_{w_{0,a,b}}\partial^v_{w_{0,a,b}}\left( P\prod_{i=1}^a\prod_{j=a+1}^b (v_i-u_j)^2 \right),
\]
and the right-hand side coincides with the action of $\rmS_\lambda^\mu$ given by \cref{lem:polrep-KLR}.
\end{proof}

As in \cref{subs:diag-Schur}, we will use the following abbreviation:
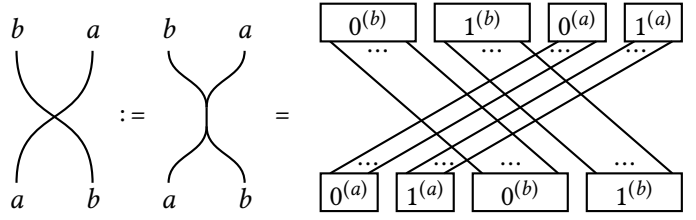
\[
\tikz[thick,xscale=.25,yscale=.25]{
	\crosin{0}{0}{4}{7}
	\ntxt{0}{-1}{$a$}
	\ntxt{4}{-1}{$b$}
	\ntxt{0}{8}{$b$}
	\ntxt{4}{8}{$a$}
	
	\ntxt{6}{3.5}{$:=$}
	
	\dmerge{8}{0}{12}{3}
	\draw (10,3) -- (10,4);
	\dsplit{8}{4}{12}{7}
	\ntxt{8}{-1}{$a$}
	\ntxt{12}{-1}{$b$}
	\ntxt{8}{8}{$b$}
	\ntxt{12}{8}{$a$}
	
	\ntxt{14}{3.5}{$=$}
	
	\opbox{16}{-1.5}{19}{0.5}{$0^{(a)}$}
	\opbox{20}{-1.5}{23}{0.5}{$1^{(a)}$}
	\opbox{24}{-1.5}{29}{0.5}{$0^{(b)}$}
	\opbox{30}{-1.5}{35}{0.5}{$1^{(b)}$}
	
	\draw (16.5,0.5) -- (28.5,7.5);
	\draw (18.5,0.5) -- (30.5,7.5);
	\ntxt{18.5}{1.25}{$\ldots$}
	\ntxt{28.6}{7.25}{$\ldots$}
	\draw (20.5,0.5) -- (32.5,7.5);
	\draw (22.5,0.5) -- (34.5,7.5);
	\ntxt{22.5}{1.25}{$\ldots$}
	\ntxt{32.6}{7.25}{$\ldots$}
	
	\draw (24.5,0.5) -- (16.5,7.5);
	\draw (28.5,0.5) -- (20.5,7.5);
	\ntxt{26}{1.25}{$\ldots$}
	\ntxt{19}{7.25}{$\ldots$}
	\draw (30.5,0.5) -- (22.5,7.5);
	\draw (34.5,0.5) -- (26.5,7.5);
	\ntxt{32}{1.25}{$\ldots$}
	\ntxt{25}{7.25}{$\ldots$}
	
	\opbox{16}{7.5}{21}{9.5}{$0^{(b)}$}
	\opbox{22}{7.5}{27}{9.5}{$1^{(b)}$}
	\opbox{28}{7.5}{31}{9.5}{$0^{(a)}$}
	\opbox{32}{7.5}{35}{9.5}{$1^{(a)}$}
	
	\ntxt{36}{3.5}{.}
}
\]
Let $\lambda,\mu\in \Comp(n)$ be such that $\mu$ is obtained by permuting components in $\lambda=(\lambda_1,\ldots,\lambda_r)$, and $w\in\fk{S}_r$ the corresponding permutation.
We can define the permutation element $\rmR_\lambda^\mu(w)\in e_\lambda R(n\delta) e_\mu$ as in \cref{subs:diag-Schur}; recall that it depends not only on $w$, but also on the choice of a reduced decomposition of $w$.

\subsection{Basis in $e R(n\delta) e$}
Let $\lambda,\mu\in\Comp(n)$.
To each element $w\in \doubleS{\ui_\mu}{\ui_\lambda}$, we can associate a pair $(x,y)\in\doubleS{\mu}{\lambda}\times \doubleS{\mu}{\lambda}$, where $x$ is the restriction of $w$ to positions colored by $0$, and $y$ to positions colored by $1$.
This induces a bijection $\doubleS{\ui_\mu}{\ui_\lambda}\xra{\sim} \doubleS{\mu}{\lambda}\times \doubleS{\mu}{\lambda}$.
We will use this bijection implicitly from now on, and write $\psi_{(x,y)}$ instead of $\psi_w$.
In the case $x=y$, we may also write $\psi_{(x)}$ instead of $\psi_{(x,x)}$ by abuse of notation.

As in \cref{subs:basis-KLR}, consider the elements $\ui'_\lambda,\ui'_\mu\in I^{(n\delta)}$ (both depending on $\lambda$, $\mu$, $w$) characterized by
\[
\overline{\ui'_\lambda}=\overline{\ui_\lambda},\qquad \overline{\ui'_\mu}=\overline{\ui_\mu},\qquad \frakS_{\ui'_\lambda}= w^{-1} \frakS_{\ui'_\mu} w = \frakS_{\ui_\lambda}\cap w^{-1}\frakS_{\ui_\mu} w.
\] 
The elements $\ui_\lambda$, $\ui_\mu$, $\ui'_\lambda$ and $\ui'_\mu$ here play the roles of $\ui$, $\uj$, $\ui'$ and $\uj'$ respectively in \cref{subs:basis-KLR}.
Note that in general $\ui'_\lambda$ cannot be expressed as $\ui_{\lambda'}$ for $\lambda'\in\Comp(n)$; however, this works if $x=y$.

The following is a restatement of the second part of \cref{lem:basis-KLR-divided} for the Kronecker quiver.
\begin{lm}
\label{lem:basis-gRg}
For each $\lambda,\mu\in \Comp(n)$, the following set is a basis of the $\bfk$-module $e_\mu R(n\delta) e_\lambda$:
$$
\left\{\psi_{w_{0,\ui_\mu,\ui_\mu'}}\psi_{(x,y)} Pe_\lambda: x,y\in \doubleS{\mu}{\lambda}, P\in B_{\ui'_\lambda}\right\}.
$$
\end{lm}

\begin{rmq}
\label{rem:polyn-in-KLR}\leavevmode
\begin{rmqlist}
	\item \label{rem:polyn-in-KLR:a} The lemma above implies that for any $\lambda\in\Comp(n)$ we have an isomorphism of $\bfk$-modules
	\[
	\Poll_n^{(\frakS_\lambda)^2}\to e_\lambda R(n\delta)e_{(n)},\qquad P\mapsto P\cdot \rmS_{(n)}^\lambda;
	\]
	\item \label{rem:polyn-in-KLR:b} suppose that we have $\lambda_r=1$ for some index $r$, and set $k=\lambda_1+\ldots+\lambda_{r-1}+1$.
	The quadratic relation $\psi_{2k-1}^2e_\lambda=(u_k-v_k)^21_{\uj_\lambda}$ and the fact that the idempotent $s_{2k-1}(\uj_\lambda)$ is non-cuspidal implies that the polynomial $(u_k-v_k)^2$ is in the kernel of the map
	$$
	\Poll_n^{(\frakS_\lambda)^2}\to e_\lambda C(n\delta)e_{(n)},\qquad P\mapsto P\cdot \rmS_{(n)}^\lambda.
	$$
\end{rmqlist}
\end{rmq}

\begin{exe}
Let $n=4$, $\lambda=(3,1)$, $\mu=(1,3)$, $w=(1,3,4,6,7,8,5,2)$.
In this case $x=(1,2,3,4)$, $y=(2,3,4,1)$, and the set $B_{\ui'_\lambda}$ is a basis in the vector space of polynomials in $\bfk[u_1,u_2,u_3,u_4,v_1,v_2,v_3,v_4]^{\fk{S}_2\times\fk{S}_2}$, where $\fk{S}_2\times\fk{S}_2$ acts by transpositions $u_1\leftrightarrow u_2$ and $v_1\leftrightarrow v_2$.
For example, take $P=u_1u_2v_4$.
Then the basis element $\psi_{w_{0,\ui_\mu,\ui_\mu'}}\psi_w Pe_\lambda$ is given by the following diagram:
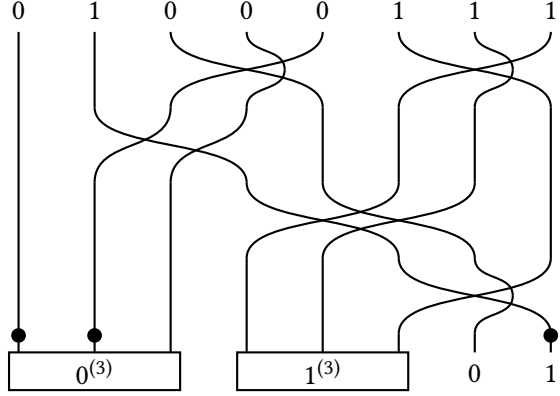
\[
\tikz[thick,xscale=.25,yscale=.25]{
\draw (0,0) -- (0,17);
\strdot{0}{0.9}
\strdot{4}{0.9}
\strdot{28}{0.9}
\draw (4,0) -- (4,9);
\draw (8,0) -- (8,9);
\draw (12,0) -- (12,5);
\draw (16,0) -- (16,5);
\draw (20,0) -- (20,1);
\draw (24,0) -- (24,1);
\draw (28,0) -- (28,1);
\crosin{20}{1}{28}{5}
\draw (24,1) .. controls (24,2.2) and (26,1.8) .. (26,3);
\draw (26,3) .. controls (26,4.2) and (24,3.8) .. (24,5);
\crosin{12}{5}{20}{9}
\crosin{16}{5}{24}{9}
\draw (28,5) -- (28,13);
\draw (16,9) -- (16,13);
\draw (20,9) -- (20,13);
\draw (24,9) -- (24,13);
\draw (4,9) .. controls (4,11.4) and (8,10.6) .. (8,13);
\draw (8,9) .. controls (8,11.4) and (12,10.6) .. (12,13);
\draw (12,9) .. controls (12,11.4) and (4,10.6) .. (4,13);
\draw (4,13) -- (4,17);
\crosin{20}{13}{28}{17}
\draw (24,13) .. controls (24,14.2) and (26,13.8) .. (26,15);
\draw (26,15) .. controls (26,16.2) and (24,15.8) .. (24,17);
\crosin{8}{13}{16}{17}
\draw (12,13) .. controls (12,14.2) and (14,13.8) .. (14,15);
\draw (14,15) .. controls (14,16.2) and (12,15.8) .. (12,17);
\ntxt{0}{18}{$0$}
\ntxt{4}{18}{$1$}
\ntxt{8}{18}{$0$}
\ntxt{12}{18}{$0$}
\ntxt{16}{18}{$0$}
\ntxt{20}{18}{$1$}
\ntxt{24}{18}{$1$}
\ntxt{28}{18}{$1$}
\opbox{-0.5}{-2}{8.5}{0}{$0^{(3)}$}
\opbox{11.5}{-2}{20.5}{0}{$1^{(3)}$}
\ntxt{24}{-1.25}{$0$}
\ntxt{28}{-1.25}{$1$}
}
\]
\end{exe}

\begin{lm}
\label{lem:nondiag-zero}
Let $w$ be such that $x\ne y$.
Then the reduced decomposition of $w$ can be chosen in such a way that every basis element $\psi_{w_{0,\ui_\mu,\ui_\mu'}}\psi_w Pe_\lambda$ goes to zero under the quotient map $e R(n\delta)e\to e C(n\delta)e$.
\end{lm}
\begin{proof}
For a composition $\mu=(\mu_1,\ldots, \mu_k)$ and $r\in[1,n]$, let $L(r)\in [1,k]$ be the unique index for which $\mu_1+\ldots+\mu_{L(r)-1}<r\leq \mu_1+\ldots+\mu_{L(r)}$. 
If $x\ne y$, then we can find an index $r\in [1,n]$ such that $L(x(r))>L(y(r))$.
In effect, assume the contrary, that is that for every $r\in [1,n]$ we have $L(x(r))\leqslant L(y(r))$.
Since $\sum_{1\leq r\leq n}L(x(r))=\sum_{1\leq r\leq n}L(y(r))$, this implies $L(x(r))= L(y(r))$ for every $r$.
On the other hand, since $x$ is the shortest element in $\frakS_\mu x$, the values of $L(x(r))$ for each $r$ determine $x$ uniquely, so we must have $x=y$.

Let $r$ be as above.
Further, let $a$ be the position of the $r$-th appearance of ``$0$'' in $\uj_\lambda$ (counting from the left), and $b$ is the position of the $r$-th appearance of ``$1$'' in $\uj_\lambda$.
We have $1\leq a<b\leq 2n$ and $w(a)>w(b)$.
Let $c\in [a,b-1]$ be the unique index such that $(\uj_\lambda)_c=0$ and $(\uj_\lambda)_{c+1}=1$.
We have 
\[
w(c)\geqslant w(c-1)\geqslant\ldots\geqslant w(a+1) \geqslant w(a)\quad>\quad w(b)\geqslant w(b-1)\geqslant\ldots\geqslant w(c+1).
\]
Then we can pick such reduced decomposition of $w$ that on the bottom of the diagram for $\psi_w1_{\uj_\lambda}$ we cross all strands at positions $[a,c]$ with all strands at positions $[c+1,b]$.
This implies that $\psi_w1_{\uj_\lambda}$ is zero in $C(n\delta)$ because it factors through a non-cuspidal idempotent, and thus $\psi_{w_{0,\ui_\mu,\ui_\mu'}}\psi_w Pe_\lambda$ is zero in $e C(n\delta)e$.
\end{proof}

The following example illustrates the proof.
\begin{exe}
Let $\lambda=(3,1)$, $\mu=(2,2)$, and $w=(1,5,6,3,4,7,2,8)$.
Then we have $x=(1,3,4,2)$ and $y=(1,2,3,4)$.
In this case we can take $r=2$, because $L(x(2))=2$ and $L(y(2))=1$.
Then $a=2$, $b=5$, $c=3$, and we should fix a reduced decomposition of $w$ such that on the bottom of the diagram of $\psi_w1_{\uj_\lambda}$ the second and the third strands cross the fourth and the fifth.
The bottom of this diagram will look as follows:
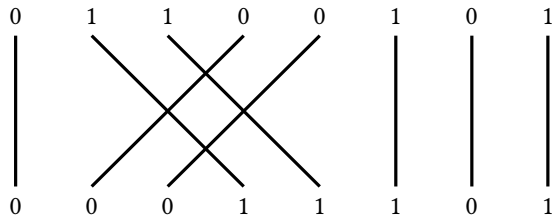
\[ 
\tikz[very thick,baseline={([yshift=+.6ex]current bounding box.center)}]{
    \draw (1,-1) node[below] {\small $0$} -- (1,1) node[above] {\small $0$}; 
    \draw (2,-1) node[below] {\small $0$} -- (4,1) node[above] {\small $0$};
    \draw (3,-1) node[below] {\small $0$} -- (5,1) node[above] {\small $0$};
    \draw (4,-1) node[below] {\small $1$} -- (2,1) node[above] {\small $1$};
    \draw (5,-1) node[below] {\small $1$} -- (3,1)  node[above] {\small $1$};
    \draw (6,-1) node[below] {\small $1$} -- (6,1) node[above] {\small $1$}; 
    \draw (7,-1) node[below] {\small $0$} -- (7,1) node[above] {\small $0$};
    \draw (8,-1) node[below] {\small $1$} -- (8,1) node[above] {\small $1$}; 
}   
\]
On the top of this diagram we have a non-cuspidal sequence $0 1 1 0 0 1 0 1$.
Therefore the element $\psi_w 1_{\uj_\lambda}$ is zero in $C(n\delta)$, because it factors through a non-cuspidal idempotent.
\end{exe}

From now on, we will always assume that the reduced decompositions are chosen as in \cref{lem:nondiag-zero}.
The following statement is an immediate corollary of \cref{lem:basis-gRg,lem:nondiag-zero}.

\begin{corr}
\label{corr:span-gCg}
The algebra $e C(n\delta)e$ is spanned by the set 
$$
\left\{\psi_{w_{0,\ui_\mu,\ui_\mu'}}\psi_{(x)} Pe_\lambda:x\in \frakS_{\mu}\backslash\frakS_n/\frakS_{\lambda}, P\in B_{\ui'_\lambda}\right\}.
$$
\end{corr}

\begin{rmq}
\label{rem:span-gCg-thick}
Consider a basis element $\psi_{w_{0,\ui_\mu,\ui_\mu'}}\psi_{(x)} Pe_\lambda$ as above.
Let $\lambda',\mu'\in\Comp(n)$ be such that $\frakS_{\lambda'}=\frakS_\lambda\cap x^{-1}\frakS_\mu x$ and $\frakS_{\mu'}=\frakS_\mu\cap x\frakS_\lambda x^{-1}$.
Then, for an appropriate choice of a reduced decomposition of $x$, the element $\psi_{(x)}$ can be written as $\psi_{(x)}=\psi_1\psi_2\psi_3$, where $\psi_3=\rmS_\lambda^{\lambda'}$, $\psi_2=\rmR_{\lambda'}^{\mu'}$ and $\psi_{w_{0,\ui_\mu,\ui_\mu'}}\psi_1=\rmM_{\mu'}^\mu$.
In particular, we see that each element of the algebra $e C(n\delta)e$ is a linear combination of diagrams containing splits, merges and polynomials.
\end{rmq}

\subsection{Comparison with sheaves on $\bbP^1$}
In this section we will establish a relation between $eC(n\delta)e$ and the Schur algebra of projective line $\Sc_n=\Sc_n^{\bbP^1}$.

We say that a representation $M\in E_{n\delta}$ is \emph{regular} if there exists an invertible linear combination of two arrows in $\Gamma$.
Regular representations form an open subvariety $E_{n\delta}^{\rm reg}\subset E_{n\delta}$. 
Similarly, we define $\widetilde \bfF_{n\delta}^{\rm reg}\subset \widetilde \bfF_{n\delta}$, $\bfZ_{n\delta}^{\rm reg}\subset \bfZ_{n\delta}$ as inverse images of $E_{n\delta}^{\rm reg}$ under the natural maps $\widetilde \bfF_{n\delta}\to E_{n\delta}$ and $\bfZ_{n\delta}\to E_{n\delta}$ respectively.
We also set $\widetilde \bfF_{\ui}^{\rm reg}=\widetilde \bfF_{\ui}\cap\widetilde \bfF_{n\delta}^{\rm reg}$, $\bfZ_{\ui,\uj}^{\rm reg}=\bfZ_{\ui,\uj}\cap \bfZ_{n\delta}^{\rm reg}$.
Let us make the following standard observation:
\begin{lm}
If a sequence $\ui\in I^{n\delta}$ is non-cuspidal, then $\widetilde \bfF_\ui^{\rm reg}=\emptyset$.
\end{lm}
\begin{proof}
Each sub-representation of a regular representation has dimension vector of the form $m_0\alpha_0+m_1\alpha_1$ such that $m_0\geqslant m_1$.
In particular, a regular representation cannot stabilize a flag of non-cuspidal type.
\end{proof}

\begin{corr}
If $\ui\in I^{n\delta}$ is non-cuspidal, then the idempotent $1_\ui$ lies in the kernel of the pullback map
$$
R(n\delta)\simeq H^*_{G_{n\delta}}(\bfZ_{n\delta})\to H^*_{G_{n\delta}}(\bfZ^{\rm reg}_{n\delta}).
$$ 
In particular, the pullback yields a map $C(n\delta)\to H_*^{G_{n\delta}}(\bfZ^{\rm reg}_{n\delta})$.
\end{corr}

\begin{rmq}\label{rem:map-cusp-Schur-geom}
It will be more important for us to have a truncated version by the idempotents $e_\lambda$.
Let us write $\bfZ_{n\delta,e} = \bigsqcup_{\lambda,\mu \in \Comp(n)} \bfZ_{\ui_\lambda,\ui_\mu}$.
Then the pullback map $e R(n\delta)e\simeq H^*_{G_{n\delta}}(\bfZ_{n\delta,e})\to H^*_{G_{n\delta}}(\bfZ^{\rm reg}_{n\delta,e})$ factors through $e C(n\delta)e \to H^*_{G_{n\delta}}(\bfZ^{\rm reg}_{n\delta,e})$.
\end{rmq}

Recall~\cite{Bei_CSPP1978} that we have an equivalence of bounded derived categories
\begin{equation}\label{DerCatEq}
RHom(\Ocal(-1)\oplus\Ocal,-):D^b(\ona{Coh}\bbP^1)\to D^b(\ona{Rep}\Gamma).
\end{equation}
Restricting this map to torsion sheaves of length $n$ and representations with dimension vector $n\delta$ respectively, we obtain an open embedding of algebraic stacks
\[
\varepsilon:\T_n\hookrightarrow \ona{Rep}_{n\delta}\Gamma,\qquad \cal F\mapsto \left(\begin{tikzcd}\Gamma(\cal F)\ar[r,shift left=0.5ex,"\Gamma(\Ocal(1))"]\ar[r,shift right=0.5ex] & \Gamma(\cal F(1))\end{tikzcd}\right).
\]
Moreover, the image of $\varepsilon$ is precisely the substack of regular representations.
Let $\phi_n:H^*(\ona{Rep}_{n\delta}\Gamma,\bbk)\to H^*(\T_n,\bbk)$ be the corresponding pullback map.

\begin{lm}\label{lm:phiFormula}
The ring homomorphism
\[
\phi_n: \Poll_n^{\fk{S}_n\times \fk{S}_n}\to \bf{P}_n^{\fk{S}_n}
\]
is obtained as a restriction of the following map to invariants: 
\begin{equation}\label{eq:polyRestr}
	\Poll_n\to \bf{P}_n;\qquad u_i\mapsto x_i,\quad v_i\mapsto x_i+c_i.
\end{equation}
\end{lm}
\begin{proof}
By universal coefficients, it suffices to prove the statement for $\bbk = \bbZ$.
We have the following commutative diagram:
\[
\begin{tikzcd}
	\T_1^n\ar[r,"\bigoplus"]\ar[d,hook]& \T_n\ar[d,hook]\\
	(\ona{Rep}_{\delta}\Gamma)^n\ar[r,"\bigoplus"] & \ona{Rep}_{n\delta}\Gamma
\end{tikzcd}
\]
Since pullback is functorial, and pullback along the lower horizontal map realizes the inclusion $\Poll_n^{\fk{S}_n\times \fk{S}_n}\subset \Poll_n$, it suffices to prove our claim for $n=1$.
We will use the notations from \cref{ex:n=1}.
	
The variables $u$, $v$ are first Chern classes of tautological line bundles on $\ona{Rep}_{(1,1)}\Gamma$, which associate to a representation $U\rightrightarrows V$ vector spaces $U$ and $V$ respectively.
We denote by $\pi$ the projection map $\T_1\times \bbP^1\to \T_1$.
By definition of the embedding $\varepsilon$, restriction of these line bundles to $\T_1$ is $\Gamma(\cal E)$ and $\Gamma(\cal E(1))$ respectively, where $\Gamma(-)=R^0\pi_*(-)$ denotes the sheaf of global sections along $\bbP^1$.
Let us compute Chern character of $\Gamma(\cal E(k))$, $k\in\bbZ$, by applying Grothendieck-Riemann-Roch theorem.
\begin{align*}
	\ona{ch}(\Gamma(\cal E(k))) &= \ona{ch}(R\pi_*\cal E(k)) = \pi_*\left( \ona{ch}(\cal E)\ona{ch}(\Ocal(k))\ona{td}(\bbP^1) \right)\\
	& = \pi_*\left( \exp(x) (1-\exp(-c-p))\exp(kp) (1+p) \right)\\
	& = \pi_*\left( \exp(x)(c+p-cp)(1+kp)(1+p) \right) = \pi_*\left( \exp(x)(c+p+kcp)\right)\\
	& = \exp(x)(1+kc) = \exp(x+kc).
\end{align*}
In particular $\ona{ch}(\Gamma(\cal E)) = \exp(x)$ and $\ona{ch}(\Gamma(\cal E(1))) = \exp(x+c)$, so that $\phi_1(u) = x$ and $\phi_1(v) = x+c$.
\end{proof}

Let $\lambda\in\Comp(n)$.
\begin{lm}\label{lm:flagrestr}
We have a natural isomorphism $\FT_\lambda \simeq \T_n\times_{\ona{Rep}_{n\delta}} [\widetilde{\bfF}_{\ui_\lambda}/G_{n\delta}]$.
\end{lm}
\begin{proof}
Let $M\in \ona{Rep}_{n\delta}\Gamma$, and $M'\subset M$ a subrepresentation with dimension vector $n'\delta$, $n'< n$.
Since the vertex $0$ in $\Gamma$ has only incoming arrows, $M'$ uniquely determines a compatible flag $M'_0\subset M'\subset M$, with $\dim M'_0 = n'\alpha_0$.
Therefore the derived equivalence~\eqref{DerCatEq} provides us with an injective map
\begin{gather*}
\FT_\lambda \to [\widetilde{\bfF}_{\ui_\lambda}^{\rm reg}/G_{n\delta}] = \T_n \times_{\ona{Rep}_{n\delta}} [\widetilde{\bfF}_{\ui_\lambda}/G_{n\delta}],\\
(\cal E_1\subset \ldots \subset \cal E_k) \mapsto \left(0 \rightrightarrows \Gamma(\cal E_1(1))\right)\subset \varepsilon(\cal E_1)\subset \ldots \subset \left( \Gamma(\cal E_{k-1}) \rightrightarrows \Gamma(\cal E_k(1)) \right)\subset \varepsilon(\cal E_k).
\end{gather*}

Further, let $M$ be regular, and $M'$ as above.
Since restriction of an isomorphism is an isomorphism, then $M'$ is also regular.
Therefore every flag in $\bfF_{\ui_\lambda}^{\rm reg}$ comes from a flag in $\FT_\lambda$, and the map above is an isomorphism.
\end{proof}

As before, let us define $\phi_\lambda: \Poll_n^{\fk{S}_\lambda\times \fk{S}_\lambda}\to \bf{P}_n^{\fk{S}_\lambda}$ as pullback along the open embedding $\FT_\lambda\subset [\widetilde{\bfF}_{\ui_\lambda}/G_{n\delta}]$.
The following corollary is proved completely analogously to \cref{lm:phiFormula}.
\begin{corr}\label{philamFormula}
The map $\phi_\lambda$ is obtained as a restriction of~\eqref{eq:polyRestr} to the invariants.
\end{corr}

\begin{lm}\label{lm:phiinj}
For $\bfk$ a field of characteristic zero, the maps $\phi_n$, $\phi_\lambda$ are surjective.
\end{lm}
\begin{proof}
It is clearly enough to prove the statement for $\phi_n$.
By a theorem of Weyl~\cite[II.3]{Wey_CGTI1939}, the ring ${\bf P}_n^{\frakS_n}$ is generated by elements
\[
p_{k,0} = \sum_{i} x_i^k,\quad p_{k,1} = \sum_{i} c_ix_i^k.
\]
However, we have
\[
p_{k,0} = \phi_n\left(\sum_{i} u_i^k\right),\quad p_{k,1} = \frac{1}{k}\sum_i \left( (x_i+c_i)^k - x_i^k \right) = \frac{1}{k}\phi_n\left( \sum_{i} v_i^k - \sum_{i} u_i^k \right),
\]
and so we may conclude.
\end{proof}

Another immediate corollary from \cref{lm:flagrestr} is that we have $\FT_\mu\times_{\T_n} \FT_\lambda\simeq [\bfZ^{\rm reg}_{\ui_\mu,\ui_\lambda}/G_{n\delta}]$. 
The resulting restriction map $eR(n\delta)e\to \Sc_n$ is a homomorphism of algebras by smooth base change and functoriality of pullbacks.
Furthermore, it descends to a homomorphism $\Phi_n:eC(n\delta)e\to \Sc_n$ by \cref{rem:map-cusp-Schur-geom}.

\begin{rmq}
Note that both $eC(n\delta)e$ and $\Sc_n$ are defined over any commutative ring $\bbk$, in particular $\bbk=\bb F_p$ finite field.
Since $\Phi_n$ is essentially a pullback along an open embedding, it is also defined for any $\bbk$.
This will become important at the end of this section.
\end{rmq}

\begin{prop}
\label{prop:map-Phi}
The algebra homomorphism $\Phi_n$ sends each thick diagram in $e C(n\delta)e$ to the same diagram in $\Sc_n$, replacing each polynomial $P\in \Poll_a^{(\frakS_a)^2}$ on a strand of thickness $a$ by $\phi_a(P)$.
\end{prop}
\begin{proof}
Follows from \cref{lm:geomSMcusp,lm:phiFormula}.
\end{proof}

Restricting to polynomial operators, we have the following commutative square:
\begin{equation}\label{Jn-square}
\begin{tikzcd}
	\Poll_n^{\frakS_n^2}\ar[r,"\phi_n"]\ar[d,"\iota_n"] & \bf{P}_n^{\frakS_n}\ar[d,hook]\\
	eC(n\delta)e\ar[r,"\Phi_n"]& \Sc_n 
\end{tikzcd}
\end{equation}
Let $J_n$ be the kernel of $\iota_n$.
It is clear that $J_n\subset \Ker\phi_n$; we will show in \cref{lem:Jn=Ker} that this is actually an equality.
In order to prove this, we will need some preparations. 

\subsection{Shuffle products}
Consider the product 
$*\colon \Poll_n^{(\frakS_n)^2}\times \Poll_m^{(\frakS_m)^2}\to \Poll_{n+m}^{(\frakS_{n+m})^2}$ given by 

\begin{equation}\label{eq:fusion-pic}
\begin{aligned}
\tikz[thick,xscale=.25,yscale=.25]{
	\ntxt{0}{-0.5}{$n+m$}
	\draw (0,1) -- (0,5);
	\opbox{-2}{5}{2}{7}{$P*Q$}
	\draw (0,7) -- (0,11);
	\ntxt{0}{12}{$n+m$}
	
	\ntxt{4.5}{6}{$=$}
	
	\draw (11,1) -- (11,2);
	\ntxt{11}{-0.5}{$n+m$}
	\dsplit{8}{2}{14}{5}
	\ntxt{8}{3}{$n$}
	\ntxt{14}{3}{$m$}
	\opbox{7}{5}{9}{7}{$P$}
	\opbox{13}{5}{15}{7}{$Q$}
	\dmerge{8}{7}{14}{10}
	\draw (11,10) -- (11,11);
	\ntxt{11}{12}{$n+m$}
	\ntxt{16}{6}{.}
}
\end{aligned}   
\end{equation}

Thanks to the proof of \cref{lm:geomSMcusp}, we have the following expression for shuffle product:
\[
P*Q = \partial^u_{w_{0,a,b}}\partial^v_{w_{0,a,b}}\left( (P\otimes Q)\prod_{i=1}^a\prod_{j=a+1}^b (v_i-u_j)^2 \right),
\]

We also consider the shuffle product 
$$
*\colon \bf{P}_n^{\frakS_n}\times \bf{P}_m^{\frakS_m}\to \bf{P}_{n+m}^{\frakS_{n+m}}
$$
given by the same picture \eqref{eq:fusion-pic}, but using diagrammatic calculus in $\Sc_n$ instead of diagrammatic calculus in $e R(n\delta) e$.

\begin{lm}
The map
\[
\bigoplus_n \phi_n: \bigoplus_n \Poll_n^{(\frakS_n)^2}\to \bigoplus_n \bf{P}_n^{\frakS_n}
\]
is a homomorphism of algebras (with respect to the operations $*$).
\end{lm}
\begin{proof}
Follows from the definitions and \cref{prop:map-Phi}.
\end{proof}

For a polynomial $f\in \Poll_n$, denote by $\ev(f)$ the polynomial in $\bfk[v_1,\ldots,v_n]$ obtained from $f$ after evaluation $u_1=u_2=\ldots=u_n=0$.
Set $D_n^u=\partial^u_1\partial^u_2\ldots\partial^u_{n-1}$ and $D_n^v=\partial^v_1\partial^v_2\ldots\partial^v_{n-1}$, where $\partial_i^u$ denote Demazure operators in variables $u_1,\ldots,u_n$, and $\partial_i^v$ are defined analogously.

\begin{rmq}
\label{rem:c(Du)-coeff}
Note that $D^u_n(u_n^{n-1})=(-1)^{n-1}$. 
This identity allows to simplify expressions of the form $\ev(D^u_n(P))$, where $P\in\bfk[v_1,\dots,v_n]^{\frakS_n}[u_n]$.
Write $P=\sum_ru_n^rP_r$, where $P_r\in\bfk[v_1,\ldots,v_n]^{\frakS_n}$. Then we have
$$
\ev(D^u_n(P))=(-1)^{n-1}(P_{n-1}).
$$
\end{rmq}

Denote by $\sigma_k^{(n)}$ the $k$-th elementary symmetric polynomial on the variables $v_1,\ldots,v_n$; we use the convention $\sigma^{(n)}_0=1$.
We also denote the unit in $\Poll_n$ by $1_n$.

\begin{lm}
\label{lem:1*f}
Assume $1\leqslant k\leqslant n$. We have $\ev(1_{n-1}*(v-u)u^{k-1})=(-1)^{k-1}\sigma_k^{(n)}$.
\end{lm}
\begin{proof}
We have 
$$
1_{n-1}*(v-u)u^{k-1}=D_n^uD_n^v\left[(v_1-u_n)^2(v_2-u_n)^2\ldots(v_{n-1}-u_n)^2 (v_n-u_n)u_n^{k-1}\right].
$$

First, let $k=1$. We have
\begin{align*}
\ev(1_{n-1}*(v-u)) & = \ev\left(D_n^uD_n^v{\left[(v_1-u_n)^2(v_2-u_n)^2\ldots(v_{n-1}-u_n)^2 (v_n-u_n)\right]}\right)\\
 & = \ev(D_n^u[(v_1-u_n)(v_2-u_n)\ldots (v_n-u_n)])\\
 & =v_1+v_2+\ldots+v_n=\sigma_1^{(n)}.
\end{align*}

Now, assume $k>1$.
We fix $k$ and proceed by induction on $n$.
If $n=k$, we have
\begin{align*}
\ev(1_{n-1}*(v-u)u^{k-1}) & = \ev\left(D_n^uD_n^v{\left[(v_1-u_n)^2(v_2-u_n)^2\ldots(v_{n-1}-u_n)^2 (v_n-u_n)u_n^{n-1}\right]}\right)\\
 & = (-1)^{n-1}(D_n^v(v_1^2v_2^2\ldots v_{n-1}^2 v_n))\\
 & = (-1)^{n-1}(v_1v_2\ldots v_{n-1} v_n)=(-1)^{n-1}\sigma_n^{(n)}.
\end{align*}
where the second equality follows from \cref{rem:c(Du)-coeff}.

\medskip
Now assume $n>k>1$.
Let us write $Q(u,v) = (v_1-u_n)^2(v_2-u_n)^2\ldots(v_{n-2}-u_n)^2$ for brevity.
We have
\begin{align*}
\partial^v_{n-1}\left[Q(u,v)(v_{n-1}-u_n)^2 (v_n-u_n)u_n^{k-1}\right] & = Q(u,v)(v_{n-1}-u_n)(v_n-u_n)u_n^{k-1} \partial_{n-1}^v (v_{n-1}-u_n)\\
 & = Q(u,v)(v_{n-1}-u_n)(v_n-u_n)u_n^{k-1}.
\end{align*}
This implies
\begin{align*}
\ev(1_{n-1}*(v-u)u^{k-1}) & = \ev\left(D_n^uD_n^v{\left[Q(u,v)(v_{n-1}-u_n)^2 (v_n-u_n)u_n^{k-1}\right]}\right)\\
& = \ev\left(D_n^u D_{n-1}^v{\left[Q(u,v)(v_{n-1}-u_n) (v_n-u_n)u_n^{k-1}\right]}\right)\\
& = -\ev\left(D_{n-1}^u D_{n-1}^v{\left[Q(u,v)(v_{n-1}-u_{n-1})(v_n-u_{n-1})u_{n-1}^{k-2}\right]}\right)\\
& = -v_n\ev(1_{n-2}*(u-v)u^{k-2})+\ev(1_{n-2}*(v-u)u^{k-1})\\
& = -v_n(-1)^{k-2}\sigma^{(n-1)}_{k-1}+(-1)^{k-1}\sigma^{(n-1)}_{k}=(-1)^{k-1}\sigma^{(n)}_k,
\end{align*}
where the third equality follows from \cref{rem:c(Du)-coeff}. 
\end{proof}

For each positive integer $k$ we set $\widetilde f_k=(v-u)u^{k-1}\in \Poll_1$, and $\widetilde t_{n,k}=1_{n-1}*{\widetilde f}_k\in \Poll_n^{(\frakS_n)^2}$.
The following proposition follows from \cref{lem:1*f}.
\begin{prop}
\label{prop:gen-1*f}
The commutative ring $\Poll_n^{(\frakS_n)^2}$ is generated by $\bfk[u_1,\ldots,u_n]^{\frakS_n}$ together with elements $\widetilde t_{n,k}$ for $k\in [1;n]$.
\end{prop}

\subsection{Spanning set of $\Poll_n^{(\frakS_n)^2}/J_n$}

\begin{lm}
\label{lem:prod-of-t}
For each positive integer $r$, we have the following equality in $\Poll_n^{(\frakS_n)^2}/J_n$:
$$
\widetilde t_{n,k_1}\cdot \widetilde t_{n,k_2}\cdot \ldots \cdot \widetilde t_{n,k_r}=
\begin{cases}
1_{n-r}* \widetilde f_{k_1}* \widetilde f_{k_2}*\ldots * \widetilde f_{k_r}, &\mbox{ if } r\leqslant n,\\
0, &\mbox{ if }r>n.
\end{cases} 
$$
\end{lm}
\begin{proof}
First, we prove the case $r\leqslant n$.
We prove the statement by induction on $r$.
It is enough to prove the following equality for $r<n$:
$$
(1_{n-r}* \widetilde f_{k_1}*\ldots * \widetilde f_{k_r})\cdot (1_{n-1}*\widetilde f_{k})=1_{n-r-1}* \widetilde f_k * \widetilde f_{k_1}* \widetilde f_{k_2}*\ldots * \widetilde f_{k_r}.
$$ 
We can rewrite this as the following diagrammatic identity in $e C(n\delta)e$:
\begin{equation}\label{eq:deer}
\begin{aligned}
\tikz[thick,xscale=.25,yscale=.25]{
	\draw[very thin,dashed] (-3,11) -- (30,11);
	
	\ntxt{5}{-2}{$n$}
	\dsplit{3}{-1}{7}{1}
	\ntxt{1.5}{0}{$n-1$}
	\ntxt{7.5}{0}{$1$}
	\draw (3,1) -- (3,4);
	\draw (7,1) -- (7,2);
	\strex{7}{1.4}
	\opbox{6}{2}{8}{4}{$u^k$}
	\dmerge{3}{4}{7}{6}
	\draw (5,6) -- (5,7);
	\dsplit{0}{7}{10}{10}
	\draw (5,7) .. controls (5,8.5) and (3,8.5) .. (3,10);
	\ntxt{-1}{8.5}{$n-r$}
	\ntxt{4.8}{8.5}{$1$}
	\ntxt{10}{8.5}{$1$}
	\draw (0,10) -- (0,13.5);
	\draw (3,10) -- (3,11.5);
	\strex{3}{10.4}
	\draw (10,10) -- (10,11.5);
	\strex{10}{10.4}
	\opbox{2}{11.5}{4}{13.5}{$u^{k_1}$}
	\ntxt{6.5}{10}{$\ldots$}
	\opbox{9}{11.5}{11}{13.5}{$u^{k_r}$}
	\dmerge{0}{13.5}{10}{15.5}
	\draw (3,13.5) .. controls (3,14.5) and (5,14.5) .. (5,15.5);
	
	\ntxt{13}{7}{$=$}
	
	\ntxt{22}{-2}{$n$}
	\draw (22,-1) -- (22,0);
	\draw (22,0) .. controls (22,2.5) and (15.5,2.5) .. (15.5,5);
	\draw (22,0) .. controls (22,2.5) and (18.5,2.5) .. (18.5,5);
	\draw (22,0) .. controls (22,5) and (20,5) .. (20,10);
	\draw (22,0) .. controls (22,5) and (27,5) .. (27,10);
	\draw (15.5,5) -- (15.5,8);
	\draw (18.5,5) -- (18.5,6);
	\strex{18.5}{5.4}
	\opbox{17.5}{6}{19.5}{8}{$u^k$}
	\ntxt{18}{1}{$n-r-1$}
	\ntxt{19.75}{3.5}{$1$}
	\ntxt{22}{3.5}{$1$}
	\ntxt{24}{3.5}{$1$}
	\dmerge{15.5}{8}{18.5}{10}
	\draw (17,10) -- (17,13.5);
	\draw (20,10) -- (20,11.5);
	\strex{20}{10.4}
	\draw (27,10) -- (27,11.5);
	\strex{27}{10.4}
	\opbox{19}{11.5}{21}{13.5}{$u^{k_1}$}
	\ntxt{23.5}{10}{$\ldots$}
	\opbox{26}{11.5}{28}{13.5}{$u^{k_r}$}
	\dmerge{17}{13.5}{27}{15.5}
	\draw (20,13.5) .. controls (20,14.5) and (22,14.5) .. (22,15.5);
	
	\ntxt{30}{7}{$,$}
}
\end{aligned}   
\end{equation}
where a cross on a strand means the polynomial $v-u$.
It suffices to prove an equality of parts below the dashed line.

First, note that by \cref{rem:polyn-in-KLR:a} both pictures are equal in $e R(n\delta)e$ to some polynomials in $\Poll_n^{(\frakS_{\lambda})^2}$, where $\lambda=(n-r,1,1,\ldots,1)$.
The left picture gives the polynomial 
\[
(v_{n-r+1}-u_{n-r+1})\ldots (v_n-u_n) D^u_nD^v_n[(v_1-u_n)^2\ldots (v_{n-1}-u_n)^2(v_n-u_n)u_n^k],
\]
while the right picture gives 
\[
(v_{n-r+1}-u_{n-r+1})\ldots (v_n-u_n)D^u_{n-r}D^v_{n-r}[(u_{n-r}-v_1)^2\ldots (u_{n-r}-v_{n-r-1})^2(u_{n-r}-v_{n-r})u_{n-r}^k].
\]
We want to show that the images of these polynomials in $e_\lambda C(n\delta)e_{(n)}$ coincide.
By \cref{rem:polyn-in-KLR:b}, it is enough to show that these polynomials are equal modulo the ideal generated by $(v_{n-r+1}-u_{n-r+1})^2,\ldots, (v_{n}-u_{n})^2$.
Note that both polynomials already contain the factor $(v_{n-r+1}-u_{n-r+1})\ldots (v_n-u_n)$.
Therefore, it is enough to prove the congruence
\[
D^u_nD^v_n[(v_1-u_n)^2\ldots (v_{n-1}-u_n)^2(v_n-u_n)u_n^k]
\equiv
D^u_{n-r}D^v_{n-r}[(u_{n-r}-v_1)^2\ldots (u_{n-r}-v_{n-r-1})^2(u_{n-r}-v_{n-r})u_{n-r}^k]
\]
modulo the ideal $(v_{n-r+1}-u_{n-r+1},\ldots, v_{n}-u_{n})$.
Indeed, we have
\begin{align*}
D^u_nD^v_n&[(v_1-u_n)^2\ldots (v_{n-1}-u_n)^2(v_n-u_n)u_n^k]\\
 & = D^u_nD^v_{n-r}[(v_1-u_n)^2\ldots (v_{n-r-1}-u_n)^2(v_{n-r}-u_n)\cdots(v_n-u_n)u_n^k]\\
 & \equiv D^u_{n-r}D^v_{n-r}[(v_1-u_{n-r})^2\ldots (v_{n-r-1}-u_{n-r})^2(v_{n-r}-u_{n-r})u_{n-r}^k].
\end{align*}
Let us justify the congruence between the second and the third line above.
When we apply the sequence of Demazure operators $\partial^u_{n-r}\ldots \partial^u_{n-1}$ to  
\[
u_n^k(v_1-u_n)^2\ldots (v_{n-r-1}-u_n)^2(v_{n-r}-u_n)\cdots(v_n-u_n)
\]
(the order is important!) and use Leibniz rule $\partial^u_t(fg)=\partial^u_t(f)g+s^u_r(f)\partial^u_t(g)$, the only situation when the result is not in the ideal appears when 
\begin{itemize}
\item Demazure operator $\partial^u_{n-1}$ hits $(v_n-u_n)$ (and applies $s^u_{n-1}$ to other factors),
\item Demazure operator $\partial^u_{n-2}$ hits $(v_{n-1}-u_{n-1})$ (and applies $s^u_{n-2}$ to other factors),

and so on, until
\item Demazure operator $\partial^u_{n-r}$ hits $(v_{n-r+1}-u_{n-r+1})$ (and applies $s^u_{n-r}$ to other factors). 
\end{itemize}


Now, let us prove the case $r>n$.
For this we need to show that 
$$
(\widetilde f_{k_1}*\ldots * \widetilde f_{k_n})\cdot (1_{n-1}*\widetilde f_{k})=0.
$$ 
For this, it is enough to show that the left hand side of \eqref{eq:deer} is zero in $e C(n\delta)e$ for $r=n$. This left hand side is given by the polynomial
$$
(v_{1}-u_{1})\ldots (v_n-u_n) D^u_nD^v_n[(v_1-u_n)^2\ldots (v_{n-1}-u_n)^2(v_n-u_n)u_n^k].
$$
We have 
$$
D^u_nD^v_n[(v_1-u_n)^2\ldots (v_{n-1}-u_n)^2(v_n-u_n)u_n^k]=D^u[(v_1-u_n)\ldots (v_{n-1}-u_n)(v_n-u_n)u_n^k]\equiv 0.
$$
The congruence is justified in the same way as in the previous case.
\end{proof}

\begin{corr}
\label{cor:f-comm-Pol/J}
We have $\widetilde f_{k_1}*\widetilde f_{k_2}=\widetilde f_{k_2}*\widetilde f_{k_1}$ in $\Poll_2^{(\frakS_2)^2}/J_2$.
\end{corr}

For each positive integer $n$, let us fix a basis $\widetilde B_n$ of $\bfk[u_1,\ldots,u_n]^{\frakS_n}$.
\begin{lm}
\label{lem:span-Pol/J}
The algebra $\Poll_n^{(\frakS_n)^2}/J_n$ is spanned by the following set
$$
\{P*\widetilde f_{k_1}* \widetilde f_{k_2}*\ldots *\widetilde f_{k_r} : r\in[0;n], P\in \widetilde B_{n-r},~ 0<k_1\leqslant k_2\leqslant\ldots\leqslant k_r\}.
$$
\end{lm}
\begin{proof}
By \cref{prop:gen-1*f}, $\Poll_n^{(\frakS_n)^2}$ in generated (as an algebra) by $\bfk[u_1,\ldots,u_n]^{\frakS_n}$ and by the elements $\widetilde t_{n,k}$ for $k\in \bbZ_{>0}$.
Then \cref{lem:prod-of-t} implies that $\Im\phi_n$ is generated as a $\bfk[u_1,\ldots,u_n]^{\frakS_n}$-module by elements of the form $1_{n-r}*\widetilde f_{k_1}*\widetilde f_{k_2}*\ldots *\widetilde f_{k_r}$.
When we multiply $1_{n-r}*\widetilde f_{k_1}*\widetilde f_{k_2}*\ldots *\widetilde f_{k_r}$ by an element of $\bfk[u_1,\ldots,u_n]^{\frakS_n}$, we get an linear combination of elements of the form $P* \widetilde f_{k'_1}*\widetilde f_{k'_2}*\ldots *\widetilde f_{k'_r}$, where $P\in \bfk[u_1,\ldots,u_{n-r}]^{\frakS_{n-r}}$ and $k'_i\geqslant k_i$.
Note that we can also reorder the factors using \cref{cor:f-comm-Pol/J}.
This implies that the desired set spans $\Poll_n^{(\frakS_n)^2}/J_n$.
\end{proof}

From here until \cref{sec:no-inj-surj}, let us assume that $\bbk$ is either a field of characteristic $0$ or $\bbZ$.
Set $f_k=cx^{k-1}\in \bf{P}_1$ and $t_{n,k}=1_{n-1}*f_k\in \bf{P}_n^{\frakS_n}$. 
Under $\phi_n$, the basis $\widetilde B_n$ of $\bfk[u_1,\ldots,u_n]^{\frakS_n}$ defines an analogous basis of $\bfk[x_1,\ldots,x_n]^{\frakS_n}$, which we denote by the same symbol.

\begin{lm}
\label{lem:basis-Imphi}
The following set is a $\bbk$-basis of $\Im\phi_n$:
\[
\cal B:=\{P*f_{k_1}*f_{k_2}*\ldots *f_{k_r} : r\in[0;n], P\in \widetilde{B}_{n-r}, 0<k_1\leqslant k_2\leqslant\ldots\leqslant k_r\}.
\]
\end{lm}
\begin{proof}

The fact that the set $\cal B$ spans $\Im\phi_n$ follows from \cref{lem:span-Pol/J} together with commutative square~\eqref{Jn-square}.
Next, we prove linear independence.
It is enough to do this for $\bfk=\bbQ$.

Let us choose a specific basis of $\bbQ[x_1,\ldots,x_n]^{\frakS_n}$. 
Namely, let
\[
\mathcal P_n=\{\lambda=(\lambda_1,\ldots,\lambda_n)\in \bbZ^n ;~0\leqslant\lambda_1\leqslant\lambda_2\leqslant\ldots\leqslant\lambda_n\},
\] 
and write $B_n=\{m_\lambda : \lambda\in \mathcal P_n\}$, where $m_\lambda=\sum_{w\in\frakS_n}x_{w(1)}^{\lambda_1}\ldots x_{w(n)}^{\lambda_n}$ are the monomial symmetric functions.

For any $t\geqslant 0$, consider the element $e_t=x^t\in \bf{P}_1$.
Then the set 
\[
\mathcal B'=\{e_{t_1}*e_{t_2}*\ldots *e_{t_{n-r}}*f_{k_1}*f_{k_2}*\ldots * f_{k_r};~r\in[0;n],~ 0\leqslant t_1\leqslant t_2\leqslant\ldots\leqslant t_{n-r},~0 < k_1\leqslant k_2\leqslant\ldots\leqslant k_{r} \}
\] 
is a basis in $\bf{P}_n^{\frakS_n}$, see~\cite{FR_CHAK2019} for details.
Consider a lexicographic order on $\mathcal B$, where we assume 
\[
e_0> e_1> e_2>\ldots >f_1> f_2> f_3>\ldots 
\]
Let $\lambda\in \mathcal P_n$.
An easy induction argument shows that $1_n=\frac{1}{n!}(1*1*\ldots * 1)$.
We can therefore write
\[
m_\lambda=\frac{1}{n!} (1*1*\ldots * 1)m_\lambda=\frac{1}{n!}\sum_{w\in \frakS_n}e_{\lambda_{w(1)}}*\ldots *e_{\lambda_{w(n)}},
\]
where we have used that $\rmS_{n}^{1^n}\in \Sc_n$ commutes with any polynomial in $\bf{P}_n^{\frakS_n}$.
Next, we apply the reordering relations in~\cite[Theorem~1]{FR_CHAK2019} to write $m_\lambda$ in terms of $\mathcal B'$.
We get 
\[
m_\lambda=e_{\lambda_{1}}*e_{\lambda_2}*\ldots *e_{\lambda_n} + \mbox{ lower terms}.
\]
Similarly, for $\lambda\in\mathcal P_{n-r}$, we can write each element $m_\lambda*f_{k_1}*\ldots *f_{k_r}\in \mathcal B$ as 
\[
m_\lambda*f_{k_1}*\ldots *f_{k_r}=e_{\lambda_{1}}*e_{\lambda_2}*\ldots *e_{\lambda_{n-r}}*f_{k_1}*\ldots *f_{k_r} + \mbox{ lower terms}.
\]
Therefore the transition matrix from $\mathcal B$ to $\mathcal B'$ is upper triangular.
This implies linear independence of $\mathcal B$.
\end{proof}

\subsection{Realization of $eC(n\delta)e$ inside $\Sc_n$}

\begin{lm}
\label{lem:Jn=Ker}
We have $J_n=\Ker\phi_n$.
\end{lm}
\begin{proof}
The inclusion $J_n\subset \Ker\phi_n$ follows from diagram~\eqref{Jn-square}.
Next, the map $\Poll_n^{(\frakS_n)^2}/J_n\to \bf{P}_n^{\frakS_n}$ induced by $\phi_n$ takes the generating set of $\Poll_n^{(\frakS_n)}/J_n$ from \cref{lem:span-Pol/J} to the basis of $\Im\phi_n$ from \cref{lem:basis-Imphi}.
This implies that this map is injective, so that $J_n=\Ker\phi_n$.
\end{proof}

Fix $\lambda\in\Comp(n)$.
Denote by $J_\lambda$ the kernel of the map
$$
\iota_\lambda: \Poll_n^{(\frakS_\lambda)^2}\to e C(n\delta)e,\qquad P\mapsto e_\lambda Pe_\lambda. 
$$
 
\begin{corr}
\label{cor:J=Ker-lambda}
We have $J_\lambda=\Ker\phi_\lambda$.
\end{corr}
\begin{proof}
Thanks to~\cref{prop:map-Phi}, we have the following commutative diagram:
\[
\begin{tikzcd}
	\Poll_n^{\frakS_\lambda^2}\ar[r,"\phi_\lambda"]\ar[d,"\iota_\lambda"] & \bf{P}_n^{\frakS_{\lambda}}\ar[d,hook]\\
	eC(n\delta)e\ar[r,"\Phi_n"]& \Sc_n 
\end{tikzcd}
\]	
Since the rightmost vertical map is injective, it follows that $J_\lambda\subset \Ker\phi_\lambda$.

Let us prove the opposite inclusion.
Write $\lambda=(\lambda_1,\ldots,\lambda_k)$.
By definition of $C(n\delta)$, the map $\iota_\lambda$ can be written as a composition
\[
\Poll_n^{\frakS_\lambda^2}\xra{J_{\lambda_1}\otimes \ldots \otimes J_{\lambda_k}} \bigotimes_i eC(\lambda_i\delta)e \xra{\otimes} eC(n\delta)e.
\]
On the other hand, we have $\phi_{\lambda} = \phi_{\lambda_1}\otimes \phi_{\lambda_2}\otimes \ldots\otimes \phi_{\lambda_k}$ by \cref{philamFormula}, so that
\[
\Ker\phi_\lambda=\Ker_{\lambda_1}\otimes \Poll_{\lambda_2}^{(\frakS_{\lambda_2})^2}\otimes\ldots \otimes \Poll_{\lambda_k}^{(\frakS_{\lambda_k})^2}+\ldots+\Poll_{\lambda_1}^{(\frakS_{\lambda_1})^2}\otimes \Poll_{\lambda_2}^{(\frakS_{\lambda_2})^2}\otimes\ldots \otimes \Ker\phi_{\lambda_k}.
\]
Now, the inclusion $\Ker\phi_\lambda\subset J_\lambda$ follows from \cref{lem:Jn=Ker}.
\end{proof}

\begin{rmq}
\label{rem:span-gCg}
In view of \cref{cor:J=Ker-lambda}, the statement of \cref{corr:span-gCg} remains true if we replace $B_{\ui'_\lambda}$ by a basis of $\Poll_n^{(\frakS_{\lambda'})^2}/\Ker\phi_{\lambda'}$.
Note also that $\Poll_n^{(\frakS_{\lambda'})^2}/\Ker\phi_{\lambda'}\simeq \Im\phi_{\lambda'}$ is free over $\bfk$ by \cref{lem:basis-Imphi}.
\end{rmq}

\begin{prop}\label{prop:Phininj}
Let $\bfk$ be a field of characteristic zero or $\bfk=\bbZ$.
Then $\Phi_n$ is injective and its image is spanned by split/merge diagrams, whose strands of thickness $k$ are decorated by elements of $\Im\phi_k\subset P_k^{\frakS_k}$.
In particular, if $\bfk$ is a field of characteristic zero, then $\Phi_n$ is bijective.
\end{prop}
\begin{proof}
The statement about the image of $\Phi_n$ follows from \cref{rem:span-gCg-thick} and \cref{prop:map-Phi}. 
For injectivity, note that $\Phi_n$ takes the spanning set of $e C(n\delta)e$ from \cref{corr:span-gCg} and \cref{rem:span-gCg} to a linear independent set in $\Sc_n$, see \cref{lm:basis-Schur-zigzag}. Therefore the spanning set of $e C(n\delta)e$ is automatically a basis, and $\Phi_n$ is injective.
Finally, surjectivity in characteristic zero follows from \cref{lm:phiinj}.
\end{proof}

\begin{corr}
The algebra $eC(n\delta)e$ is generated by elementary splits and merges, polynomial $(v-u)$ on thin strands, and symmetric polynomials $\bbk[u_1,\ldots, u_k]^{\frakS_k}$ on strands of thickness $k$.
\end{corr}

\subsection{Counterexamples to injectivity/surjectivity of $\Phi_n$}\label{sec:no-inj-surj}
Let us begin by providing a certain geometric meaning for the image of $\phi_n$.
\begin{prop}\label{prop:phi=TH}
Let $\bbk = \bbZ$.
The $H_{G_n}$-submodule of $H^*(\T_n)$ generated by the tautological ring $TH^*(\T_n)$ coincides with the image of $\phi_n$.
\end{prop}
\begin{proof}
By definition of the moduli stack $\ona{Rep}_{n\delta}\Gamma$, equivariant parameters $u_i$, $v_i$ are the Chern roots of tautological vector bundles $U$, $V$ respectively.
In particular, restricting to $\T_n$ we see that the image of $\phi_n$ is generated over $\bbZ$ by the Chern classes of $\Gamma(\cal E)$ and $\Gamma(\cal E(1))$.
	
Let us write $c^0(\cal E) = \sum_i c_{i,0}(\cal E)$ and $c^1(\cal E) = \sum_i c_{i,1}(\cal E)$; we have $c(\cal E) = c^0(\cal E) + c^1(\cal E)p$.
Since $p^2 = 0$, it follows from the definition of Chern character that $\ona{ch}(\cal E) = \ona{ch}^0(\cal E) + pQ$, where $Q$ is some class in $H^*(\T_d)$, and $\ona{ch}^0(\cal E)=\ona{ch}(c^0(\cal E))$. 
	
By Grothendieck-Riemann-Roch, we have
\begin{equation*}
	\ona{ch}(\Gamma(\cal E(1)))	- \ona{ch}(\Gamma(\cal E)) = \pi_* \left( \ona{ch}(\cal E)\ona{td}(\bbP^1)p \right) =  \pi_* \left( p\ona{ch}(\cal E) \right) = \ona{ch}^0(\cal E).
\end{equation*}
In particular, we see that $c(\Gamma(\cal E(1))) = c(\Gamma(\cal E))c^0(\cal E)$.
Since the coefficients of $c(\Gamma(\cal E))$ are precisely the generators of $H_{G_n}$, we conclude that $\Im\phi_n\subset H^*_{G_n}\cdot TH^*(\T_n)$.
	
On the other hand, applying $\varepsilon$ to the standard resolution of the path algebra of $\Gamma$ (see~\cite[(1.2)]{BK_MRA1999a}) produces the following resolution of $\cal E$:
\[
0\to \Gamma(\cal E(-1))\otimes H^0(\Ocal(1))\otimes \Ocal\to \Gamma(\cal E(-1))\otimes \Ocal(1)\oplus \Gamma(\cal E)\otimes \Ocal\to \cal E\to 0.
\]
In particular, we have 
\begin{equation}\label{eq:ChernKC}
	c(\cal E) = \frac{c(\Gamma(\cal E))c(\Gamma(\cal E(-1))\otimes \Ocal(1))}{c(\Gamma(\cal E(-1)))^2} = \frac{c(\Gamma(\cal E))}{c(\Gamma(\cal E(-1))\otimes \Ocal(-1))}.
\end{equation}
The denominator is a polynomial in the Chern classes of $\Gamma(\cal E(-1))$ and $p\in H^2(\bbP^1)$.
Moreover, we have
\[
c(\Gamma(\cal E(-1))) = c(\Gamma(\cal E))^2/c(\Gamma(\cal E(1))).
\]
Therefore the formula \eqref{eq:ChernKC} shows that the Künneth-Chern classes $c_{i,j}(\cal E)$ are expressed as polynomials of the Chern classes of $\Gamma(\cal E)$ and $\Gamma(\cal E(1))$.
Thus $H_{G_n}\cdot TH^*(\T_n)\subset \Im\phi_n$, and we may conclude.
\end{proof}

Now let $\bfk=\bbZ$, and consider the map $\phi_2\colon \bbZ[u_1,u_2,v_1,v_2]^{\frakS_2^2}\to (\bbZ[x_1,x_2,c_1,c_2]/(c_1^2,c_2^2))^{\frakS_2}$.
The element $c_1c_2$ does not lie in $\Im\phi_2$ by \cref{ex:d2nonsurj} and \cref{prop:phi=TH}.
However, note that $2c_1c_2=(c_1+c_2)^2=\phi_2((v_1+v_2-u_1-u_2)^2)\in\Im\phi_2$.

Non-surjectivity of $\phi_2$ automatically implies non-surjectivity of $\Phi_2$.
In effect, $\Im\phi_2\subset \bf{P}_2^{\frakS_2}$ can be identified with $\Im\Phi_2\cap e_{(2)}\Sc_2e_{(2)}\simeq \bf{P}_2^{\frakS_2}$. 
Moreover, applying universal coefficients we obtain that $\Phi_2$ is not surjective if $\bbk = \bb F_2$.

\begin{lm}
For $\bfk=\mathbb F_2$ we have $J_2\subsetneq \Ker\phi_2$.
\end{lm}
\begin{proof}
The inclusion $J_2\subset \Ker\phi_2$ holds by the same argument as in \cref{lem:Jn=Ker}.
Let us show that this inclusion is strict.
First, note that both $J_2$ and $\Ker\phi_2$ are homogeneous ideals in $\Poll_2^{\frakS_2^2}$.
Let us add the base ring ($\bbZ$ or $\mathbb F_p$) to our notation as a superscript.
Since $\phi_2((v_1+v_2-u_1-u_2)^2)=2c_1c_2=0$ over $\mathbb F_2$, the ideal $\Ker\phi^{\mathbb F_2}_2$ contains a generator of degree $2$.
Thus in order to prove that the inclusion $J^{\mathbb F_2}_2\subset \Ker\phi^{\mathbb F_2}_2$ is strict, it suffices to show that $J^{\mathbb F_2}_2$ is generated by elements of degrees strictly greater than $2$.

Since $\Ker(e R_{\mathbb F_2}(n\delta)e\to e C_{\mathbb F_2}(n\delta)e)$ is the clearly reduction modulo $2$ of $\Ker(e R_{\bbZ}(n\delta)e\to e C_{\bbZ}(n\delta)e)$, we deduce that $J^{\mathbb F_2}_2$ is reduction modulo $2$ of $J_2^\bbZ$.
So, in order to show that $J^{\mathbb F_2}_2$ is generated by elements of degrees greater than $2$, it is enough to show the same for $J^{\bbZ}_2$.
At the same time we know that $J^{\bbZ}_2=\Ker\phi^{\bbZ}_2$, and it is easy to check that $\Ker\phi^\bbZ_2$ has no elements of degree $1$ or $2$.
This completes the proof.
\end{proof}
The lemma above implies that $\Phi_2$ is not injective for $\bfk=\mathbb F_2$.
Indeed, take $P\in (\Ker\phi_2\backslash J_2)\subset \Poll_2^{(\frakS_2)^2}$.
Then $Pe_{(2)}$ is a non-zero element in $e C(2\delta)e$ that lies in the kernel of $\Phi_2$.

\subsection{Positive characteristic}
We have seen that the map $\Phi^\bfk_n\colon e C_\bfk(n\delta)e\to \Sc^\bfk_n$ is an isomorphism when $\bfk$ is a field of characteristic zero.
However, in general this map is neither injective nor surjective over a field of positive characteristic, as we have seen in \cref{sec:no-inj-surj}.
This behavior is explained by non-surjectivity of $\Phi^{\bbZ}_n$, because $\Phi_n^\bbk$ is obtained from $\Phi_n^\bbZ$ by base change $\bbk\otimes_\bbZ -$.

Let $\widetilde{\Sc}^{\bbZ}_n:=\Im\Phi_n^{\bbZ}\subset \Sc^\bbZ_n$.
\cref{prop:Phininj} implies that $\widetilde{\Sc}^{\bbZ}_n$ is a sublattice of full rank.
Now, for any field $\bfk$ define $\widetilde{\Sc}^{\bfk}_n:=\bfk\otimes_\bbZ \widetilde{\Sc}^{\bbZ}_n$.
The following theorem follows immediately from \cref{prop:Phininj}.
\begin{thm}\label{thm:cusp-computed}
Let $\bbk$ be a field.
We have an isomorphism of algebras $e C_\bbk(n\delta)e\simeq \widetilde{\Sc}^\bfk_n$.
\end{thm}

When $\ona{char}\bfk=0$, we have $\widetilde{\Sc}^{\bfk}_n=\Sc^{\bfk}_n$; however, the two algebras are quite different in general.
The algebra $\widetilde{\Sc}^\bfk_n$ is rather explicit: it has nice diagrammatic description and an explicit basis.
Moreover, the action of $\widetilde{\Sc}^\bbZ_n$ on the polynomial representation $P_n^\bbZ$ preserves the $\bbZ$-submodule $\widetilde{P}^\bbZ_n:= \bigoplus_\lambda \Im\phi_\lambda$ by \cref{prop:map-Phi}.
In particular, we obtain a polynomial representation $\widetilde{\Sc}^\bbk_n\curvearrowright \widetilde{P}^\bbk_n:=\bbk\otimes_\bbZ \widetilde{P}^\bbZ_n$.
We conjecture the following:

\begin{conj}\label{conj:pol-Fp-faith}
The polynomial representation $\widetilde{P}^\bbk_n$ of $\widetilde{\Sc}^\bbk_n$ is faithful.
\end{conj}

In positive characteristic neither $\widetilde{\Sc}^\bbk_n$ nor $\widetilde{P}^\bbk_n$ has a concise geometric description, so the argument from the proof of \cref{SchurPolRepFaith} does not apply here.
It would be interesting to realize $\widetilde{\Sc}^\bfk_n$ as homology of some variety. 

\begin{exe}
Let us illustrate how properties of $\widetilde{\Sc}^{\bb F_p}_n$ can change for different $p$.
Let $n=2$, and consider $2c_1c_2\in \Im\phi^\bbZ_2 \subset \widetilde{P}^\bbZ_2$.
It defines a non-zero element in the reduction $\widetilde{P}^{\bb F_2}_2$ by previous considerations.
Recall that we have the split operator $\rmS = \rmS_2^{(1,1)}\in \widetilde{\Sc}^{\bbZ}_2$; we denote its reduction modulo $2$ by the same letter. 
Note that $\rmS(2c_1c_2) = 0\in \widetilde{P}^{\bb F_2}_2$, because tautologically $c_1c_2\in \Im\phi_{(1,1)}^\bbZ = \bf{P}_2^\bbZ$.
This implies that $\widetilde{P}^{\bb F_2}_2$ admits a non-trivial submodule supported completely on the thick string.

On the other hand, it is an easy exercise to verify that $\Im\phi_2^\bbZ$ together with $c_1c_2$ generate the whole $\bf P_2^{\frakS_2}$ as a ring.
Therefore $\widetilde{\Sc}^{\bb F_p}_2=\Sc^{\bb F_p}_2$, $\widetilde{P}^{\bb F_p}_2 = P^{\bb F_p}_2$ for $p>2$.
However, the split operator $\rmS$ now acts on $\bf P_2^{\frakS_2}\subset P^{\bb F_p}_2$ by embedding it into $\bf P_2$, so that no submodule of $P^{\bb F_p}_2$ can be supported solely on the thick string.
\end{exe}
\appendix
\section{Several parity questions}\label{app:parity}

\subsection{Parity of quiver flag varieties, type $A_1^{(1)}$}
As before, let $\Gamma = 1\rightrightarrows 0$ be the Kronecker quiver.
Pick a representation $M = (U\rightrightarrows V)\in \ona{Rep}\Gamma$ with dimension vector $\bf{v}$.
For any increasing sequence of dimension vectors $\underline{\bf{v}}=(\bf{v}_1<\ldots<\bf{v}_k=\bf{v})$, consider the \textit{quiver flag variety}
\[
\bfF_{\underline{\bf{v}}}(M) = \left\{ M_1\subset\ldots M_k=M : \dim M_i=\bf{v}_i \right\}.
\]
The goal of this section is to prove the following theorem:
\begin{thm}
\label{thm:parity-fibres}
The flag variety $\bfF_{\underline{\bf{v}}}(M)$ has no odd cohomology groups.
\end{thm}

Before giving the proof, we will need some preparations.
Let us say that \cref{thm:parity-fibres} holds for a representation $M$ if we have $H^{\rm odd}(\bfF_{\underline{\bf{v}}}(M),\bbZ)=0$ for any $\underline{\bf{v}}$.

Recall (e.g., see~\cite{Sch_LHA2012}) that isomorphism classes of indecomposable representations of $\Gamma$ can be classified into three distinct families:

\begin{center}
\begin{tabular}{c|c|c}
	preprojective $P_n$, $n\geqslant 0$ & preinjective $I_n$, $n\geqslant 0$ & regular $R_n^{(\lambda,\mu)}$, $n>0$, $\lambda,\mu\in\bbC$\\\hline
	$\begin{tikzcd}[column sep = large]
		\bbC^n\ar[r,shift left,"\begin{pmatrix}\id_{\bbC^n}\\0\end{pmatrix}"]\ar[r,shift right,"\begin{pmatrix}0\\ \id_{\bbC^n}\end{pmatrix}"'] & \bbC^{n+1}
	\end{tikzcd}$ &
	$\begin{tikzcd}[column sep = large]
		\bbC^{n+1}\ar[r,shift left,"(\id_{\bbC^n}\text{ }0)"]\ar[r,shift right,"(0\text{ }\id_{\bbC^n})"'] & \bbC^n
	\end{tikzcd}$ & 
	\begin{tabular}{c}
	$\begin{tikzcd}[column sep = huge]
		\bbC^{n}\ar[r,shift left,"\lambda\id_{\bbC^n}+L_n"]\ar[r,shift right,"\mu\id_{\bbC^n}+L_n"'] & \bbC^n,\quad\lambda\neq \mu
	\end{tikzcd}$\\
	$\begin{tikzcd}[column sep = huge]
		\bbC^{n}\ar[r,shift left,"\lambda\id_{\bbC^n}"]\ar[r,shift right,"\lambda\id_{\bbC^n}+L_n"'] & \bbC^n
	\end{tikzcd}$
	\end{tabular}
\end{tabular}
\end{center}
Here, $L_n$ denotes a nilpotent Jordan block of rank $n$.

Furthermore, for any $n, m\geqslant 0$ we have vanishing of $\Ext$-groups:
\[
\Ext^1(P_n,P_{n+m})=\Ext^1(P_n,R_n^{(\lambda,\mu)})=\Ext^1(P_n,I_m)=\Ext^1(R_n^{(\lambda,\mu)},I_{n})=\Ext^1(I_{n+m},I_n)=0.
\]
Two representation of the form $R_n^{(\lambda,\mu)}$ are isomorphic if and only if they have the same $n$ and the same ratio $(\lambda:\mu)\in \bbP^1$.
We also have $\Ext^1(R_n^{(\lambda,\mu)},R_m^{(\lambda',\mu')})=0$ for $(\lambda:\mu)\ne (\lambda':\mu')$.

The following lemma is an immediate consequence of \cite[Proposition 2.17]{Mak_FVQG2019}.
\begin{lm}\label{lem:RuslanNoExt}
Let $M_1$ and $M_2$ be two representations of $\Gamma$ such that $\Ext^1(M_1,M_2)=0$, and such that \cref{thm:parity-fibres} holds for $M_1$ and $M_2$.
Then \cref{thm:parity-fibres} also holds for $M_1\oplus M_2$.
\end{lm}

\begin{lm}[{\cite[Theorem 3.4]{Mak_FVQG2019}}]
\label{lem:regular}
Let $M$ be a representation of $\Gamma$ satisfying $\Ext^1(M,M)=0$.
Then \cref{thm:parity-fibres} holds for $M$.
\end{lm}

\begin{rmq}
\label{rem:half-regular}
The assumption $\Ext^1(M,M)=0$ in~\cite[Theorem 3.4]{Mak_FVQG2019} is only required in order to construct a certain vector bundle on $\bfF_{\underline{\bf{v}}}(M)\times \bfF_{\underline{\bf{v}}}(M)$ with a distinguished section, which vanishes exactly on the diagonal.
This assumption can be slightly relaxed.
Namely, it may happen that $M$ has $\Ext^1(M,M)\ne 0$, but there exists another representation $M'$ such that 
\begin{itemize}
\item there exists an isomorphism of $I$-graded vector spaces $f : M\to M'$,
\item for each $\underline{\bf{v}}$, the isomorphism $\phi$ induces as isomorphism of varieties $\bfF_{\underline{\bf{v}}}(M)\simeq\bfF_{\underline{\bf{v}}}(M')$,
\item $\Ext^1(M,M') = 0$.
\end{itemize}
In this case~\cite[Theorem 3.4]{Mak_FVQG2019} is still applicable, and \cref{thm:parity-fibres} holds for $M$.
An example of such situation is $M=R_n^{(\lambda,\mu)}$ and $M'=R_n^{(\lambda',\mu')}$ with $(\lambda:\mu)\ne (\lambda':\mu')$.
More generally, we can take $M=\bigoplus_{r=1}^k R_{a_r}^{(\lambda,\mu)}$ and $M'=\bigoplus_{r=1}^k R_{a_r}^{(\lambda',\mu')}$  with $(\lambda:\mu)\ne (\lambda':\mu')$ for some positive integers $a_1,\ldots,a_k$.
\end{rmq}

\begin{proof}[Proof of \cref{thm:parity-fibres}]
Any representation $M$ can be decomposed as $M=P\oplus R\oplus I$ such that $P$ is preprojective, $R$ is regular and $I$ is preinjective.
Since $\Ext^1(P,R)=\Ext^1(R,I)=\Ext^1(P,I)=0$, it is enough to prove the statement separately for $P$, $R$ and $I$ by \cref{lem:RuslanNoExt}. 

First, let us show that the theorem holds for a preprojective representation $P$.
We know that $P$ can be decomposed as $P=\bigoplus_{r=1}^k P_{a_r}$ for some $a_1,a_2\ldots,a_k\in \bbZ_{\geqslant 0}$.
We have $\Ext^1(P_a,P_b)=0$ for $a\leqslant b$.
Applying \cref{lem:regular} and \cref{lem:RuslanNoExt}, we see that \cref{thm:parity-fibres} holds for each $P_{a_r}$ and then for $P$ as well.
The same argument also proves the statement for preinjective representations.

Now, let $R$ be a regular representation.
We can decompose $R$ as $R=\bigoplus_{r=1}^k R^{(\lambda_r:\mu_r)}$, where $(\lambda_r:\mu_r)\in \bbP^1$ are different for different $r$'s, and $R^{(\lambda_r:\mu_r)}$ is isomorphic to a direct sum of representations of the form $R_n^{(\lambda_r,\mu_r)}$.
Since $\Ext^1(R^{(\lambda_i:\mu_i)},R^{(\lambda_j:\mu_j)}) = 0$ for $i\ne j$, it suffices to prove the statement for each $R^{(\lambda_r:\mu_r)}$.
We conclude by applying \cref{rem:half-regular}.
\end{proof}

\subsection{Parity sheaves}
The theory of parity sheaves has been developed in~\cite{JMW_PS2014}.
It takes as an input a complex algebraic variety $Y$ with an action of a complex algebraic group $G$, such that $Y$ has a $G$-invariant stratification $Y=\coprod_{\lambda}Y_\lambda$ satisfying some parity conditions~\cite[(2.1),(2.2)]{JMW_PS2014}.
For each stratum $Y_\lambda$ and a local system $\mathfrak L$ on $Y_\lambda$, it produces a certain indecomposable complex $\mathcal E(\lambda,\mathfrak L)$ supported on $\overline{Y_\lambda}$, which satisfies a list of properties.
(This complex is not well-defined in general, but it is unique if it exists.)

The case $Y=E_\alpha$, $G=G_\alpha$ was studied in~\cite{Mak_CBKA2015} for a Dynkin quiver.
In this situation, we have a finite stratification of $E_\alpha$ by $G_\alpha$-orbits, and each stratum admits only trivial $G_\alpha$-equivariant local systems.
The existence of parity sheaf for each stratum is then proved.
However, \cite{Mak_CBKA2015} did not prove that in positive characteristics the Lusztig sheaf $\mathcal L_\alpha = (\pi_\alpha)_*\underline{\bbk}$ is a direct sum of shifts of parity sheaves.
This was done later in~\cite{McN_RTGE2017,Mak_FVQG2019}.
The key point was to prove that the fibers of maps $\widetilde{\bfF}_{\ui}\to E_\alpha$ have no odd cohomology groups.
Note that \cref{thm:parity-fibres} proves an analogous statement about the fibers for the Kronecker quiver.

Nevertheless, the proposition below shows that there is no satisfactory theory of parity sheaves for the Kronecker quiver, already for dimension vector $\alpha=2\delta$.
So, despite \cref{thm:parity-fibres}, we cannot deduce in this case that the sheaf $\mathcal L_{\alpha}$ is a direct sum of shifts of parity sheaves.

\begin{prop}
Let $\Gamma$ be the Kronecker quiver. 
There is no algebraic stratification (in the sense of \cite[Definition 3.2.23]{CG_RTCG2010}) of $E_{2\delta}$ into smooth connected locally closed subsets such that 
\begin{itemize}
\item each stratum is $G_{2\delta}$-invariant,
\item each stratum satisfies \cite[(2.1),(2.2)]{JMW_PS2014},
\item the subset $E^{\rm reg}_{2\delta}\subset E_{2\delta}$ is a union of strata.
\end{itemize}
\end{prop}
\begin{proof}
Suppose that such a stratification exists.
The first two assumptions are simply the assumptions in~\cite{JMW_PS2014} that allow to apply the theory of parity sheaves.
In particular, the constant sheaf $\underline\bfk$ on $E_{2\delta}$ is a parity sheaf (up to a shift).

Consider the inclusion map $\iota\colon E^{\rm reg}_{2\delta}\to E_{2\delta}$. 
The third assumption together with argument in~\cite[Corollary~4.2]{McN_RTGE2017} show that the map $\Ext_{G_{2\delta}}^*(\underline\bfk,\underline\bfk)\to \mathcal \Ext_{G_{2\delta}}^*(\iota^*\underline\bfk,\iota^*\underline\bfk)$ must be surjective. 
Recall that we the following commutative diagram, where the horizontal maps are isomorphisms:
\[
\begin{tikzcd}
	\bfk[u_1,u_2,v_1,v_2]^{\frakS_2^2}\ar[d,"\phi_2"]\ar[r] & H^{G_{2\delta}}_*(E_{2\delta})\ar[d]\ar[r] & \Ext^*_{G_{2\delta}}(\underline\bfk,\underline\bfk)\ar[d] \\
	(\bfk[x_1,x_2,c_1,c_2]/(c_1^2,c_2^2))^{\frakS_2}\ar[r] & H^{G_{2\delta}}_*(E^{\rm reg}_{2\delta})\ar[r] & \Ext^*_{G_{2\delta}}(\iota^*\underline\bfk,\iota^*\underline\bfk)
\end{tikzcd}
\]
However, we have seen in \cref{sec:no-inj-surj} that the map $\phi_2$ is not surjective for $\bfk=\mathbb F_2$.
\end{proof}

\bibliography{zot-bib}{}

\begin{thebibliography}{McN17b}

\bibitem[AHR15]{AHR_GSSC2015}
Pramod~N. Achar, Anthony Henderson, and Simon Riche.
\newblock Geometric {{Satake}}, {{Springer}} correspondence, and small
  representations {{II}}.
\newblock {\em Represent. Theory}, 19:94--166, 2015.

\bibitem[Ati57]{Ati_VBEC1957}
Michael Atiyah.
\newblock Vector {{Bundles Over}} an {{Elliptic Curve}}.
\newblock {\em Proceedings of the London Mathematical Society},
  s3-7(1):414--452, 1957.

\bibitem[Bei78]{Bei_CSPP1978}
Alexander~A. Beilinson.
\newblock Coherent sheaves on {{Pn}} and problems of linear algebra.
\newblock {\em Funct. Anal. Its Appl.}, 12(3):214--216, July 1978.

\bibitem[{Bia}73]{Bia_TAAG1973a}
Andrzej {Bia{\l}ynicki-Birula}.
\newblock Some {{Theorems}} on {{Actions}} of {{Algebraic Groups}}.
\newblock {\em The Annals of Mathematics}, 98(3):480, November 1973.

\bibitem[BK99]{BK_MRA1999a}
Michael~C. Butler and Alastair~D. King.
\newblock Minimal {{Resolutions}} of {{Algebras}}.
\newblock {\em J. Algebra}, 212(1):323--362, 1999.

\bibitem[BN15]{BN_EST2015}
David {Ben-Zvi} and David Nadler.
\newblock Elliptic {{Springer Theory}}.
\newblock {\em Compositio Math.}, 151(8):1568--1584, August 2015.

\bibitem[BS12]{BS_HAEC2012}
Igor Burban and Olivier Schiffmann.
\newblock On the {{Hall}} algebra of an elliptic curve, {{I}}.
\newblock {\em Duke Mathematical Journal}, 161(7):1171--1231, 2012.

\bibitem[CG10]{CG_RTCG2010}
Neil Chriss and Victor Ginzburg.
\newblock {\em Representation {{Theory}} and {{Complex Geometry}}}.
\newblock {Birkh\"auser Basel}, 2010.

\bibitem[FM81]{FM_CFSS1981}
William Fulton and Robert MacPherson.
\newblock Categorical framework for the study of singular spaces.
\newblock {\em Mem. Amer. Math. Soc.}, 31(243):vi+165, 1981.

\bibitem[FMW98]{FMW_PBEC1998}
Robert Friedman, John~W. Morgan, and Edward Witten.
\newblock Principal {{G}}-bundles over elliptic curves.
\newblock {\em Mathematical Research Letters}, 5(1):97--118, 1998.

\bibitem[FR19]{FR_CHAK2019}
Hans Franzen and Markus Reineke.
\newblock On the cohomological {{Hall}} algebra of the {{Kronecker}} quiver.
\newblock {\em arXiv:1904.09224 [math]}, April 2019.

\bibitem[Ful98]{Ful_IT1998}
William Fulton.
\newblock {\em Intersection Theory}, volume~2 of {\em Ergebnisse Der
  {{Mathematik}} Und Ihrer {{Grenzgebiete}}. 3. {{Folge}}. {{A Series}} of
  {{Modern Surveys}} in {{Mathematics}} [{{Results}} in {{Mathematics}} and
  {{Related Areas}}. 3rd {{Series}}. {{A Series}} of {{Modern Surveys}} in
  {{Mathematics}}]}.
\newblock {Springer-Verlag, Berlin}, second edition, 1998.

\bibitem[GHS11]{GHS_MMCH2011}
Oscar {Garc{\'i}a-Prada}, Jochen Heinloth, and Alexander Schmitt.
\newblock On the motives of moduli of chains and {{Higgs}} bundles.
\newblock {\em arXiv:1104.5558 [math]}, April 2011.

\bibitem[Gre95]{Gre_HAHA1995}
James~A. Green.
\newblock Hall algebras, hereditary algebras and quantum groups.
\newblock {\em Invent Math}, 120(1):361--377, December 1995.

\bibitem[Hei12]{Hei_CMSC2012}
Jochen Heinloth.
\newblock Cohomology of the moduli stack of coherent sheaves on a curve.
\newblock In Carel Faber, Gavril Farkas, and Robin {de Jong}, editors, {\em
  Geometry and {{Arithmetic}}}, pages 165--171. {European Mathematical Society
  Publishing House}, {Zuerich, Switzerland}, October 2012.

\bibitem[HS09]{HS_TAEC2009}
Tara Holm and Reyer Sjamaar.
\newblock Torsion and abelianization in equivariant cohomology.
\newblock {\em Transform. Groups}, 13(3-4):585--615, June 2009.

\bibitem[Hsi75]{Hsi_CTTT1975}
W.~Y. Hsiang.
\newblock {\em Cohomology {{Theory}} of {{Topological Transformation Groups}}}.
\newblock Ergebnisse Der {{Mathematik}} Und Ihrer {{Grenzgebiete}}. 2.
  {{Folge}}. {Springer-Verlag}, {Berlin Heidelberg}, 1975.

\bibitem[JMW14]{JMW_PS2014}
Daniel Juteau, Carl Mautner, and Geordie Williamson.
\newblock Parity sheaves.
\newblock {\em J. Amer. Math. Soc.}, 27(4):1169--1212, 2014.

\bibitem[KL09]{KL_DACQ2009}
Mikhail Khovanov and Aaron~D. Lauda.
\newblock A diagrammatic approach to categorification of quantum groups {{I}}.
\newblock {\em Represent. Theory}, 13:309--347, 2009.

\bibitem[Kle15]{Kle_AHWC2015}
Alexander~S. Kleshchev.
\newblock Affine highest weight categories and affine quasihereditary algebras.
\newblock {\em Proceedings of the London Mathematical Society},
  110(4):841--882, April 2015.

\bibitem[KM17a]{KM_ISD2017}
Alexander Kleshchev and Robert Muth.
\newblock Imaginary {{Schur}}-{{Weyl}} duality.
\newblock {\em Mem. Amer. Math. Soc.}, 245(1157):xvii+83, 2017.

\bibitem[KM17b]{KM_SKAA2017}
Alexander Kleshchev and Robert Muth.
\newblock Stratifying {{KLR}} algebras of affine {{ADE}} types.
\newblock {\em Journal of Algebra}, 475:133--170, April 2017.

\bibitem[KM19]{KM_AZAI2019}
Alexander Kleshchev and Robert Muth.
\newblock Affine zigzag algebras and imaginary strata for {{KLR}} algebras.
\newblock {\em Trans. Amer. Math. Soc.}, 371(7):4535--4583, 2019.

\bibitem[LP97]{potierlectures}
Joseph Le~Potier.
\newblock {\em Lectures on Vector Bundles}, volume~54 of {\em Cambridge
  {{Studies}} in {{Advanced Mathematics}}}.
\newblock {Cambridge University Press, Cambridge}, 1997.

\bibitem[Mak15]{Mak_CBKA2015}
Ruslan Maksimau.
\newblock Canonical basis, {{KLR}} algebras and parity sheaves.
\newblock {\em Journal of Algebra}, 422:563--610, January 2015.

\bibitem[Mak19]{Mak_FVQG2019}
Ruslan Maksimau.
\newblock Flag versions of quiver {{Grassmannians}} for {{Dynkin}} quivers have
  no odd cohomology over {{Z}}.
\newblock {\em arXiv:1909.04907 [math]}, September 2019.

\bibitem[McN17a]{McN_RTGE2017}
Peter~J. McNamara.
\newblock Representation theory of geometric extension algebras.
\newblock {\em arXiv:1701.07949 [math]}, 2017.

\bibitem[McN17b]{McN_RKAI2017}
Peter~J. McNamara.
\newblock Representations of {{Khovanov}}-{{Lauda}}-{{Rouquier}} algebras
  {{III}}: Symmetric affine type.
\newblock {\em Math. Z.}, 287(1-2):243--286, 2017.

\bibitem[Min20]{Min_CHAH2020}
Alexandre Minets.
\newblock Cohomological {{Hall}} algebras for {{Higgs}} torsion sheaves, moduli
  of triples and sheaves on surfaces.
\newblock {\em Selecta Math. (N.S.)}, 26(2):Paper No. 30, 67, 2020.

\bibitem[MS17]{MS_HSAT2017}
Hugh Morton and Peter Samuelson.
\newblock The {{HOMFLYPT}} skein algebra of the torus and the elliptic {{Hall}}
  algebra.
\newblock {\em Duke Mathematical Journal}, 166(5):801--854, April 2017.

\bibitem[Neg14]{Neg_SAR2014}
Andrei Negu{\c t}.
\newblock The shuffle algebra revisited.
\newblock {\em International Mathematics Research Notices},
  2014(22):6242--6275, 2014.

\bibitem[Prz19]{Prz_QSAC2019}
Tomasz Przezdziecki.
\newblock Quiver {{Schur}} algebras and cohomological {{Hall}} algebras.
\newblock {\em arXiv:1907.03679 [math]}, July 2019.

\bibitem[Rin90]{Rin_HAQG1990}
Claus~Michael Ringel.
\newblock Hall algebras and quantum groups.
\newblock {\em Invent Math}, 101(1):583--591, December 1990.

\bibitem[Rou08]{Rou_2A2008}
Raphael Rouquier.
\newblock 2-{{Kac}}-{{Moody}} algebras.
\newblock {\em arXiv:0812.5023 [math]}, December 2008.

\bibitem[Rou12]{Rou_QHA22012}
Raphael Rouquier.
\newblock Quiver {{Hecke}} algebras and 2-{{Lie}} algebras.
\newblock {\em Algebra Colloq.}, 19(2):359--410, 2012.

\bibitem[Sav18]{Sav_AWPA2018}
Alistair Savage.
\newblock Affine wreath product algebras.
\newblock {\em International Mathematics Research Notices}, May 2018.

\bibitem[Sch12]{Sch_LHA2012}
Olivier Schiffmann.
\newblock Lectures on {{Hall}} algebras.
\newblock In {\em Geometric Methods in Representation Theory.
  \{\vphantom\}{{II}}\vphantom\{\}}, volume~24 of {\em S\'emin. {{Congr}}.},
  pages 1--141. {Soc. Math. France, Paris}, 2012.

\bibitem[SV13]{SV_EHAE2013}
Olivier Schiffmann and Eric Vasserot.
\newblock The elliptic {{Hall}} algebra and the equivariant {{K}}-theory of the
  {{Hilbert}} scheme of {{A2}}.
\newblock {\em Duke Mathematical Journal}, 162(2):279--366, 2013.

\bibitem[SVV19]{SVV_CCQL2019}
Peng Shan, Michela Varagnolo, and Eric Vasserot.
\newblock Coherent categorification of quantum loop algebras : The {{SL}}(2)
  case.
\newblock {\em arXiv:1912.03325 [math]}, December 2019.

\bibitem[SW14]{SW_QSAF2014}
Catharina Stroppel and Ben Webster.
\newblock Quiver {{Schur}} algebras and q-{{Fock}} space.
\newblock {\em arXiv:1110.1115 [math]}, July 2014.

\bibitem[VV11]{VV_CBK2011}
Michela Varagnolo and Eric Vasserot.
\newblock Canonical bases and {{KLR}}-algebras.
\newblock {\em J. Reine Angew. Math.}, 659:67--100, 2011.

\bibitem[Wey39]{Wey_CGTI1939}
Hermann Weyl.
\newblock {\em The {{Classical Groups}}: {{Their Invariants}} and
  {{Representations}}}.
\newblock {Princeton University Press}, 1939.

\end{thebibliography}
\bibliographystyle{alpha}

\end{document}